\documentclass[a4paper,11pt,reqno]{amsart}

\usepackage[utf8x]{inputenc}
\usepackage{hyperref}
\usepackage{xcolor}
\hypersetup{
    colorlinks,
    linkcolor={red!50!black},
    citecolor={blue!50!black},
    urlcolor={blue!80!black}
}

\usepackage{microtype}
\usepackage{MnSymbol}
\usepackage{textcomp}

\usepackage{BOONDOX-cal}

\makeatletter
\newcommand{\eqnum}{\refstepcounter{equation}\textup{\tagform@{\theequation}}}
\makeatother

\makeatletter \@addtoreset{equation}{subsection} \makeatother
\renewcommand{\theequation}{\thesubsection.\arabic{equation}}

\newtheorem{thm}[equation]{Theorem}
\newtheorem*{thm*}{Theorem}
\newtheorem{lem}[equation]{Lemma}
\newtheorem{cor}[equation]{Corollary}
\newtheorem{prop}[equation]{Proposition}
\newtheorem{defthm}[equation]{Definition/Theorem}
\newtheorem{defprop}[equation]{Definition/Proposition}

\theoremstyle{definition}
\newtheorem{defn}[equation]{Definition}
\newtheorem{fakedefn}[equation]{``Definition''}
\newtheorem{rem}[equation]{Remark}
\newtheorem{remind}[equation]{Reminder}
\newtheorem{exam}[equation]{Example}
\newtheorem{exer}[equation]{Exercise}
\newtheorem{constr}[equation]{Construction}

\newtheorem{notat}[equation]{Notation}
\newtheorem{warn}[equation]{Warning}
\newtheorem{quest}[equation]{Question}
\newtheorem*{exam*}{Example}

\newcommand\arXiv[1]{\href{http://arxiv.org/abs/#1}{arXiv:#1}}
\newcommand\mathAG[1]{\href{http://arxiv.org/abs/math/#1}{math.AG/#1}}

\newcommand{\keroref}[1]{\href{https://kerodon.net/tag/#1}{#1}}
\newcommand{\kerocite}[1]{\cite[Tag~\keroref{#1}]{Kerodon}}

\newcommand{\spref}[1]{\href{https://stacks.math.columbia.edu/tag/#1}{#1}}
\newcommand{\spcite}[1]{\cite[Tag~\spref{#1}]{Stacks}}

\usepackage{xparse}

\usepackage{xspace}

\newcommand{\changelocaltocdepth}[1]{%
  \addtocontents{toc}{\protect\setcounter{tocdepth}{#1}}%
  \setcounter{tocdepth}{#1}}

\newcommand{\nc}{\newcommand}
\nc{\renc}{\renewcommand}
\nc{\ssec}{\subsection}
\nc{\sssec}{\subsubsection}
\nc{\on}{\operatorname}
\nc{\term}[1]{#1\xspace}

\usepackage{tikz}
\usetikzlibrary{matrix}
\usepackage{tikz-cd}

\tikzset{
  commutative diagrams/.cd,
  arrow style=tikz,
  diagrams={>=latex}}
\makeatletter
\tikzset{
  column sep/.code=\def\pgfmatrixcolumnsep{\pgf@matrix@xscale*(#1)},
  row sep/.code   =\def\pgfmatrixrowsep{\pgf@matrix@yscale*(#1)},
  matrix xscale/.code=%
    \pgfmathsetmacro\pgf@matrix@xscale{\pgf@matrix@xscale*(#1)},
  matrix yscale/.code=%
    \pgfmathsetmacro\pgf@matrix@yscale{\pgf@matrix@yscale*(#1)},
  matrix scale/.style={/tikz/matrix xscale={#1},/tikz/matrix yscale={#1}}}
\def\pgf@matrix@xscale{1}
\def\pgf@matrix@yscale{1}
\makeatother

\usepackage[inline]{enumitem}
\setlist[enumerate,1]{label={(\alph*)},itemsep=\parskip}

\newlist{thmlist}{enumerate}{1}
\setlist[thmlist,1]{
  label={\em(\roman*)}, ref={(\roman*)},
  itemsep=0.5em,
  align=right,widest=vi)}

\newlist{thmlistbis}{enumerate}{1}
\setlist[thmlistbis,1]{
  label={\em(\roman*~\textit{bis})},
  ref={(\roman*}~\textit{bis}\upshape{)},
  itemsep=0.5em,
  align=right, widest=vi)}

\newlist{defnlist}{enumerate}{2}
\setlist[defnlist,1]{
  label={(\roman*)}, ref={(\roman*)},
  itemsep=0.5em,
  align=right, widest=vi)}

\setlist[defnlist,2]{
  label={(\alph*)}, ref={(\alph*)},
  itemsep=0.75em,
  labelsep=0em,labelindent=0em,leftmargin=*,align=left,widest=vi),
  topsep=0.75em}

\newlist{inlinelist}{enumerate*}{1}
\setlist[inlinelist,1]{label={(\alph*)}}

\newlist{inlinedefnlist}{enumerate*}{1}
\definecolor{green}{HTML}{38550C}
\setlist[inlinedefnlist,1]{label={\color{green}(\roman*)}}

\nc{\sA}{\ensuremath{\mathcal{A}}\xspace}
\nc{\sB}{\ensuremath{\mathcal{B}}\xspace}
\nc{\sC}{\ensuremath{\mathcal{C}}\xspace}
\nc{\sD}{\ensuremath{\mathcal{D}}\xspace}
\nc{\sE}{\ensuremath{\mathcal{E}}\xspace}
\nc{\sF}{\ensuremath{\mathcal{F}}\xspace}
\nc{\sG}{\ensuremath{\mathcal{G}}\xspace}
\nc{\sH}{\ensuremath{\mathcal{H}}\xspace}
\nc{\sI}{\ensuremath{\mathcal{I}}\xspace}
\nc{\sJ}{\ensuremath{\mathcal{J}}\xspace}
\nc{\sK}{\ensuremath{\mathcal{K}}\xspace}
\nc{\sL}{\ensuremath{\mathcal{L}}\xspace}
\nc{\sM}{\ensuremath{\mathcal{M}}\xspace}
\nc{\sN}{\ensuremath{\mathcal{N}}\xspace}
\nc{\sO}{\ensuremath{\mathcal{O}}\xspace}
\nc{\sP}{\ensuremath{\mathcal{P}}\xspace}
\nc{\sQ}{\ensuremath{\mathcal{Q}}\xspace}
\nc{\sR}{\ensuremath{\mathcal{R}}\xspace}
\nc{\sS}{\ensuremath{\mathcal{S}}\xspace}
\nc{\sT}{\ensuremath{\mathcal{T}}\xspace}
\nc{\sU}{\ensuremath{\mathcal{U}}\xspace}
\nc{\sV}{\ensuremath{\mathcal{V}}\xspace}
\nc{\sW}{\ensuremath{\mathcal{W}}\xspace}
\nc{\sX}{\ensuremath{\mathcal{X}}\xspace}
\nc{\sY}{\ensuremath{\mathcal{Y}}\xspace}
\nc{\sZ}{\ensuremath{\mathcal{Z}}\xspace}

\nc{\bA}{\ensuremath{\mathbf{A}}\xspace}
\nc{\bB}{\ensuremath{\mathbf{B}}\xspace}
\nc{\bC}{\ensuremath{\mathbf{C}}\xspace}
\nc{\bD}{\ensuremath{\mathbf{D}}\xspace}
\nc{\bE}{\ensuremath{\mathbf{E}}\xspace}
\nc{\bF}{\ensuremath{\mathbf{F}}\xspace}
\nc{\bG}{\ensuremath{\mathbf{G}}\xspace}
\nc{\bH}{\ensuremath{\mathbf{H}}\xspace}
\nc{\bI}{\ensuremath{\mathbf{I}}\xspace}
\nc{\bJ}{\ensuremath{\mathbf{J}}\xspace}
\nc{\bK}{\ensuremath{\mathbf{K}}\xspace}
\nc{\bL}{\ensuremath{\mathbf{L}}\xspace}
\nc{\bM}{\ensuremath{\mathbf{M}}\xspace}
\nc{\bN}{\ensuremath{\mathbf{N}}\xspace}
\nc{\bO}{\ensuremath{\mathbf{O}}\xspace}
\nc{\bP}{\ensuremath{\mathbf{P}}\xspace}
\nc{\bQ}{\ensuremath{\mathbf{Q}}\xspace}
\nc{\bR}{\ensuremath{\mathbf{R}}\xspace}
\nc{\bS}{\ensuremath{\mathbf{S}}\xspace}
\nc{\bT}{\ensuremath{\mathbf{T}}\xspace}
\nc{\bU}{\ensuremath{\mathbf{U}}\xspace}
\nc{\bV}{\ensuremath{\mathbf{V}}\xspace}
\nc{\bW}{\ensuremath{\mathbf{W}}\xspace}
\nc{\bX}{\ensuremath{\mathbf{X}}\xspace}
\nc{\bY}{\ensuremath{\mathbf{Y}}\xspace}
\nc{\bZ}{\ensuremath{\mathbf{Z}}\xspace}

\nc{\bbA}{\ensuremath{\mathbb{A}}\xspace}
\nc{\bbB}{\ensuremath{\mathbb{B}}\xspace}
\nc{\bbC}{\ensuremath{\mathbb{C}}\xspace}
\nc{\bbD}{\ensuremath{\mathbb{D}}\xspace}
\nc{\bbE}{\ensuremath{\mathbb{E}}\xspace}
\nc{\bbF}{\ensuremath{\mathbb{F}}\xspace}
\nc{\bbG}{\ensuremath{\mathbb{G}}\xspace}
\nc{\bbH}{\ensuremath{\mathbb{H}}\xspace}
\nc{\bbI}{\ensuremath{\mathbb{I}}\xspace}
\nc{\bbJ}{\ensuremath{\mathbb{J}}\xspace}
\nc{\bbK}{\ensuremath{\mathbb{K}}\xspace}
\nc{\bbL}{\ensuremath{\mathbb{L}}\xspace}
\nc{\bbM}{\ensuremath{\mathbb{M}}\xspace}
\nc{\bbN}{\ensuremath{\mathbb{N}}\xspace}
\nc{\bbO}{\ensuremath{\mathbb{O}}\xspace}
\nc{\bbP}{\ensuremath{\mathbb{P}}\xspace}
\nc{\bbQ}{\ensuremath{\mathbb{Q}}\xspace}
\nc{\bbR}{\ensuremath{\mathbb{R}}\xspace}
\nc{\bbS}{\ensuremath{\mathbb{S}}\xspace}
\nc{\bbT}{\ensuremath{\mathbb{T}}\xspace}
\nc{\bbU}{\ensuremath{\mathbb{U}}\xspace}
\nc{\bbV}{\ensuremath{\mathbb{V}}\xspace}
\nc{\bbW}{\ensuremath{\mathbb{W}}\xspace}
\nc{\bbX}{\ensuremath{\mathbb{X}}\xspace}
\nc{\bbY}{\ensuremath{\mathbb{Y}}\xspace}
\nc{\bbZ}{\ensuremath{\mathbb{Z}}\xspace}

\nc{\mrm}[1]{\ensuremath{\mathrm{#1}}\xspace}
\nc{\mit}[1]{\ensuremath{\mathit{#1}}\xspace}
\nc{\mbf}[1]{\ensuremath{\mathbf{#1}}\xspace}
\nc{\mcal}[1]{\ensuremath{\mathcal{#1}}\xspace}
\nc{\msc}[1]{\ensuremath{\mathscr{#1}}\xspace}
\nc{\mfr}[1]{\ensuremath{\mathfrak{#1}}\xspace}

\nc{\sub}{\subseteq}
\nc{\too}{\longrightarrow}
\nc{\hook}{\hookrightarrow}
\nc{\hooklongrightarrow}{\lhook\joinrel\longrightarrow}
\nc{\hooklong}{\hooklongrightarrow}
\nc{\hooklongleftarrow}{\longleftarrow\joinrel\rhook}
\nc{\twoheadlongrightarrow}{\relbar\joinrel\twoheadrightarrow}
\nc{\longrightleftarrows}{\ \raisebox{0.3ex}{\(\mathrel{\substack{\xrightarrow{\rule{1em}{0em}} \\[-1ex] \xleftarrow{\rule{1em}{0em}}}}\)}\ }

\renc{\ge}{\geqslant}
\renc{\le}{\leqslant}

\nc{\id}{\mathrm{id}}

\DeclareMathOperator{\Hom}{\on{Hom}}
\nc{\uHom}{\underline{\smash{\Hom}}}
\DeclareMathOperator{\Maps}{\on{Maps}}
\DeclareMathOperator{\Aut}{\on{Aut}}
\DeclareMathOperator{\End}{\on{End}}
\DeclareMathOperator{\Sym}{\on{Sym}}
\nc{\uEnd}{\underline{\smash{\End}}}
\DeclareMathOperator{\codim}{\on{codim}}
\nc{\colim}{\varinjlim}
\renc{\lim}{\varprojlim}
\nc{\Cofib}{\on{Cofib}}
\nc{\Fib}{\on{Fib}}
\nc{\initial}{\varnothing}
\nc{\op}{\mathrm{op}}

\DeclareMathOperator*{\fibprod}{\times}

\renc{\setminus}{\smallsetminus}

\usepackage{mathtools}
\DeclarePairedDelimiter\abs{\lvert}{\rvert}%

\newcommand{\thmref}[1]{Theorem~\ref{#1}}

\newcommand{\secref}[1]{Sect.~\ref{#1}}

\newcommand{\lemref}[1]{Lemma~\ref{#1}}
\newcommand{\propref}[1]{Proposition~\ref{#1}}
\newcommand{\corref}[1]{Corollary~\ref{#1}}

\newcommand{\remref}[1]{Remark~\ref{#1}}
\newcommand{\defnref}[1]{Definition~\ref{#1}}

\renewcommand{\eqref}[1]{(\ref{#1})}
\newcommand{\constrref}[1]{Construction~\ref{#1}}

\newcommand{\exerref}[1]{Exercise~\ref{#1}}
\newcommand{\examref}[1]{Example~\ref{#1}}

\newcommand{\itemref}[1]{\ref{#1}}

\nc{\bDelta}{\mathbf{\Delta}}
\nc{\Set}{\mrm{Set}}
\nc{\SSet}{\mrm{SSet}}
\nc{\Fun}{\on{Fun}}
\nc{\Cat}{\mrm{Cat}}
\nc{\Catoo}{\mrm{Cat}_\infty}
\nc{\Grpdoo}{\mrm{Grpd}_\infty}
\nc{\Sing}{\on{Sing}}
\renc{\top}{\mrm{top}}
\nc{\h}{\on{h}}
\nc{\Kan}{\mrm{Kan}}
\nc{\Fin}{\mrm{Fin}}
\nc{\Ab}{\mrm{Ab}}
\nc{\Mat}{\on{Mat}}
\nc{\Mod}{\mrm{Mod}}
\nc{\rightrightrightarrows}
  {\mathrel{\substack{\textstyle\rightarrow\\[-0.6ex]
   \textstyle\rightarrow \\[-0.6ex]
   \textstyle\rightarrow}}}
\nc{\rightrightrightrightarrows}
  {\mathrel{\substack{\textstyle\rightarrow\\[-0.6ex]
   \textstyle\rightarrow \\[-0.6ex]
   \textstyle\rightarrow \\[-0.6ex]
   \textstyle\rightarrow}}}
\nc{\pt}{\mrm{pt}}
\nc{\Anima}{\mrm{Ani}}
\nc{\Anim}{\on{Anim}}
\nc{\D}{\on{\mrm{D}}}
\nc{\Dqc}{\mrm{D}_{\mrm{qc}}}
\nc{\Dperf}{\mrm{D}_{\mrm{perf}}}
\nc{\Dcoh}{\mrm{D}_{\mrm{coh}}}
\nc{\Dpscoh}{\mrm{D}_{\mrm{pscoh}}}
\renc{\L}{\mrm{\bL}}
\nc{\R}{\mrm{\bR}}
\nc{\otimesL}{\otimes^\bL}
\nc{\Spt}{\mrm{Spt}}
\nc{\Einfty}{\sE_\infty}
\nc{\A}{\bA}
\renc{\P}{\bP}
\nc{\CRing}{\mathrm{CRing}}
\nc{\CAlg}{\mathrm{CAlg}}
\nc{\Spec}{\on{Spec}}
\nc{\pr}{\mrm{pr}}
\nc{\Grpd}{\mrm{Grpd}}
\nc{\QCoh}{\on{QCoh}}
\nc{\Coh}{\on{Coh}}
\nc{\Bun}{\on{Bun}}
\nc{\GL}{\on{GL}}
\nc{\BGL}{\mrm{BGL}}
\nc{\et}{\mathrm{\acute et}}
\nc{\Sch}{\mrm{Sch}}
\nc{\Aff}{\mrm{Aff}}
\nc{\dash}{{\textnormal{-}}}
\nc{\Shv}{\on{Shv}}
\nc{\QCohSch}{\QCoh_\Sch}
\nc{\Cart}{\on{Cart}}
\nc{\Tot}{\on{Tot}}
\nc{\Stk}{\mrm{Stk}}
\nc{\AlgStk}{\mrm{AlgStk}}
\nc{\free}{\mrm{free}}
\nc{\CAlgMod}{\mrm{CAlgMod}}
\nc{\ACRing}{\mrm{ACRing}}
\nc{\ACAlg}{\mrm{ACAlg}}
\nc{\AMod}{\mrm{AMod}}
\nc{\Poly}{\mathrm{Poly}}
\nc{\Z}{\bZ}
\nc{\Q}{\bQ}
\nc{\Der}{\on{Der}}
\nc{\LSym}{\on{LSym}}
\nc{\St}{\on{St}}
\nc{\DSch}{\mrm{DSch}}
\nc{\DStk}{\mrm{DStk}}
\nc{\cl}{\mrm{cl}}
\nc{\fibprodR}{\fibprod^\bR}
\nc{\Perf}{{\mrm{Perf}}}
\nc{\Vect}{{\mrm{Vect}}}
\nc{\MPerf}{\sM_{\mrm{Perf}}}
\nc{\univ}{{\mrm{univ}}}
\nc{\uMaps}{\underline{\smash{\on{Maps}}}}
\nc{\ev}{\mrm{ev}}

\nc{\vb}[1]{\langle{#1}\rangle}

\renc{\H}{\on{H}}
\nc{\BM}{\mrm{BM}}
\nc{\C}{\on{C}}
\nc{\Chom}{\mrm{C}_\bullet}
\nc{\Ccoh}{\mrm{C}^\bullet}
\nc{\Ccohc}{\mrm{C}_{\mrm{c}}^\bullet}
\nc{\CBM}{\mrm{C}^{\BM}_\bullet}
\renc{\sp}{\on{sp}}
\nc{\tr}{\on{tr}}
\nc{\V}{\on{\bV}}

\nc{\scr}{\term{derived commutative ring}}
\nc{\scrs}{\term{derived commutative rings}}

\nc{\inftyCat}{\term{$\infty$-category}}
\nc{\inftyCats}{\term{$\infty$-categories}}

\nc{\inftyGrpd}{\term{$\infty$-groupoid}}
\nc{\inftyGrpds}{\term{$\infty$-groupoids}}

\nc{\da}{\term{derived algebraic}}

\usetikzlibrary{decorations.markings, arrows.meta}
\tikzset{
    marrow/.style={decoration={markings,mark=at position 0.6 with {\arrow{Latex}}}, postaction=decorate}
}

\title{Lectures on algebraic stacks\vspace{-2mm}}
\author{Adeel A. Khan\vspace{-1mm}}
\date{2023-10-19}

\makeatletter
\def\l@subsection{\@tocline{2}{0pt}{4pc}{6pc}{}}
\makeatother

\begin{document}


\maketitle

\renewcommand\contentsname{\vspace{-1cm}}
\tableofcontents

\setlength{\parindent}{0em}
\parskip 0.75em

\thispagestyle{empty}


\changelocaltocdepth{2}


\setcounter{section}{-1}
\section{Overview}
\label{sec:overview}

\begin{fakedefn}
  Given a class of geometric objects, say ``gadgets'', a \emph{moduli space} for gadgets is a space whose points correspond to gadgets, modulo some notion of equivalence between gadgets:
  \[
    \text{moduli space of gadgets}
    = \{ \text{gadgets} \}/\text{equivalence}.
  \]
\end{fakedefn}

The idea of moduli theory is to transform questions about gadgets into questions about the moduli space, which we may then try to tackle, via topological, geometric, or cohomological methods.

\begin{exam}
  The complex projective $n$-space $\P^n(\bC)$ is the moduli space of lines in $\bC^{n+1}$ which pass through the origin.
  Algebraically, these are $1$-dimensional linear subspaces $L \sub \bC^{n+1}$.
\end{exam}

\begin{exam}
  The Grassmannian $\on{Gr}(k, V)$ is the moduli space of $k$-dimensional linear subspaces of a given vector space $V$, or equivalently, of $(k-1)$-dimensional linear subspaces of the projectivization $\P(V)$.
\end{exam}

One can then solve certain problems in enumerative geometry (e.g. ``how many lines in $\P^3(\bC)$ intersect four given general lines?'') by analyzing the cohomology of Grassmannians (see: Schubert calculus).

\begin{exam}
  Given an algebraic variety $X$, the moduli space of (algebraic) vector bundles on $X$ is the set of vector bundles on $X$ modulo isomorphism.
  More generally, given an algebraic group $G$, the moduli space of principal $G$-bundles is the set of principal $G$-bundles on $X$ modulo isomorphism.
\end{exam}

\begin{exam}
  Given an algebraic variety $X$, the moduli space of coherent sheaves on $X$ is the set of coherent sheaves on $X$ modulo isomorphism.
\end{exam}

In order to get a useful theory of moduli spaces, we will need to refine this naive picture in two ways.

\begin{thm}[Grothendieck]
  A scheme $X$ over a field $k$ is completely determined by its functor of points, i.e., the functor
  \[ X : \CAlg_k \to \Set \]
  sending a commutative $k$-algebra $A$ to the set of $A$-valued points $X(A)$.
  (Recall that an $A$-valued point of $X$ is a morphism of schemes $\Spec(A) \to X$.)
\end{thm}

Thus we may regard a scheme as a family (or fibration) of sets $X(A)$ parametrized by commutative algebras $A$.
That is, a scheme is literally a ``scheme'' prescribing the $A$-valued points of some algebro-geometric space.

For example, in order to define complex projective $n$-space as a scheme, it is not sufficient to specify the set $\P^n(\bC)$ as above.
Instead, we must specify the sets
\[ \P^n(A) = \{ A\text{-linear surjections}~A^{n+1} \twoheadrightarrow L \mid L~\text{projective}~A\text{-module of rank }1 \}, \]
for all commutative $\bC$-algebras $A$, together with the natural maps $\P^n(A) \to \P^n(A')$ for ring homomorphisms $A \to A'$.

Secondly, we will need to be smarter about quotients.
Let $X$ be a set and $R \sub X \times X$ an equivalence relation on $X$.
We may depict this via the diagram
\[\begin{tikzcd}
  R \ar[shift left=1.75ex]{r}{\pr_1}\ar[swap,shift right=1.75ex]{r}{\pr_2}
  & X \ar[swap]{l}{s}
\end{tikzcd}\]
where $\pr_i$ are the projections and $s : X \to X\times X$ is the diagonal (which factors through $R$ since the relation is reflexive).
This diagram defines a \emph{groupoid} $[X/R]$:
\begin{itemize}
  \item The objects of $[X/R]$ are the elements of $X$.
  \item The morphisms of $[X/R]$ are the elements of $R$.
  \item The ``source'' and ``target'' maps $\on{Mor}[X/R] \to \on{Obj}[X/R]$ are given by the projections $\pr_1$ and $\pr_2$.
  \item The ``identity'' map $\on{Obj}[X/R] \to \on{Mor}[X/R]$ (sending an object to its identity morphism) is given by the diagonal $s$.
  \item Composition of morphisms is well-defined since the relation is transitive.
  \item All morphisms are invertible since the relation is symmetric.
\end{itemize}
Note that the set of connected components $\pi_0[X/R]$ (where objects of $[X/R]$ are identified if and only if they connected by some chain of morphisms) is canonically isomorphic to the usual set-theoretic quotient $X/R$.
Unlike $X/R$, the quotient groupoid $[X/R]$ remembers \emph{how} elements are identified.

Similarly, if we have a group $G$ acting on the set $X$, there is a quotient groupoid $[X/G]$ defined by the diagram
\[\begin{tikzcd}
  G\times X \ar[shift left=1.75ex]{r}{\pr_2}\ar[swap,shift right=1.75ex]{r}{\mrm{act}}
  & X \ar[swap]{l}{s}
\end{tikzcd}\]
where the ``source'' map is the projection $(g,x) \mapsto x$, the ``target'' map is the action map $(g,x) \mapsto g\cdot x$, and the ``identity'' map is $x \mapsto (e, x)$ where $e\in G$ is the neutral element.
Whereas the set-theoretic quotient $X/G$ remembers only the binary information of whether two elements $x,y \in X$ belong to the same equivalence class, the groupoid $[X/G]$ contains one isomorphism $x \simeq y$ for every $g \in G$ such that $g\cdot x = y$.

As we will see in this course, it is highly advantageous to allow moduli spaces to have \emph{groupoids} of points rather than sets.
Combining these two ideas leads one to the replace our naive definition of moduli space above by the following:

\begin{defn}
  A \emph{stack} (over a field $k$) is a functor
  \[ \sM : \CAlg_k \to \Grpd \]
  satisfying certain conditions.
\end{defn}

One of our goals will be to prove the following theorem:

\begin{thm}\label{thm:intro/Bun}
  Let $C$ be a smooth proper curve over $k$ and $G$ an algebraic group over $k$.
  Let
  $$\sM_{\Bun_G(C)} : A \mapsto \Bun_G(C_A)^\simeq$$
  be the stack defined by the functor sending a commutative algebra $A$ to the groupoid of principal $G$-bundles on the scheme $C_A := C \otimes_k A$.
  Then $\sM_{\Bun_G(C)}$ is a \emph{smooth} algebraic stack.
\end{thm}

In particular, taking $G$ to be the general linear group $\GL_n$, we find the same holds for the moduli stack of rank $n$ vector bundles on $C$.
We also have a similar result for the moduli stack $\sM_{\Coh(C)}$ of coherent sheaves on $C$.

\begin{rem}
  It is important here to work with groupoids rather than sets; the functor sending $A$ to the set $\pi_0 \Bun_G(C)^\simeq$ of isomorphism classes of principal $G$-bundles on $C \otimes_k A$ is very poorly behaved.
\end{rem}

\begin{warn}
  Making the definition of ``stack'' above precise is more involved than for schemes, since $\Grpd$ is naturally a $2$-category (where the $2$-morphisms are natural transformations).
  In other words, in practice we only want to distinguish between groupoids up to \emph{equivalence} rather than isomorphism.
\end{warn}

One way to handle this subtlety is to use the language of $2$-categories and pseudofunctors.
In this course we will instead use the language of $\infty$-categories.
This language is much more general than that of $2$-categories, but we will see that it has some practical advantages even when all $\infty$-categories involved are ``$2$-truncated''.
Moreover, the extra generality of $\infty$-categories will also be useful to us later in the course when we study concepts like the derived category of (quasi-)coherent sheaves on a stack, and the cotangent complex of a stack.

There are several excellent textbook accounts such as \cite{HTT,Kerodon,CisinskiHCHA} focused on developing the extensive technical machinery necessary to justify the existence of the theory of $\infty$-categories.
Here we'll only give a brief and informal introduction, taking well-foundedness of the theory for granted.
We'll focus on understanding how this language is useful in the study of ``derived'' or ``homotopical'' objects that are most naturally regarded up to something weaker than isomorphism (such as quasi-isomorphism, equivalence, or weak equivalence).
Relevant examples for us will be:
\begin{itemize}
  \item the singular (co)chain complex of a topological space,
  \item the cotangent complex of a scheme or stack,
  \item the derived category of (quasi-)coherent sheaves on a scheme or stack,
  \item stacks and higher stacks.
\end{itemize}

Indeed, it is only by working $\infty$-categorically that we are able to take advantage of the descent properties satisfied by these objects.
For example, we have:

\begin{thm}\label{thm:intro/descent}\leavevmode
  \begin{thmlist}
    \item
    The assignment $X \mapsto \Ccoh(X; \bZ)$, sending a topological space to its complex of singular cochains, is a sheaf with values in the derived \inftyCat of abelian groups.
    
    \item
    The assignment $X \mapsto \Dqc(X)$, resp. $X \mapsto \Dcoh(X)$, sending a scheme to its stable \inftyCat of quasi-coherent (resp. coherent) sheaves, is a sheaf of \inftyCats.
  \end{thmlist}
\end{thm}

This ability to speak about sheaves of derived objects is one of the fundamental features of $\infty$-categories and has many far-reaching consequences (of which we shall only see a small glimpse).
Note in contrast that $X \mapsto \Ccoh(X; \bZ)$ does not satisfy descent when regarded with values in the usual derived category of abelian groups, nor in the category of chain complexes.
Likewise, the descent condition fails also if we replace $X \mapsto \Dqc(X)$ by the assignment sending $X$ to the usual derived category (regarded as an ordinary category).

Another key aspect of $\infty$-category theory is a very flexible framework for deriving functors.
Recall that the usual framework of homological algebra allows us to derive functors between abelian categories.
Using the language of $\infty$-categories we may even derive constructions like the tangent bundle (or cotangent sheaf); the corresponding derived functor, the derived tangent bundle (or cotangent complex), is naturally defined on the \inftyCat of derived stacks, which is not a derived category in the traditional sense.
In fact, the $\infty$-categorical approach is useful even for deriving abelian categories like the category of quasi-coherent sheaves on a scheme or stack: by descent we may reduce to the case of affine schemes, thereby bypassing most of the technicalities in classical approaches.

We will benefit from this flexibility in Lectures 8--10 in our study of the cotangent complex.
We will see how to compute the cotangent complexes of the moduli stacks $\sM_{\Vect(X)}$, $\sM_{\Bun_G(X)}$, and $\sM_{\Coh(X)}$ when $X$ is a curve.
This will establish the smoothness asserted in \thmref{thm:intro/Bun} and will also be used in our proof of algebraicity.
In fact, these moduli stacks are algebraic even when $X$ is higher-dimensional.
We will give a new proof of this fact by showing that the \emph{derived} versions of these moduli stacks have nice cotangent complexes even over higher-dimensional $X$.
We then appeal to the Artin--Lurie representability criterion to deduce algebraicity.
Thus our main goal in this part will be the following fundamental result, refining \thmref{thm:intro/Bun} (which unfortunately does not seem to be covered in the standard textbooks \cite{SAG,HAG2,GaitsgoryRozenblyum}):

\begin{thm}
  Let $X$ be a smooth proper scheme over a field $k$.
  Let $\sM$ be the derived moduli stack $\sM_{\Vect(X)}$ of vector bundles on $X$, $\sM_{\Coh(X)}$ of coherent sheaves on $X$, or $\sM_{\Bun_G(X)}$ of principal $G$-bundles on $X$ (for an algebraic group $G$).
  Then $\sM$ is a derived algebraic stack which is ``homotopically smooth'' in the sense that its cotangent complex $\L_{\sM}$ is perfect.
  In addition, if $X$ is of dimension $\le d$, then $\L_{\sM}$ is of Tor-amplitude $\le d-1$.
  In particular, it is smooth if $X$ is a curve and \emph{quasi-smooth} if $X$ is a surface.
\end{thm}

Let $\sM_\cl$ denote the classical moduli stack of vector bundles (resp. coherent sheaves, principal $G$-bundles) on $X$.
There is a surjective closed immersion $\sM_\cl \hook \sM$, which is an isomorphism if and only if $X$ is a curve, hence if and only if $\sM$ is smooth.
As soon as $X$ is of dimension $2$ or greater, $\sM_\cl$ is singular with unbounded cotangent complex while $\sM$ is still homotopically smooth.
In other words, the homotopical smoothness of $\sM$ is a property that can be witnessed only through the lens of derived algebraic geometry.
This phenomenon, labelled ``hidden smoothness'' by Kontsevich \cite{Kontsevich}, is the source of ``virtual'' phenomena on $\sM$.

We will conclude the notes with a brief introduction to the cohomological approach to (virtual) intersection theory on stacks following \cite{virtual}.
We will see how $\infty$-categorical descent allows us to work effectively with cohomology and Borel--Moore homology of stacks.
In particular, this gives a conceptual approach to Kontsevich's virtual fundamental class and its importnat properties, such as the virtual torus localization formula (with no need for the usual auxiliary technical hypotheses on global smooth embeddings or global resolutions and such).

These notes originate from a course taught at the NCTS, Taipei, in Fall 2022.
Thanks to Chun-Chung Hsieh, Hsueh-Yung Lin, Wille Liu, Justin Wu, Nawaz Sultani, and the other participants, as well as Emile Bouaziz, Peng Du, Panagiotis Jones, Antoine Labelle, Charanya Ravi, Gabriel Ribeiro, and Drimik Roy Chowdhury, for comments, questions, and corrections.
I am very grateful to the NCTS for the opportunity to give the lecture series, and for the 2021 NCTS Young Theoretical Scientist Award through which I was partially supported while giving the course.
I would also like to thank the National University of Singapore for hosting me during part of the course.
Additionally, I acknowledge support from the NSTC Grant 110-2115-M-001-016-MY3.


\section{\texorpdfstring{$\infty$}{Infinity}-Categories}
\label{sec:oocats}

\ssec{Simplicial sets}

      
      

  For every integer $n\ge 0$, let $[n]$ denote the finite set $\{0,1,\ldots,n\}$.
  Let $\bDelta$ denote the category whose objects are the finite sets $[n]$, for all $n\ge 0$, and whose morphisms are order-preserving maps.
  \begin{defn}
    A \emph{simplicial set} is a functor $X : \bDelta^\op \to \Set$, i.e., a contravariant functor on $\bDelta$ with values in the category of sets.
    A morphism of simplicial sets is a natural transformation of the corresponding functors.
    The category of simplicial sets is the functor category $\SSet = \Fun(\bDelta^\op, \Set)$.
  \end{defn}

  In other words, a simplicial set $X$ is a sequence of sets $X_n := X([n])$ together with a collection of maps $\alpha^* : X_n \to X_m$ for all order-preserving maps $\alpha : [m] \to [n]$, subject to the identities $\id^* = \id$ and $(\beta\circ\alpha)^* = \alpha^* \circ \beta^*$ whenever $\alpha$ and $\beta$ are composable.
  Elements of the set $X_n$ are called \emph{$n$-simplices} of $X$.

  \begin{exam}
    Given a set $X$, we let $c(X)$ denote the \emph{constant} simplicial set on $X$.
    We have $c(X)_n = X$ for all $n$, and every map $\alpha : [m] \to [n]$ in $\bDelta$ induces the identity map $\id : X \to X$.
    The assignment $X\mapsto c(X)$ defines a canonical functor
    \[ c : \Set \to \SSet \]
    which is fully faithful.
  \end{exam}

  \begin{notat}
    Given integers $n\ge 0$ and $0\le i\le n$, we denote by
    \[ \delta_n^i : [n-1] \to [n] \]
    the injective map that ``skips'' $i$, and by
    \[ \sigma_n^i : [n+1] \to [n] \]
    the surjective map that ``doubles'' $i$.
    Given a simplicial set $X$, the induced maps
    \[ d_n^i := X(\delta_n^i) : X_n \to X_{n-1} \]
    are called \emph{face maps} and the induced maps
    \[ s_n^i := X(\sigma_n^i) : X_n \to X_{n+1} \]
    are called \emph{degeneracy maps}.
  \end{notat}

  \begin{rem}
    To specify a simplicial set $X$, it is enough to specify the sets $X_n$ along with face and degeneracy maps satisfying certain relations (which one can read off the simplex category $\bDelta$).
    We will often depict $X$ by the diagram
    \[
      \cdots
      \rightrightrightrightarrows X_2
      \rightrightrightarrows X_1
      \rightrightarrows X_0
    \]
    where for simplicity we only draw the face maps.
  \end{rem}

  \begin{exam}
    For every $n\ge 0$, the \emph{standard $n$-simplex} is a simplicial set $\Delta^n$ whose set of $k$-simplices ($k\ge 0$) is
    \[  \Delta^n_k = \Hom_\bDelta([k], [n]). \]
    That is, an $k$-simplex of $\Delta^n$ is an increasing sequence of integers $(a_0,\ldots,a_k)$ with $0\le a_i \le a_j \le n$ for all $i\le j$.
    Given a morphism $\alpha : [j] \to [k]$, the induced map
    \[ \alpha^* : \Delta^n_k \to \Delta^n_j \]
    sends $([k] \to [n])$ to the composite $([j] \to [k] \to [n])$.
  \end{exam}

  \begin{rem}
    Let $X$ be a simplicial set.
    By the Yoneda lemma, the datum of an $n$-simplex $x \in X_n$ is the same as that of a morphism $x : \Delta^n \to X$.
  \end{rem}

  \begin{exam}
    For every $n\ge 0$ and $0\le k\le n$, let $\Delta^{n-1} \to \Delta^n$ denote the map of standard simplices induced by $\delta_n^k : [n-1] \to [n]$; on $i$-simplices it sends $([i] \to [n-1])$ to $([i] \to [n-1] \to [n])$.
    Its image is a simplicial subset
    \[ \partial^k\Delta^n \sub \Delta^n \]
    called the \emph{$k$th face} of the standard $n$-simplex.
    The union of $\partial^k\Delta^n$ over $k$ is a simplicial subset
    \[ \partial\Delta^n \sub \Delta^n \]
    called the \emph{boundary} of the standard $n$-simplex.
  \end{exam}

  \begin{exam}
    For every $n\ge 0$ and $0\le k\le n$, the union of the faces $\partial^j\Delta^n$ over $j\ne k$ is a simplicial subset
    \[ \Lambda^n_k \sub \Delta^n \]
    called the \emph{$k$th horn} of the standard $n$-simplex.
    In other words, $\Lambda^n_k$ is the boundary $\partial\Delta^n$ minus the $k$th face $\partial^k\Delta^n$.
  \end{exam}

\ssec{Categories as simplicial sets}

  \begin{constr}
    Let $\sC$ be a category.
    The \emph{nerve of $\sC$} is a simplicial set $N(\sC)$ defined as follows.
    For every $n\ge 0$, we set
    \[ N(\sC)_n := \Fun([n], \sC). \]
    For every $\alpha : [m] \to [n]$ in $\bDelta$, $\alpha^* : N(\sC)_n \to N(\sC)_m$ is given by composition: $([n] \to \sC) \mapsto ([m] \to [n] \to \sC)$.
  \end{constr}

  \begin{rem}
    In other words, $n$-simplices of $N(\sC)$ are strings
    \[ c_0 \to c_1 \to \cdots \to c_n \]
    of morphisms in $\sC$ (where $c_i$ are objects of $\sC$).
    For example, $0$-simplices are objects of $\sC$, $1$-simplices are morphisms of $\sC$, $2$-simplices are diagrams $c_0 \to c_1 \to c_2$ in $\sC$, and so on.
    Informally speaking, the simplicial set $N(\sC)$ contains all the information about the category $\sC$.
  \end{rem}

  \begin{exer}\label{exer:nerve}
    The assignment $\sC \mapsto N(\sC)$ determines a fully faithful functor $N : \Cat \to \SSet$ from the category of categories to the category of simplicial sets.
    Moreover, it admits a left adjoint $\tau : \SSet \to \Cat$ (hint: left Kan extension along $\Delta \to \Cat$).
  \end{exer}

  \exerref{exer:nerve} means that we can think of categories as ``special'' simplicial sets, or conversely of simplicial sets as ``generalized'' categories.
  The second point of view leads to the question: to what extent do simplicial sets admit a category theory?

  \begin{defn}
    Let $X$ be a simplicial set.
    \begin{defnlist}
      \item
      An \emph{object} $x$ of $X$ is a $0$-simplex $x \in X_0$, or by Yoneda, a morphism $x : \Delta^0 \to X$.

      \item
      A \emph{morphism} $f$ in $X$ is a $1$-simplex $f \in X_1$, or by Yoneda, a morphism $f : \Delta^1 \to X$.

      \item
      The \emph{source} (resp. \emph{target}) of a morphism $f$ in $X$ is the image of the corresponding $1$-simplex $f \in X_1$ along the face map $d_1^1 : X_1 \to X_0$ (resp. $d_1^0 : X_1 \to X_0$).
      Equivalently, in terms of the corresponding morphism $f : \Delta^1 \to X$, these are the composites
      \[ \Delta^0 \to \Delta^1 \xrightarrow{f} X \]
      with the inclusions of the two faces of the standard $1$-simplex.

      \item
      For an object $x$ of $X$, the \emph{identity morphism} $\id_x \in X_1$ is the image of the $0$-simplex $x \in X_0$ by the degeneracy map $s_0^0 : X_0 \to X_1$.
      Equivalently, in terms of the corresponding morphism $x : X^0 \to X$, it is the composite
      \[ \Delta^1 \to \Delta^0 \xrightarrow{x} X \]
      with the (unique) morphism $\Delta^1 \to \Delta^0$.
    \end{defnlist}
  \end{defn}

  \begin{notat}
    We write $s := d_1^1 : X_1 \to X_0$ and $t := d_1^0 : X_1 \to X_0$ for the source and target maps, respectively.
    Given two objects $x,y \in X_0$, we use the shorthand $f : x \to y$ to indicate that $f \in X_1$ is a $1$-simplex with source $x = s(f)$ and target $y = t(f)$.
  \end{notat}

  A defining feature of morphisms in category theory is that they are \emph{composable}.
  So to justify the above definitions we need to understand how composition should work in this context.

  \begin{rem}
    Consider the horn $\Lambda^2_1$, which we may depict as ``two $\Delta^1$'s attached at a $\Delta^0$'':
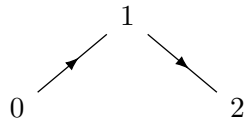
    \[ \begin{tikzcd}
      & 1 \ar[dash, marrow]{rd} &
      \\
      0 \ar[dash, marrow]{ru}
      &
      & 2
    \end{tikzcd} \]
    A more precise way of putting this is that there is a cartesian and cocartesian square of simplicial sets
    \[ \begin{tikzcd}
      \Delta^0 \ar[hookrightarrow]{r}\ar[hookrightarrow]{d}
      & \Delta^1 \ar[hookrightarrow]{d}
      \\
      \Delta^1 \ar[hookrightarrow]{r}
      & \Lambda^2_1.
    \end{tikzcd} \]
  \end{rem}

  \begin{defn}
    Let $X$ be a simplicial set.
    \begin{defnlist}
      \item
      A \emph{composable pair of morphisms} in $X$ is a morphism $\Lambda^2_1 \to X$.
      This is the same data as that of two morphisms $f$ and $g$ in $X$ such that $t(f) = s(g)$.

      \item
      A \emph{composition} of a composable pair $\sigma : \Lambda^2_1 \to X$ is a $2$-simplex $\widetilde{\sigma}$ extending $\sigma$.
      That is, it is a morphism $\widetilde{\sigma} : \Delta^2 \to X$ such that $\widetilde{\sigma}|_{\Lambda^2_1} = \sigma$.

      \item
      More generally, given a morphism $\sigma : \Lambda^n_k \to X$, where $n\ge 2$ and $0<k<n$, a \emph{composition} of $\sigma$ is an extension $\widetilde{\sigma} : \Delta^n \to X$.
    \end{defnlist}
  \end{defn}

  The question now becomes: in which simplicial sets do compositions exist (uniquely)?
  If we include the uniqueness requirement, it turns out that this exactly characterizes nerves of categories.

  \begin{prop}[Grothendieck--Segal]\label{prop:nerve}
    Let $X$ be a simplicial set.
    Then the following two conditions are equivalent:
    \begin{thmlist}
      \item
      $X$ belongs to the essential image of the fully faithful functor $N : \Cat \to \SSet$ (see \exerref{exer:nerve}).

      \item
      For every $n\ge 2$ and every $0<k<n$, the map
      \[ \Hom(\Delta^n, X) \to \Hom(\Lambda^n_k, X) \]
      given by restriction along the inclusion $\Lambda^n_k \sub \Delta^n$, is bijective.
    \end{thmlist}
  \end{prop}

\ssec{Groupoids and Kan complexes}

  In the previous subsection we attempted to set up a ``category theory'' for simplicial sets, but just ended up recovering usual category theory.
  Things start to get more interesting if we relax the \emph{uniqueness} condition on composition.
  Let's first play with this idea in the context of groupoids (categories in which all morphisms are invertible).

  \begin{rem}
    There is a variant of \propref{prop:nerve} which characterizes \emph{groupoids} $\sC$ by the bijectivity of the restriction map
    \[ \Hom(\Delta^n, X) \to \Hom(\Lambda^n_k, X) \]
    for all $n\ge 2$ and $0\le k\le n$.
    The edge cases $k=0$ and $k=n$ correspond to invertibility of morphisms (rather than composition).
  \end{rem}


  \begin{defn}
    A simplicial set X is called a \emph{Kan complex} if
    \[ \Hom(\Delta^n, X) \to \Hom(\Lambda^n_k, X) \]
    is \emph{surjective} for all $n\ne 0$ and $0\le k\le n$.
    For example, for a category $\sC$, $N(\sC)$ is a Kan complex if and only if $\sC$ is a groupoid.
  \end{defn}

  Kan complexes (named after Daniel Kan) are like generalized groupoids where compositions (and inverses) exist, but not uniquely.
  This weak composition still turns out to yield surprisingly good behaviour, at least up to homotopy:

  \begin{thm}[Milnor]\label{thm:Milnor}
    There is an equivalence between the homotopy category\footnote{%
      The \emph{homotopy category} of CW complexes is the categorical localization (in the sense of Gabriel--Zisman, see \cite[Chap.~I]{GabrielZisman}) with respect to weak homotopy equivalences.
      This is analogous to the derived category (which is a categorical localization with respect to quasi-isomorphisms).
      The homotopy category of Kan complexes is defined similarly; for the definition of homotopy groups and weak homotopy equivalences of Kan complexes, see \cite[Chap.~VI, \S 3]{GabrielZisman}.
      See \cite[Chap.~VII, \S 3]{GabrielZisman} for a proof of \thmref{thm:Milnor}.
    } of CW complexes and that of Kan complexes, given by the construction $X \mapsto \Sing(X)_\bullet$ sending a CW complex to its singular simplicial set, whose $n$-simplices are continuous maps $\Delta^n_\top \to X$ (with $\Delta^n_\top$ the \emph{topological} standard $n$-simplex).
  \end{thm}

  In view of \thmref{thm:Milnor} we can think of objects ($0$-simplices) of a Kan complex as points in a space, and of morphisms as paths between points.
  Through this equivalence, we see that up to homotopy, composition in a Kan complex is not that bad: for example, composites not only exist but are unique at least \emph{up to homotopy}.
  This is encouraging.

\ssec{\texorpdfstring{$\infty$}{Infinity}-Categories as weak Kan complexes}

  Building on what we have seen so far, our next hope to isolate a class of simplicial sets where composition is well-behaved up to homotopy, but where not all morphisms are required to be invertible.
  \[\begin{tikzcd}
    \text{groupoid} \ar{r}\ar{d}
    & \text{category} \ar{d}
    \\
    \text{Kan complex} \ar{r}
    & ?
  \end{tikzcd}\]

  \begin{defn}[Boardman--Vogt]
    A simplicial set $X$ is called a \emph{weak Kan complex} (a.k.a. \emph{quasi-category}) if the restriction map
    \[ \Hom(\Delta^n, X) \to \Hom(\Lambda^n_k, X) \]
    is \emph{surjective} for all $n\ge 2$ and $0<k<n$.
  \end{defn}

  \begin{constr}
    Given a weak Kan complex $X$ and two morphisms $f,g : x \to y$ in $X$, a \emph{homotopy} $f \sim g$ is a $2$-simplex $\sigma : \Delta^2 \to X$ of the form
    \[ \begin{tikzcd}
      & y \ar{rd}{\id} & \\
      x \ar{ru}{f}\ar{rr}{g}
      & & y.
    \end{tikzcd} \]
    This defines an equivalence relation on the set of morphisms $x \to y$.
    The \emph{homotopy category} $\h(X)$ is the category whose set of objects is $X_0$ and, for $x,y \in X_0$, the set $\Hom_{\h(X)}(x,y)$ is the set of equivalence classes of morphisms $x \to y$.
    Since $X$ is a weak Kan complex, we can compose such equivalence classes and check that this gives a well-defined category $\h(X)$.
    Moreover, one can prove that $\h(X) \simeq \tau(X)$ (where $\tau$ is as in \exerref{exer:nerve}).
  \end{constr}

  \begin{defn}
    Let $X$ be a weak Kan complex.
    A morphism $f : x \to y$ is an \emph{isomorphism} if it is invertible, i.e., if there exists a morphism $g : y \to x$ and homotopies $f\circ g \sim \id_y$ and $g\circ f\sim \id_x$.
    One can prove that a morphism is an isomorphism if and only if it induces an isomorphism in $\h(X)$.
    We say that $X$ is an \emph{$\infty$-groupoid} if every morphism in $X$ is an isomorphism.
  \end{defn}

  \begin{thm}[Joyal]
    Let $X$ be a weak Kan complex.
    Then $X$ is an $\infty$-groupoid if and only if $X$ is a Kan complex.
  \end{thm}

  \begin{rem}\label{rem:adf-8h}
    We can think of weak Kan complexes as those simplicial sets which admit compositions up to coherent homotopy.
    Moreover, after the extensive work of André Joyal and Jacob Lurie, weak Kan complexes do admit a full ``category theory'':
    \begin{defnlist}
      \item
      Given weak Kan complexes $X$ and $Y$, the internal hom $\uHom(X, Y)$ behaves like a functor category $\Fun(X, Y)$.
      Recall that the $n$-simplices of $\uHom(X, Y)$ are maps $\Delta^n\times X\to Y$.

      \item
      Given a weak Kan complex $X$ and objects $x,y$ of $X$, there is a Kan complex $\Maps_X(x,y)$ of maps $x \to y$, defined by the cartesian square
      \[\begin{tikzcd}
        \Maps_X(x,y) \ar{r}\ar{d}
        & \uHom(\Delta^1,X) \ar{d}{(s,t)}
        \\
        \Delta^0 \ar{r}{(x,y)}
        & X \times X.
      \end{tikzcd}\]
      This is a replacement for the Hom-set $\Hom(x,y)$.
    \end{defnlist}
  \end{rem}

  This justifies the following definition:

  \begin{defn}
    An \emph{\inftyCat} is a weak Kan complex.
  \end{defn}

  \begin{rem}
    The only difference between the two terms is that \emph{weak Kan complex} refers to a specific model (``shadow'') of the platonic notion of \emph{\inftyCat}; similarly for Kan complexes vs. $\infty$-groupoids.
    In accordance with this point of view, we will simply write $\sC$ instead of $N(\sC)$ when we want to think of an ordinary category $\sC$ as an \inftyCat (as opposed to a weak Kan complex).
    Similarly, we will use letters like $\sC$ and $\sD$ instead of $X$ and $Y$ when we want to think of them as \inftyCats, and we will write $\Fun(\sC, \sD)$ instead of $\uHom(X,Y)$.
  \end{rem}

\ssec{The \texorpdfstring{$\infty$}{infinity}-category of (\texorpdfstring{$\infty$}{infinity}-)groupoids}

  Recall that for a category $\sC$ we have the weak Kan complex $N(\sC)$ in which an $n$-simplex is determined by the following data:
  \begin{itemize}
    \item objects $C_i \in \sC$ for $0\le i\le n$;
    \item morphisms $f_{i,j} : C_i \to C_j$ for $0\le i< j\le n$;
  \end{itemize}
  satisfying the relations $f_{j,k} \circ f_{i,j} = f_{i,k}$ for all $0\le i< j< k\le n$.

  \begin{constr}
    Let $\sC$ be a $2$-category.
    The \emph{nerve} (or \emph{Duskin nerve}) of $\sC$ is a simplicial set $N^D(\sC)$.
    An $n$-simplex of $N^D(\sC)$ is determined by the following data:
    \begin{itemize}
      \item objects $C_i \in \sC$ for $0\le i\le n$;
      \item morphisms $f_{i,j} : C_i \to C_j$ for $0\le i< j\le n$;
      \item $2$-morphisms $f_{j,k} \circ f_{i,j} \Rightarrow f_{i,k}$ for all $0\le i< j< k\le n$;
    \end{itemize}
    This data is required to satisfy certain compatibility relations involving $\mu_{i,j,k}$, $\mu_{i,j,l}$, $\mu_{i,k,l}$, and $\mu_{j,k,l}$ (for $0\le i<j<k<l \le n$).
  \end{constr}

  Informally speaking, the difference between $N^D(\sC)$ and the nerve of the underlying $1$-category (where we discard the $2$-morphisms) is that the diagrams
  \[\begin{tikzcd}
    C_i \ar{r}{f_{i,j}}\ar[bend right=50,swap]{rr}{f_{i,k}}
    & C_j \ar{r}{f_{j,k}}
    & C_k
  \end{tikzcd}\]
  only commute up to the specified natural transformation $\mu_{i,j,k}$ (which need not be invertible).

  \begin{thm}[Duskin]
    Let $\sC$ be a $2$-category.
    The following conditions are equivalent:
    \begin{thmlist}
      \item
      $\sC$ is a $(2,1)$-category; i.e., the $2$-morphisms of $\sC$ are all invertible.

      \item
      $N^D(\sC)$ is a weak Kan complex.
    \end{thmlist}
  \end{thm}

  \begin{notat}
    Let $\sC$ be a category, resp. $(2,1)$-category.
    We will write simply $\sC$ for the \inftyCat whose underlying weak Kan complex is $N(\sC)$, resp. $N^D(\sC)$.
  \end{notat}

  \begin{exam}
    Groupoids naturally form a $2$-category whose $2$-morphisms are natural transformations.
    We denote by $\Grpd$ the $(2,1)$-category where we discard the non-invertible $2$-morphisms.
    We also use the same notation for the associated \inftyCat (whose underlying weak Kan complex is $N^D(\Grpd)$).
  \end{exam}

  \begin{constr}\label{constr:Kan}
    Consider the (large) simplicial set $\Kan_\bullet$ in which an $n$-simplex is given by the following data:
    \begin{itemize}
      \item Kan complexes $K_i$ for $0\le i\le n$.
      \item Maps of Kan complexes $f_{i,j} : K_i \to K_j$ for $0\le i\le j\le n$.
      \item A system of ``coherent homotopies'' up to which $f_{i,j}$ are compatible under composition.
    \end{itemize}
    For example, a $2$-simplex is a tuple $(X,Y,Z,f,g,h,\sigma)$, where $X$, $Y$ and $Z$ are Kan complexes, $f : X \to Y$, $g : Y \to Z$, $h : X \to Z$ are maps, and $\sigma \in \uHom(X, Z)_1$ with $d_1^0(\sigma) = g\circ f$ and $d_1^1(\sigma) = h$.
    Here $\sigma$ is a ``homotopy'' $g\circ f \simeq h$.
    The term ``coherent'' is a shorthand which indicates that not only do we have such homotopies, we are also given higher homotopies between these homotopies (starting from $n\ge 3$), even higher homotopies between those homotopies, and so on.
  \end{constr}

  \begin{rem}
    The simplicial set $\Kan_\bullet$ is a weak Kan complex.
    In fact, it is an instance of a general construction called the \emph{homotopy coherent nerve} which takes a simplicially enriched category (in this case, that of Kan complexes) as input and yields a weak Kan complex as output.
    See \kerocite{00KS}.
  \end{rem}

  \begin{rem}\label{rem:07pgsu}
    There is a map of weak Kan complexes $N(\Set) \to \Kan_\bullet$ that sends $X \mapsto c(X)$ on $0$-simplices.
    As a functor of \inftyCats, it is fully faithful with essential image spanned by Kan complexes $X$ that are homotopy equivalent to a constant simplicial set, or equivalently, which satisfy $\pi_i(X) = 0$ for all $i>0$.
  \end{rem}

  Later on, we will see how the \inftyCat corresponding to the weak Kan complex $\Kan_\bullet$ is a very fundamental object called the \inftyCat of \emph{anima}.
  In practice, we will work with this \inftyCat by manipulating its universal properties and deliberately avoid any considerations involving the simplices of $\Kan_\bullet$.

  \begin{notat}
    We write $\Grpdoo$ for the \inftyCat corresponding to the weak Kan complex $\Kan_\bullet$ (cf.\ \remref{rem:adf-8h}).
    By \thmref{thm:Milnor}, $\Grpdoo$ may be regarded as the \inftyCat of \inftyGrpds.
  \end{notat}


\section{Sheaves and stacks}
\label{sec:sheaves}

\ssec{Sheaves}

  Let $\sC$ be a site, i.e., a category equipped with a Grothendieck topology $\tau$.
  Roughly speaking, $\tau$ amounts to a notion of \emph{covering sieves} for every object in the category $\sC$.
  For simplicity, we will assume that $\tau$ arises from the following construction.

  \begin{constr}\label{constr:topology}
    Assume that $\sC$ admits fibred products and finite coproducts, and satisfies the following conditions:
    \begin{defnlist}
      \item
      For any finite collection of objects $X_i \in \sC$, any morphism $f : \coprod_i X_i \to Y$, and any morphism $Y' \to Y$, the canonical morphism $\coprod_i X_i \fibprod_Y Y' \to (\coprod_i X_i)\fibprod_Y Y'$ is invertible.

      \item
      Coproducts are disjoint: for any pair of objects $X$ and $Y$ in $\sC$, the fibred product $X \fibprod_{X\coprod Y} Y$ is an initial object in $\sC$.
    \end{defnlist}
    Let $S$ be a collection of morphisms in $\sC$, which contains all isomorphisms and is stable under composition, and satisfies the following conditions:
    \begin{defnlist}
      \item
      For every morphism $f : X \to Y$ in $S$, the base change $f' : X \fibprod_Y Y' \to Y'$ along any morphism $Y' \to Y$ in $\sC$ belongs to $S$.

      \item
      For every finite collection of morphisms $(f_i : X_i \to Y_i)_i$ in $S$, the induced morphism $\coprod_i X_i \to \coprod_i Y_i$ belongs to $S$.
    \end{defnlist}
    Then there is a Grothendieck topology $\tau$ on $\sC$ where a sieve on an object $X \in \sC$ is covering if and only if it contains a finite collection of morphisms $\{X_i \to X\}_i$ such that the induced morphism $\coprod_i X_i \to X$ belongs to $S$.
    We refer to this as the Grothendieck topology \emph{generated by} the collection $S$.
    See \cite[Prop.~A.3.2.1]{SAG}.
  \end{constr}

  \begin{exam}
    Let $X$ be a topological space.
    Then the category $\sU(X)$ of opens $U \sub X$ (where there is a morphism $U \to V$ if and only if $U \sub V$) admits a Grothendieck topology generated by surjections.
    In particular, a sieve on $U \sub X$ is covering if and only if it contains a finite collection of morphisms $(U_i \hook U)_i$ such that $U = \bigcup_i U_i$.
  \end{exam}

  \begin{exam}
    Let $X$ be a scheme.
    The \emph{small étale site} $X_\et$ is the category of étale morphisms $U \to X$ (where $U$ is a scheme), with the Grothendieck topology generated by surjections (equivalently, faithfully flat morphisms).
  \end{exam}

  \begin{exam}
    Let $X$ be a scheme.
    The \emph{big étale site} is the category $\Sch_{/X}$ of (arbitrary) morphisms $Y \to X$ (where $Y$ is a scheme), with the Grothendieck topology generated by étale surjections (equivalently, faithfully flat and étale morphisms).
    (We also have the étale topology on the category $\Sch$, which is the special case where $X=\Spec(\bZ)$.)
  \end{exam}

  \begin{defn}
    Let $\sV$ be an \inftyCat.
    Let $F$ be a diagram in $\sV$ indexed by an \inftyCat $I$, i.e., a functor of \inftyCats $F : I \to \sC$.
    Suppose given an object $V \in \sV$ and a natural transformation $\alpha : V_\mrm{cst} \to F$ where $V_\mrm{cst}$ denotes the constant diagram $(i \in I) \mapsto (V \in \sV)$.
    We say that the pair $(V, \alpha)$ in $\sV$ \emph{exhibits $V$ as the limit of $F$} if for every object $V' \in \sV$ the induced functor of mapping \inftyGrpds
    \[ \Maps_\sV(V', V) \to \Maps_{\Fun(I, \sV)}(V'_\mrm{cst}, F), \]
    sending $(V' \to V) \mapsto (V'_\mrm{cst} \to V_\mrm{cst} \to F)$, is invertible.
    In this case, we will write
    \[ V \simeq \lim_{i\in I} F_i := \lim(F). \]
  \end{defn}

  \begin{defn}
    Let $\sV$ be an \inftyCat.
    A \emph{presheaf} on $\sC$ with values in $\sV$ is a functor $F : \sC^\op \to \sV$.
    A presheaf $F$ is a \emph{sheaf} if it satisfies the following conditions:
    \begin{defnlist}
      \item $F$ sends finite coproducts in $\sC$ to products in $\sV$.
      In other words, for every finite collection $(X_i)_i$ of objects of $\sC$, the canonical morphism
      \[ F(\coprod_i X_i) \to \prod_i F(X_i) \]
      is invertible.

      \item For every morphism $f : U \to X$ in $S$, let $U_\bullet$ denote the \v{C}ech nerve of $f$, i.e., the simplicial object
      \[
        \cdots
        \rightrightrightrightarrows U \fibprod_U U \fibprod_X U
        \rightrightrightarrows U \fibprod_X U
        \rightrightarrows U
      \]
      whose $n$th term is the $(n+1)$-fold fibre power $U \fibprod_X \cdots \fibprod_X U$.
      Then the canonical map
      \[ F(X) \to \lim_{[n]\in\bDelta} F(U_n) \]
      is invertible.
      In other words, the diagram
      \[
        F(X)
        \to F(U)
        \rightrightarrows F(U \fibprod_X U)
        \rightrightrightarrows F(U \fibprod_X U \fibprod_X U)
        \rightrightrightrightarrows \cdots
      \]
      exhibits $F(X)$ as the limit of $F(U_\bullet)$.
    \end{defnlist}
  \end{defn}

  \begin{notat}
    We denote by $\Shv(\sC; \sV)$ the full subcategory of $\Fun(\sC^\op, \sV)$ spanned by sheaves.
  \end{notat}

  When $\sV$ is a $1$-category (resp. $(2,1)$-category), this is equivalent to the usual sheaf condition:

  \begin{prop}\label{prop:siris}
    Let $V^\bullet : \bDelta \to \sV$ be a cosimplicial diagram in an \inftyCat $\sV$.
    If $\sV$ is equivalent to a $1$-category, then the limit of $V^\bullet$ is identified with the equalizer of $V^0\rightrightarrows V^1$:
    \[ \lim_{[n]\in\bDelta} V^n \simeq \lim(V^0 \rightrightarrows V^1). \]
    Similarly, if $\sV$ is equivalent to a $(2,1)$-category, then it is identified with the $2$-limit:
    \[ \lim_{[n]\in\bDelta} V^n \simeq \on{2\dash\lim}(V^0 \rightrightarrows V^1 \rightrightrightarrows V^2). \]
  \end{prop}

  \begin{rem}
    More generally, say $\sV$ is an \emph{$(n,1)$-category} if all mapping \inftyGrpds $\Maps_\sV(V, V')$ are $(n-1)$-truncated (have trivial higher homotopy groups $\pi_i\Maps_\sV(V,V') = 0$ for $i\ge n$).
    In this case the limit of $V^\bullet$ is isomorphic to the limit of the restriction $V|_{\bDelta_{\le n}}$ to the full subcategory of $\bDelta$ spanned by the objects $[0], [1], \ldots, [n]$.
    (This follows from a variant of Quillen's Theorem A, because the inclusion $\bDelta_{\le n} \hook \bDelta$ is an \emph{$n$-final} functor, i.e. the category $\bDelta_{\le n} \fibprod_{\bDelta} \bDelta_{/[m]}$ has $n$-connected nerve for every $[m] \in \bDelta$.)
    Note also that a limit over $\bDelta_{\le 1}$ is (by an easy finality argument) isomorphic to the limit over the subcategory where the morphism $[1] \to [0]$ is discarded; i.e., it is the equalizer of the two parallel arrows $V^0 \twoheadrightarrow V^1$.
  \end{rem}

  \begin{defn}
    A \emph{stack} is a sheaf of groupoids on the category of schemes (with the étale topology).
  \end{defn}

  In other words, a stack $\sX$ is a functor
  \[ \sX : \Sch^\op \to \Grpd \]
  to the $(2,1)$-category of groupoids such that for every finite collection of étale morphisms $(U_i \to U)_i$ which is jointly surjective, the diagram
  \[
    \sX(U)
    \to \prod_i \sX(U_i)
    \rightrightarrows \prod_{i,j} \sX(U_i\fibprod_U U_j)
    \rightrightrightarrows \prod_{i,j,k} \sX(U_i\fibprod_U U_j\fibprod_U U_k)
  \]
  is a limit diagram in the $(2,1)$-category of groupoids.

\ssec{Bases of topologies}

  Denote by $\Aff \sub \Sch$ the full subcategory spanned by affine schemes.
  Note that the étale topology on $\Sch$ restricts to $\Aff$, and is still generated in the sense of \constrref{constr:topology} by étale surjections.

  \begin{thm}\label{thm:commissionship}
    Let $\sV$ be an \inftyCat admitting limits.
    The canonical functor
    \[ \Fun(\Sch^\op, \sV) \to \Fun(\Aff^\op, \sV), \]
    given by restriction along the inclusion $\Aff \sub \Sch$, restricts to an equivalence of \inftyCats
    \[ \Shv(\Sch; \sV) \to \Shv(\Aff; \sV). \]
  \end{thm}

  \begin{exam}
    Stacks can be defined as sheaves of groupoids on $\Aff$.
  \end{exam}

  \thmref{thm:commissionship} is a special case of a following more general result.

  \begin{defn}\label{defn:basis}
    Let $\sC$ be a site with Grothendieck topology generated by a class of morphisms $S$ as in \constrref{constr:topology}.
    Let $\sC_0 \sub \sC$ be a full subcategory which is closed under fibred products and finite coproducts, and regard it with the Grothendieck topology generated by $S \cap \sC_0$ (the subclass of morphisms in $S$ whose source and target belong to $\sC_0$).
    We say that $\sC_0$ is a \emph{basis} for $\sC$ if for every object $X\in\sC$ there exists a collection of morphisms $(Y_i \to X)_i$ such that $Y_i \in \sC_0$, the coproduct $\coprod_i Y_i$ exists in $\sC$, and $\coprod_i Y_i \to X$ belongs to $S$.
  \end{defn}

  \begin{thm}\label{thm:comp}
    In the situation of \defnref{defn:basis}, the functor
    \[ \Fun(\sC^\op, \sV) \to \Fun(\sC_0^\op, \sV) \]
    restricts to an equivalence
    \[ \Shv(\sC; \sV) \to \Shv(\sC_0; \sV) \]
    for all \inftyCats $\sV$ admitting limits.
  \end{thm}

  We can moreover give a more precise version of \thmref{thm:comp}.

  \begin{defn}
    Let $i : \sC_0 \hook \sC$ be a fully faithful functor of categories.
    Let $F_0 : \sC_0^\op \to \sV$ be a presheaf with values in an \inftyCat $\sV$ admitting limits.
    The \emph{right Kan extension} of $F_0$, denoted\footnote{%
      The explanation for the notation $i_*(F_0)$ is that the assignment $F_0 \mapsto F$ actually determines a right adjoint $i_*$ to the restriction functor $i^* : \Fun(\sC^\op, \sV) \to \Fun(\sC_0^\op, \sV)$.
    }
    \[ F := \on{RKE}_{\sC_0 \hook \sC}(F_0) := i_*(F_0) \]
    is the unique limit-preserving functor $F : \sC^\op \to \sV$ which restricts to $F_0$.
    Explicitly, it is given by the formula
    \[ F(X) \simeq \lim_{(Y, f)} F(Y) \]
    where the limit is taken over the category of pairs $(Y, f)$ where $Y \in \sC_0$ and $f : i(Y) \to X$ is a morphism in $\sC$ (and morphisms $(Y',f') \to (Y,f)$ are morphisms $Y' \to Y$ in $\sC_0$ which are compatible with $f$ and $f'$).
  \end{defn}

  \begin{thm}\label{thm:RKE}
    In the situation of \defnref{defn:basis}, let $F : \sC^\op \to \sV$ be a $\sV$-valued presheaf on $\sC$ where $\sV$ is an \inftyCat with limits.
    Then $F$ is a sheaf if and only if the following conditions hold:
    \begin{thmlist}
      \item $F_0 := F|_{\sC_0}$ is a sheaf on $\sC_0$.
      \item $F$ is the right Kan extension of $F_0$ along $\sC_0 \to \sC$.
    \end{thmlist}
  \end{thm}

  See \cite[Cor.~A.8]{Aoki} for a proof.


\section{The stack of quasi-coherent sheaves}
\label{sec:qcoh}

\ssec{Cartesian fibrations}
  
  Given a scheme $X$, quasi-coherent sheaves on $X$ form a category $\QCoh(X)$.
  For any morphism $f : X \to Y$, we have the adjoint pair of functors
  \[
    f^* : \QCoh(Y) \rightleftarrows \QCoh(X) : f_*
  \]
  where the left adjoint $f^*$ is inverse image and the right adjoint $f_*$ is direct image.
  Note that for a pair of composable morphisms $f : X \to Y$ and $g : Y \to Z$, the diagram
  \[ \begin{tikzcd}
    \QCoh(Z) \ar{rr}{g^*}\ar[swap]{rd}{(g\circ f)^*}
    & & \QCoh(Y) \ar{ld}{f^*}
    \\
    & \QCoh(X)
  \end{tikzcd} \]
  does not commute in the $1$-category of categories.
  Instead, there is an invertible natural transformation
  \[ (g \circ f)^* \to f^* \circ g^* \]
  up to which it commutes.
  Together with this extra piece of data, the diagram above does determine a commutative diagram in the \inftyCat $\Grpd$.

  In particular, the assignment $X \mapsto \QCoh(X)$ cannot be assembled into a functor $\Sch^\op \to \Cat'$ into the $1$-category $\Cat'$ of categories, but only into a functor
  \[ \QCoh : \Sch^\op \to \Cat \]
  into the $(2,1)$-category\footnote{%
    Recall that by convention, we identify $(2,1)$-categories $\sC$ with the corresponding \inftyCat (whose underlying weak Kan complex is the Duskin nerve $N^D(\sC)$).
  } (or equivalently, \inftyCat) of categories.
  Still, a precise construction of this functor requires more than just the data of the invertible natural transformation above for all pairs of morphisms $f$ and $g$; for example, we need to require compatibilities between this data whenever we have three composable morhpisms.
  This is somewhat messy, so we will prefer to take the following alternative perspective.

  \begin{defn}
    Let $\pi : \sE \to \sC$ be a functor of categories.
    Let $f : C \to D$ be a morphism in $\sC$ and $\widetilde{D} \in \sE$ a lift of $D$ (so that $\pi(\widetilde{D}) = D$).
    Let $\widetilde{f} : \widetilde{C} \to \widetilde{D}$ be a lift of $f$, i.e. $\pi(\widetilde{C}) = C$ and $\pi(\widetilde{f}) = f$.
    We say that $\widetilde{f}$ is \emph{$\pi$-cartesian} if for every $E \in \sE$ we require that the canonical map
    \[
      \Hom_{\sE}(E, \widetilde{C})
      \to \Hom_{\sE}(E, \widetilde{D}) \fibprod_{\Hom_\sC(\pi(E), D)} \Hom_{\sC}(\pi(E), \pi(\widetilde{C}))
    \]
    is bijective.
    Informally speaking, $\widetilde{f}$ is terminal among all lifts of $f$ with target $\widetilde{D}$.
  \end{defn}

  \begin{defn}
    Let $\pi : \sE \to \sC$ be a functor of categories.
    We say that $\pi$ is a \emph{cartesian fibration} if for every $E \in \sE$, every $C\in\sC$, and every morphism $f : C \to \pi(E)$ in $\sC$, there exists a lift $\widetilde{C} \in \sE$ of $C$ and a $\pi$-cartesian morphism $\widetilde{f} : \widetilde{C} \to E$ lifting $f$.
  \end{defn}

  \begin{exam}
    Let $\QCohSch$ denote the category of pairs $(X,\sF)$ where $X\in\Sch$ and $\sF \in \QCoh(X)$.
    A morphism $(X',\sF') \to (X,\sF)$ is a morphism $f : X' \to X$ together with a morphism $\phi : f^*\sF \to \sF'$ in $\QCoh(X')$.
    Given morphisms $(f,\phi) : (X',\sF') \to (X,\sF)$ and $(g,\psi) : (X'',\sF'') \to (X',\sF')$, the composite is defined by
    \[ f \circ g : X'' \to X' \to X \]
    and
    \[ g^*f^*\sF \xrightarrow{g^*\phi} g^*\sF' \xrightarrow{\psi} \sF''. \]
    Then the projection $(X,\sF) \mapsto X$ determines a functor
    \[ \QCoh_\Sch \to \Sch \]
    which is a cartesian fibration.
  \end{exam}

  \begin{constr}
    Let $\sC$ be a fixed category.
    We denote by $\Cart(\sC)$ the $(2,1)$-category whose objects are cartesian fibrations $\pi : \sE \to \sC$, whose $1$-morphisms $(\pi' : \sE' \to \sC) \to (\pi : \sE \to \sC)$ are morphisms $f : \sE' \to \sE$ that are compatible with $\pi$ and $\pi'$, and whose $2$-morphisms $f \Rightarrow g$ (where $f,g$ are morphisms $\pi' \to \pi$) are invertible natural transformations $\theta : f \Rightarrow g$ such that for all $E' \in \sE'$, $\pi : \sE \to \sC$ sends
    \[ \theta_{E'} : f(E') \to g(E') \]
    to the identity of $\pi(f(E')) = \pi(g(E'))$.
  \end{constr}

  \begin{thm}[Grothendieck]\label{thm:cartesian}
    For every category $\sC$, there is an equivalence of \inftyCats
    \[ \Fun(\sC^\op, \Cat) \to \Cart(C). \]
  \end{thm}

  \begin{defn}
    Given $F : \sC^\op \to \Cat$, the corresponding cartesian fibration is called the \emph{unstraightening} of $F$.
    Given a cartesian fibration $\pi : \sE \to \sC$, the \emph{straightening} of $\pi$ is the (essentially unique) presheaf of categories whose unstraightening is $\pi$.
  \end{defn}

  \begin{rem}
    Lurie proved a generalization of \thmref{thm:cartesian}, where $\sC$ is allowed to be an \inftyCat and $\Cat$ is replaced by the \inftyCat of \inftyCats.
  \end{rem}

  \begin{defn}
    We let $\QCoh : \Sch^\op \to \Cat$ denote the unstraightening of the cartesian fibration $\QCohSch \to \Sch$.
    We let $\QCoh^\simeq : \Sch^\op \to \Grpd$ denote the presheaf of groupoids given by sending
    \[ X \mapsto \QCoh(X)^\simeq \]
    where $(-)^\simeq$ indicates that we discard all non-invertible morphisms.
  \end{defn}

\ssec{Descent for quasi-coherent sheaves}

  \begin{thm}\label{thm:descent}
    The presheaf of categories $\QCoh : \Sch^\op \to \Cat$ satisfies étale descent.
  \end{thm}

  In fact, we will see that it even satisfies descent for the \emph{fpqc}\footnote{%
    This stands for ``fid\`element plat et quasi-compact'' in French.
  } topology on $\Sch$, which is the Grothendieck topology generated by faithfully flat quasi-compact morphisms.

  \begin{cor}
    The presheaf of groupoids $\QCoh^\simeq : \Sch^\op \to \Grpd$ satisfies descent.
    In particular, it is a stack.
  \end{cor}
  \begin{proof}
    The functor $\Cat \to \Grpd$ sending $\sC \mapsto \sC^\simeq$ preserves limits.
    In fact, it is right adjoint to the inclusion $\Grpd \hook \Cat$.
    Thus the sheaf condition for $\QCoh^\simeq$ follows from that of $\QCoh$.
  \end{proof}

  To prove \thmref{thm:descent} we will apply a general descent criterion.

  \begin{defn}
    Given a cosimplicial diagram $X^\bullet : \bDelta \to \sC$ in an \inftyCat $\sC$, we refer to its limit as the \emph{totalization} of $X^\bullet$, and write
    \[ \Tot(X^\bullet) := \lim_{\bDelta} X^\bullet. \]
  \end{defn}

  \begin{defn}
    Let $\bDelta_+$ denote the category whose objects are the finite sets $[n] = \{0,1,\ldots,n\}$ for all $n\ge -1$, where $[-1] = \initial$ by convention, and whose morphisms are order-preserving maps.
    An \emph{augmented cosimplicial diagram} in any \inftyCat $\sC$ is a functor $X^\bullet : \bDelta_+ \to \sC$.
    We will depict $X^\bullet$ by the diagram
    \[
      X^{-1}
      \to X^0
      \rightrightarrows X^1
      \rightrightrightarrows X^2
      \rightrightrightrightarrows \cdots.
    \]
    We will say this is a \emph{limit diagram} if the induced morphism of cosimplicial diagrams
    \[ X^{-1}_\mrm{cst} \to X^\bullet|_{\bDelta} \]
    exhibits $X^{-1}$ as the limit (totalization) of $X^\bullet|_{\bDelta}$.
  \end{defn}

  \begin{defn}
    Let $\bDelta_{-\infty}$ denote the category whose objects are the finite sets $[n]$ for all $n\ge -1$, and whose morphisms $[m] \to [n]$ are order-preserving maps $[m] \cup \{-\infty\} \to [n] \cup \{-\infty\}$ which preserve $-\infty$.
    A \emph{splitting} of a cosimplicial diagram $X^\bullet : \bDelta \to \sC$ is an extension to a functor $X^\bullet : \bDelta_{-\infty} \to \sC$.
    A simplicial diagram $X^\bullet$ is \emph{split} if it admits a splitting.
    In this case, $X^\bullet|_{\bDelta_+}$ is a limit diagram, i.e. we have $\Tot(X^\bullet|_{\bDelta}) \simeq X^{-1}$.
  \end{defn}

  \begin{thm}[Descent criterion]\label{thm:descent criterion}
    Let $\sC^\bullet : \bDelta_+ \to \Catoo$ \footnote{%
      Here the \inftyCat $\Catoo$ of \inftyCats may be defined just as we defined the \inftyCat of \inftyGrpds, replacing Kan complexes with weak Kan complexes.
    } be an augmented cosimplicial diagram of \inftyCats, which we depict as follows:
    \begin{equation}\label{eq:preteressential}
      \sC^{-1}
      \xrightarrow{F} \sC^0
      \rightrightarrows \sC^1
      \rightrightrightarrows \sC^2
      \rightrightrightrightarrows \cdots.
    \end{equation}
    Suppose the following conditions hold:
    \begin{thmlist}
      \item
      The functor $F : \sC^{-1} \to \sC^0$ is conservative.

      \item
      For every morphism $\alpha : [m] \to [n]$ in $\bDelta_+$, let $\beta : [m+1] \to [n+1]$ denote the unique morphism which commutes with $\delta^0 : [m] \to [m+1]$ and $\delta^0 : [n] \to [n+1]$, and consider the commutative square
      \[\begin{tikzcd}
        \sC^m \ar{r}{d^0}\ar{d}{\alpha}
        & \sC^{m+1} \ar{d}{\beta}
        \\
        \sC^n \ar{r}{d^0}
        & \sC^{n+1}.
      \end{tikzcd}\]
      Then the horizontal arrows admit right adjoints $d^{0,R}$ which also commute with the vertical arrows; more precisely, the natural transformation
      \[
        \alpha \circ d^{0,R}
        \xrightarrow{\mrm{unit}} d^{0,R} \circ d^0 \circ \alpha \circ d^{0,R}
        \simeq d^{0,R} \circ \beta \circ d^0 \circ d^{0,R}
        \xrightarrow{\mrm{counit}} d^{0,R} \circ \beta
      \]
      is invertible.

      \item
      The functor $F : \sC^{-1} \to \sC^0$ preserves totalizations of $F$-split simplicial diagrams in $\sC^{-1}$.
      That is, for every cosimplicial diagram $X^\bullet$ in $\sC^{-1}$ whose image $F(X^\bullet)$ is split, the canonical map $F(\Tot(X^\bullet)) \to \Tot (F(X^\bullet))$ is invertible.
    \end{thmlist}
    Then the induced functor
    \[ \sC^{-1} \to \Tot(\sC^\bullet|_{\bDelta}) \]
    is an equivalence of \inftyCats.
    That is, \eqref{eq:preteressential} is a limit diagram.
  \end{thm}

  This result is a corollary of the monadicity theorem of Barr--Beck, generalized to \inftyCats by Lurie.
  See \cite[Cor.~4.7.5.3]{SAG}.

  Consider the functor
  \begin{equation*}
    \CRing \to \Cat, \quad R \mapsto \Mod_R
  \end{equation*}
  sending a commutative ring $R$ to the category $\Mod_R$ of $R$-modules, and a ring homomorphism $\phi : R \to S$ to the extension of scalars functor
  \[\phi^* := (-) \otimes_R S : \Mod_R \to \Mod_S.\]
  It can be defined using \thmref{thm:cartesian}, or alternatively is the restriction of $\QCoh : \Sch^\op \to \Cat$ to affine schemes (under the equivalence $\Aff^\op \simeq \CRing$).

  \begin{thm}\label{thm:module descent}
    The functor $R \mapsto \Mod_R$ satisfies descent for the flat topology, i.e. the Grothendieck topology on $\CRing$ generated by faithfully flat ring homomorphisms.
  \end{thm}
  \begin{proof}
    Let $\phi : R \to S$ be a faithfully flat ring homomorphism.
    We need to show that the augmented cosimplicial diagram
    \[
      \Mod_R
      \xrightarrow{\phi^*} \Mod_S
      \rightrightarrows \Mod_{S \otimes_R S}
      \rightrightrightarrows \cdots
    \]
    is a limit diagram, to which end we apply the criterion of \thmref{thm:descent criterion}:

    \begin{defnlist}
      \item
      Conservativity of the functor $\phi^*$ is a consequence of the assumption that $\phi$ is faithfully flat.
      
      \item
      For every morphism $\alpha : [m] \to [n]$ in $\bDelta_+$, consider the commutative square
      \[\begin{tikzcd}
        \Mod_{S^{\otimes m+1}} \ar{r}{d^{0,*}}\ar{d}{\alpha}
        & \Mod_{S^{\otimes m+2}} \ar{d}{\beta}
        \\
        \Mod_{S^{\otimes n+1}} \ar{r}{d^{0,*}}
        & \Mod_{S^{\otimes n+2}},
      \end{tikzcd}\]
      where the tensor powers are taken over $R$ and the horizontal arrows are extension of scalars along $d^0 : S^{\otimes m+1} \to S^{\otimes m+2}$.
      These functors are left adjoint to the restriction of scalars functors $d^0_*$.
      The vertical arrows are flat (as base changes of flat homomorphisms), and hence commute with $d^0_*$ by the flat base change formula.

      \item
      Let $M^\bullet$ be a cosimplicial diagram in $\Mod_R$ such that the cosimplicial diagram $\phi^*(M^\bullet)$ in $\Mod_S$ is split.
      The claim is that the $S$-module homomorphism
      \[ \phi^*(\Tot(M^\bullet)) \to \Tot(\phi^*(M^\bullet)) \]
      is invertible.
      Since $\Mod_R$ and $\Mod_S$ are $1$-categories, these totalizations can be computed as equalizers (\propref{prop:siris}).
      Since $\phi^*$ is an exact functor, it preserves equalizers.
    \end{defnlist}
    We also need to show that the functor $R \mapsto \Mod_R$ preserves finite products:
    \begin{itemize}
      \item The category of modules over the zero ring is equivalent to the trivial category.

      \item 
      For any pair of commutative rings $R_1$ and $R_2$, the category of modules over $R_1 \times R_2$ is equivalent to the product of categories $\Mod_{R_1} \times \Mod_{R_2}$.
    \end{itemize}
    It follows that $R \mapsto \Mod_R$ is a sheaf for the flat topology.
  \end{proof}

  The proof of \thmref{thm:descent} is almost the same, using the following lemma:

  \begin{lem}\label{lem:fpqc conservativity}
    Let $f : X \to Y$ be a morphism of schemes.
    If $f$ is faithfully flat and quasi-compact, then the inverse image functor $f^* : \QCoh(Y) \to \QCoh(X)$ is conservative.
  \end{lem}
  \begin{proof}
    Let $(V_j \to X)_j$ be a (possibly infinite) family of affine opens $V_j \sub Y$ covering $Y$.
    For each $j$, let $V'_j := V_j \fibprod_Y X = f^{-1}(V_j) \sub X$.
    Since $f$ is quasi-compact, $V'_j$ is quasi-compact.
    Thus there exists a finite family $(U_{i,j} \to V'_j)_i$ where $U_{i,j} \sub V'_j$ are affine opens covering $V'_j$.
    We have the commutative square
    \begin{equation*}
      \begin{tikzcd}
        \coprod_{i} U_{i,j} \ar{r}{f_j}\ar{d}
        & V_j \ar{d}
        \\
        X \ar{r}{f}
        & Y
      \end{tikzcd}
    \end{equation*}
    where $f_{j} : \coprod_{i} U_{i,j} \to V'_j \to V_j$ is a faithfully flat of affine schemes, since faithfully flat morphisms are stable under composition and base change.
    
    Now let $\phi : \sF \to \sG$ be a morphism of quasi-coherent sheaves on $Y$, whose inverse image $f^*(\phi) : f^*(\sF) \to f^*(\sG)$ is invertible.
    We need to show that $\phi$ is itself invertible.
    It will clearly suffice to show that each restriction $\phi|_{V_j}$ is invertible.
    Since $f_j$ is a faithfully flat morphism between affines, we know (by definition) that $f_j^*$ is conservative (since it corresponds to extension of scalars of a module along a faithfully flat ring homomorphism), so it will moreover suffice to show that $\phi$ is invertible after inverse image along $\coprod_i U_{i,j} \to V_j \sub Y$.
    But the latter factors through $f$, hence is invertible by assumption.
  \end{proof}

  \begin{proof}[Proof of \thmref{thm:descent}]
    Same as the proof of \thmref{thm:module descent}, except \lemref{lem:fpqc conservativity} is used to check the first condition of \thmref{thm:descent criterion}.
  \end{proof}

\ssec{Quasi-coherent sheaves on stacks}

  Since the full subcategory $\Aff \sub \Sch$ is a basis with respect to the Zariski topology, we have by \thmref{thm:RKE} the following further consequence of \thmref{thm:descent}:

  \begin{cor}
    The presheaf of categories $\QCoh : \Sch^\op \to \Cat$ is right Kan extended from its restriction to $\Aff$.
  \end{cor}

  Under the equivalence $\Aff^\op \simeq \CRing$, $\QCoh|_{\Aff}$ is identified with the functor $\CRing \to \Cat$ sending $R \mapsto \Mod_R$.
  Thus we have:

  \begin{cor}\label{cor:inweave}
    For every scheme $X$, there is a canonical equivalence of categories
    \[ \QCoh(X) \simeq \lim_{(R, x)} \Mod_R \]
    where the limit is taken over the category of pairs $(R, x)$ where $R\in\CRing$ and $x \in X(R)$ is an $R$-point, where morphisms $(R, x) \to (R', x')$ are ring homomorphisms $R \to R'$ such that $X(R) \to X(R')$ sends $x$ to $x'$.
  \end{cor}

  \begin{defn}
    Let $\sX$ be a stack.
    We define the category $\QCoh(\sX)$ of quasi-coherent sheaves on $\sX$ as the limit
    \[ \QCoh(\sX) := \lim_{(R,x)} \Mod_R \]
    over the category of pairs $(R, x)$ where $R\in\CRing$ and $x \in \sX(R)$ is an $R$-point.
    More precisely, we define 
    $$\QCoh : \Stk^\op \to \Cat$$
    as the right Kan extension of the presheaf $\Spec(R) \mapsto \Mod_R$ along the inclusion $\Aff \hook \Stk$, where $\Stk$ is the \inftyCat of stacks.
  \end{defn}

  \begin{rem}
    By definition, a quasi-coherent sheaf $\sF$ on a stack $\sX$ amounts to the following data:
    \begin{defnlist}
      \item
      For every commutative ring $R$ and every $R$-point $x \in \sX(R)$, an $R$-module $\sF(x)$.
      \item
      For every ring homomorphism $R \to R'$, $R$-point $x \in \sX(R)$, and $R'$-point $x' \in \sX(R')$ such that $\sX(R) \to \sX(R')$ sends $x \mapsto x'$, an $R'$-module isomorphism
      \[\alpha_{x,x'} : \sF(x) \otimes_R R' \simeq \sF(x').\]
    \end{defnlist}
    This data is subject to the following cocycle condition: for every pair of ring homomorphisms $R \to R'$ and $R' \to R''$, $R$-point $x \in \sX(R)$, $R'$-point $x' \in \sX(R')$, and $R''$-point $x'' \in \sX(R'')$ such that $\sX(R) \to \sX(R')$ sends $x \mapsto x'$ and $\sX(R') \to \sX(R'')$ sends $x' \mapsto x''$, there is a commutative diagram
    \[\begin{tikzcd}
      (\sF(x) \otimes_R R') \otimes_{R'} R'' \ar[equals]{d}\ar{r}{\alpha_{x,x'}}
      & \sF(x') \otimes_{R'} R''\ar{r}{\alpha_{x',x''}}
      & \sF(x'') \ar[equals]{d}
      \\
      \sF(x) \otimes_R R'' \ar{rr}{\alpha_{x,x''}}
      & 
      & \sF(x'').
    \end{tikzcd}\]
  \end{rem}


\section{Quotient stacks}
\label{sec:torsors}

  Let $G$ be a group.
  As we discussed in \secref{sec:overview}, a $G$-action on a set $X$ may be encoded by the diagram
  \[\begin{tikzcd}
    G\times X \ar[shift left=1ex]{r}\ar[swap,shift right=1ex]{r}
    & X \ar{l}
  \end{tikzcd}\]
  where the two rightwards arrows are the action map $(g,x)\mapsto g\cdot x$ and the projection $(g,x) \mapsto x$, and the leftward arrow is the section $x \mapsto (e, x)$ where $e\in G$ is the neutral element.
  The colimit of this diagram is the (set-theoretic) quotient $X/G$.
  If we instead take the colimit in the \inftyCat $\Grpd$, call it the \emph{quotient groupoid} $[X/G]$, then we may recover the above diagram by taking the fibred product of the canonical functor $X \to [X/G]$ along itself: there is a canonical isomorphism in $\Grpd$
  \[ X \fibprod_{[X/G]} X \simeq G \times X \]
  under which the two projections of $X \fibprod_{[X/G]} X$ are identified with the action and projection maps $G\times X \to X$, and the diagonal $X \to X \fibprod_{[X/G]} X$ is identified with the section $(e,\id) : X \to G\times X$.
  Informally speaking, the quotient groupoid $[X/G]$ remembers everything about the $G$-action on $X$.

  This story may be generalized to the case where $X$ is a groupoid or even \inftyGrpd.
  In that setting, the diagram above must be replaced by a simplicial diagram
  \[
    \cdots
    \rightrightrightrightarrows G \times G \times X
    \rightrightrightarrows G \times X
    \rightrightarrows X.
  \]
  We let $\sX = [X/G]$ denote the geometric realization, i.e. the colimit taken in \inftyGrpds.
  Then we can recover the above diagram as the \v{C}ech nerve of the quotient map $X \twoheadrightarrow \sX$.
  This leads to a bijective correspondence between \emph{groupoid objects} acting on $X$ (certain simplicial diagrams) and maps $X \twoheadrightarrow \sX$ which are surjective on $\pi_0$, where one direction is formation of the geometric realization, and the other is formation of the \v{C}ech nerve.

  Note that it is essential here to work with \inftyGrpds rather than ordinary groupoids: in the $(2,1)$-category $\Grpd$, not every groupoid object is effective, so the above correspondence breaks down.
  This is one way in which $\infty$-toposes, such as $\Grpdoo$, are better behaved than $1$-toposes and $2$-toposes like $\Set$ and $\Grpd$, respectively.

  For the same reason, when we sheafify this story, we will work with the \inftyCat $\Stk_\infty$ of \emph{$\infty$-stacks} rather than the $(2,1)$-category $\Stk$ of stacks.
  Since $\Stk_\infty$ is again an $\infty$-topos \cite[Def.~6.1.0.4]{HTT}, we obtain an equivalence between groupoid objects and effective epimorphisms in $\Stk_\infty$ (\thmref{thm:effepi}).

  A secondary reason to work with $\infty$-stacks is that quotient stacks force us out of the world of $1$-stacks in any case: the quotient $[U/G]$ of a stack $U$ by a group stack $G$ is in general a \emph{$2$-stack}, not a $1$-stack.
  If $G$ is a group $n$-stack, then $BG = [\pt/G]$ is an $(n+1)$-stack (\propref{prop:quotient truncation}), so at every stage we are pushed into higher stacks.
  See \examref{exam:B2Gm} for a concrete instance of this phenomenon.



  \ssec{$\infty$-Stacks}\label{ssec:oostacks}

    \begin{defn}\label{defn:ngrpd}
      An \inftyGrpd $X$ is \emph{$n$-truncated} (for $n \ge -1$) if $\pi_i(X, x) = 0$ for all $i > n$ and all base points $x$.
      An \emph{$n$-groupoid} is an $n$-truncated \inftyGrpd.
    \end{defn}

    \begin{rem}\label{rem:ngrpd-examples}
      A $0$-groupoid is a set, and a $1$-groupoid is an ordinary groupoid.
      The fully faithful functors $\Set \hook \Grpd \hook \Grpdoo$ of \secref{sec:oocats} identify sets with $0$-truncated \inftyGrpds and ordinary groupoids with $1$-truncated \inftyGrpds.
      More generally, for each $n \ge 0$ there is a fully faithful inclusion of the \inftyCat of $n$-groupoids into $\Grpdoo$, which admits a left adjoint $\tau_{\le n} : \Grpdoo \to \Grpdoo$ called the \emph{$n$-truncation} functor.
    \end{rem}

    \begin{defn}\label{defn:Stkoo}
      Let $\Stk_\infty$ denote the \inftyCat of \emph{$\infty$-stacks}, i.e., étale sheaves of \inftyGrpds $\sX : \Sch^\op \to \Grpdoo$.
    \end{defn}

    \begin{defn}\label{defn:nstack}
      An $\infty$-stack $\sX$ is \emph{$n$-truncated} if the \inftyGrpd $\sX(T)$ is $n$-truncated for all schemes $T$.
      An \emph{$n$-stack} is an $n$-truncated $\infty$-stack.
    \end{defn}

    \begin{rem}\label{rem:stacks-as-1stacks}
      A $0$-stack is an étale sheaf of sets, and a $1$-stack is a stack in the sense of \secref{sec:sheaves}.
      The fully faithful functor $\Grpd \hook \Grpdoo$ induces a fully faithful functor
      \[ \Stk \hook \Stk_\infty \]
      which identifies stacks with $1$-truncated $\infty$-stacks.
      Any scheme (regarded as a representable sheaf of sets) is $0$-truncated.
    \end{rem}

    \begin{exam}\label{exam:B2Gm}
      Let $H$ be an abelian group scheme over $S$.
      Since $H$ is commutative, the classifying stack $BH$ admits the structure of a group stack: the multiplication $BH \times BH \to BH$ sends a pair of $H$-torsors to their \emph{Baer sum} (contracted product over $H$).
      The classifying stack of the group stack $BH$ is a $2$-stack
      \[ B(BH) \simeq B^2 H \]
      which is not a $1$-stack unless $H$ is trivial.
      For $H = \bG_m$, the $2$-stack $B^2\bG_m$ parametrizes $\bG_m$-gerbes: a $T$-point of $B^2\bG_m$ is a $\bG_m$-gerbe over $T$.
      Its set of isomorphism classes $\pi_0(B^2\bG_m(T))$ is the cohomological Brauer group $\H^2_\et(T, \bG_m)$.
      More generally, $B^n H$ is an $n$-stack for all $n \ge 0$.
    \end{exam}

  \ssec{Groupoid objects and effective epimorphisms}

    \begin{defn}
      A \emph{groupoid object} in an \inftyCat $\sC$ is a simplicial diagram $U_\bullet : \bDelta^\op \to \sC$ such that for all $[n]\in\bDelta$ and all partitions $[n]= S \cup S'$ into subsets $S$ and $S'$ such that $S \cap S'$ consists of a single element $m \in [n]$, the square
      \[\begin{tikzcd}
        U_n = U_\bullet([n]) \ar{r}\ar{d}
        & U_\bullet(S) \ar{d}
        \\
        U_\bullet(S') \ar{r}
        & U_\bullet(\{m\}) = U_0
      \end{tikzcd}\]
      is cartesian.
      A \emph{morphism} of groupoid objects is a morphism of simplicial diagrams (i.e., a natural transformation).
    \end{defn}

    \begin{rem}
      Given a groupoid object $U_\bullet$, we think of $U_0 \in \sC$ as the object of ``objects'' of $U_\bullet$ and $U_1 \in \sC$ as the object of ``morphisms'' of $U_\bullet$.
      The above condition for ``ordered''\footnote{%
        i.e., partitions $[n]=S\cup S'$ such that $s \le s'$ for all $s\in S$ and $s' \in S'$
      } partitions such as $[2] = \{0,1\} \cup \{1,2\}$ means, informally speaking, that we may fill all inner horns: for example, the morphism $(d_2^0, d_2^2) : U_2 \to U_1 \fibprod_{U_0} U_1$ is invertible, where the two morphisms $U_1 \to U_0$ are ``source'' and ``target'', so that there is a ``composition'' morphism
      \[ U_1 \fibprod_{U_0} U_1 \simeq U_2 \xrightarrow{d_2^1} U_1. \]
      Similarly, for ``unordered'' partitions such as $[2] = \{0,1\} \cup \{0,2\}$ or $[2] = \{0,2\} \cup \{1,2\}$, the condition allows us to fill outer horns (and thus choose inverses to ``morphisms'' in $U_\bullet$).
    \end{rem}

    \begin{exam}
      Given a $1$-groupoid $\sG$, its nerve $\sG_\bullet := N(\sG) \in \SSet$ defines a groupoid object in the $1$-category $\Set$.
      In fact, the assignment $\sG \mapsto \sG_\bullet$ determines an equivalence from the $1$-category of $1$-groupoids to the \inftyCat of groupoid objects in $\Set$.
      We may also regard $\sG_\bullet$ as a groupoid object in $\Grpdoo$ (via the fully faithful functor $\Fun(\bDelta^\op, \Set) \hook \Fun(\bDelta^\op, \Grpdoo)$ induced by the embedding $\Set \hook \Grpdoo$ of sets as discrete \inftyGrpds).
    \end{exam}

    \begin{exam}\label{exam:group}
      Given a group $G$, we may consider the $1$-groupoid with a single object $\ast$ and endomorphism group $\End(\ast) = G$, with composition law defined by the multiplication law of $G$.
      The corresponding groupoid object $G_\bullet$ is a simplicial diagram of the form
      \[
        \cdots
        \rightrightrightrightarrows G \times G
        \rightrightrightarrows G
        \rightrightarrows \pt,
      \]
      where the face maps are induced by the group multiplication $m : G \times G \to G$ and the degeneracy maps are induced by the neutral element $e : \pt \to G$.
    \end{exam}

    \begin{exam}
      Let $f : U \to X$ be a morphism in $\sC$ such that the iterated fibred products $U\fibprod_X \cdots \fibprod_X U$ all exist in $\sC$.
      The \emph{\v{C}ech nerve} $U_\bullet$ of $f$ is the simplicial diagram
      \[
        \cdots
        \rightrightrightrightarrows U \fibprod_X U \fibprod_X U
        \rightrightrightarrows U \fibprod_X U
        \rightrightarrows U.
      \]
      More precisely, let $\bDelta_{\le 0,+} \sub \bDelta_+$ denote the full subcategory spanned by the objects $[0]$ and $[-1]$.
      Then the morphism $f$ determines a diagram $\bDelta_{\le 0,+}^\op \to \sC$, whose right Kan extension is an augmented simplicial diagram $U_\bullet^+ : \bDelta_+^\op \to \sC$.
      The \emph{\v{C}ech nerve} of $f$ is its restriction $U_\bullet := U_\bullet^+|_{\bDelta}$.
      This is always a groupoid object in $\sC$ (see \cite[Prop.~6.1.2.11]{HTT}).
    \end{exam}

    \begin{notat}
      Given a simplicial diagram $U_\bullet : \bDelta^\op \to \sC$, we will also refer to its colimit as its \emph{geometric realization}, and denote it by
      \[ \abs{U_\bullet} := \colim_{[n]\in\bDelta} U_n \in \sC \]
      when it exists.
    \end{notat}

    \begin{defn}
      A morphism $f : U \to X$ in $\sC$ is an \emph{effective epimorphism} if the iterated fibred products $U\fibprod_X \cdots \fibprod_X U$ all exist in $\sC$, and the augmented simplicial diagram
      \[
        \cdots
        \rightrightrightrightarrows U \fibprod_X U \fibprod_X U
        \rightrightrightarrows U \fibprod_X U
        \rightrightarrows U
        \to X
      \]
      is a colimit diagram, i.e., exhibits $X$ as the geometric realization of the \v{C}ech nerve $U_\bullet$.
    \end{defn}

    We now specialize to the \inftyCat $\Stk_\infty$ of $\infty$-stacks (Definition~\ref{defn:Stkoo}).
    The following two results hold more generally for sheaves of \inftyGrpds on any site (see \cite[Thm.~6.1.0.6, Cor.~6.2.3.5]{HTT}):

    \begin{thm}[Lurie]\label{thm:effepi}\leavevmode
      \begin{thmlist}
        \item
        Let $U_\bullet$ be a groupoid object in $\Stk_\infty$.
        Then the canonical morphism $U_0 \to \abs{U_\bullet}$ is an effective epimorphism in $\Stk_\infty$.

        \item
        The assignment sending $U_\bullet$ to the morphism $U_0 \twoheadrightarrow \abs{U_\bullet}$ determines an equivalence from the \inftyCat of groupoid objects in $\Stk_\infty$ to the \inftyCat of effective epimorphisms in $\Stk_\infty$.
        Its inverse is the functor forming the \v{C}ech nerve.
      \end{thmlist}
    \end{thm}

    It will also be useful to have the following more explicit description of effective epimorphisms:

    \begin{prop}\label{prop:local surj}
      A morphism $f : U \to X$ is an effective epimorphism in $\Stk_\infty$ if and only if it is \emph{étale-locally surjective}, i.e., for every commutative ring $R$ and every $R$-point $x \in X(R)$, there exists a faithfully flat étale homomorphism $R \to R'$ such that the image of $x$ in $X(R')$ belongs to the essential image of the functor $f(R') : U(R') \to X(R')$.
    \end{prop}

    \begin{exam}\label{exam:pointless}
      A smooth morphism of schemes $f : X \to Y$ is surjective if and only if it is an effective epimorphism.
    \end{exam}

  \ssec{Group actions}

    \begin{defn}
      A \emph{group object} in an \inftyCat $\sC$ is a groupoid object $G_\bullet$ such that $G_0 \in \sC$ is a terminal object.
      Given a group object $G_\bullet$, we will often abuse notation by saying that $G := G_1 \in \sC$ is a group object.
    \end{defn}

    \begin{exam}
      For any group $G$, the corresponding groupoid object $G_\bullet$ in $\Set$ (\examref{exam:group}) is a group object.
      It may also be regarded as a group object in $\Grpd$, which we still denote by $G_\bullet$.
      Similarly, if $G : \Sch^\op \to \mrm{Grp}$ is an (étale) sheaf of groups (such as represented by a group scheme), then it determines a group object in $\Shv(\Sch; \Set)$ and hence also in $\Stk$ (or $\Stk_\infty$).
    \end{exam}

    \begin{defn}
      Let $G_\bullet$ be a group object in an \inftyCat $\sC$.
      An \emph{action} of $G_\bullet$ on an object $U \in \sC$ is a groupoid object $U_\bullet$ with an isomorphism $U_0 \simeq U$ and a morphism of simplicial diagrams $U_\bullet \to G_\bullet$ such that for every morphism $[m] \to [n]$ in $\bDelta$ sending $0 \mapsto 0$, the square
      \[\begin{tikzcd}
        U_n \ar{r}\ar{d}
        & U_m \ar{d}
        \\
        G_n \ar{r}
        & G_m
      \end{tikzcd}\]
      is cartesian.

    \end{defn}

    \begin{rem}\label{rem:Wertherian}
      In the definition above, it is equivalent to require that for every $[n]\in\bDelta$, the square
      \[\begin{tikzcd}
        U_n \ar{r}\ar{d}
        & U_0 \ar{d}
        \\
        G_n \ar{r}
        & G_0,
      \end{tikzcd}\]
      where the horizontal arrows are induced by the morphism $[0] \to [n]$ in $\bDelta$ sending $0 \mapsto 0$, is cartesian.
      In particular, there is for every $[n]\in\bDelta$ a canonical isomorphism $U_n \simeq G_n \fibprod_{G_0} U_0 \simeq G^{\times n} \times U$.
      Thus the simplicial diagram $U_\bullet$ is of the form
      \[
        \cdots
        \rightrightrightrightarrows G \times G \times U
        \rightrightrightarrows G \times U
        \rightrightarrows U.
      \]
      Moreover, the face map $d^1 : G \times U \to U$ is the projection $(g,u)\mapsto u$, since the square
      \[\begin{tikzcd}
        U_1 \ar{r}{d^1}\ar{d}
        & U_0 \ar{d}
        \\
        G_1 \ar{r}{d^1}
        & G_0
      \end{tikzcd}\]
      is cartesian by assumption.
      Similarly, the degeneracy map $s^0 : U \to G \times U$ is the section $(e', \id)$ where $e' : U \to \pt \to G$ with $\pt$ the terminal object and $e=s^0 : \pt \to G$ the ``neutral element'' (part of the data of the group object $G_\bullet$).
    \end{rem}

    \begin{notat}
      Let $U_\bullet$ be an action of a group object $G_\bullet$ on $U\in\sC$.
      By \remref{rem:Wertherian}, the face map $d^0 : U_1 \to U_0$ determines a morphism
      \[ \mrm{act} : G \times U \to U \]
      which we call the \emph{action morphism}.
      Informally speaking, we may summarize the data of the action $U_\bullet$ as an action morphism $G \times U \to U$ which satisfies the analogues of the usual axioms of a group action \emph{up to coherent homotopy}.
    \end{notat}

    \begin{defn}
      We say that an action $U_\bullet$ is \emph{trivial} if for \emph{every} morphism $[m] \to [n]$ the square
      \[\begin{tikzcd}
        U_n \ar{r}\ar{d}
        & U_m \ar{d}
        \\
        G_n \ar{r}
        & G_m
      \end{tikzcd}\]
      is cartesian.
      In particular, the action morphism $G \times U \to U$ is in this case also the projection.
      For any object $U$, the trivial action on $U$ is defined by the groupoid $U^{\mrm{triv}}_\bullet := G_\bullet \fibprod_{G_0} U$.
    \end{defn}

    \begin{notat}
      We will often abuse language by saying ``let $G$ be a group object acting on $U$'' instead of ``let $U_\bullet$ be an action of a group object $G_\bullet$ on $U$''.
      If we need to speak of $U_\bullet$, we will refer to it as the \emph{action groupoid} (or \emph{quotient groupoid}).
      If $G$ acts on $U$ and $V$ (meaning that there are actions $U_\bullet$ and $V_\bullet$ with $U_0 \simeq U$ and $V_0 \simeq V$), a \emph{$G$-equivariant morphism} $f : U \to V$ is a morphism of simplicial diagrams $f_\bullet : U_\bullet \to V_\bullet$ fitting in a commutative diagram as follows.
      \[\begin{tikzcd}
        U_\bullet\ar{rr}{f_\bullet}\ar{rd}
        & & V_\bullet \ar{ld}
        \\
        & G_\bullet &
      \end{tikzcd}\]
    \end{notat}

  \ssec{Torsors}

    \begin{defn}\label{defn:torsor}
      Let $X$ be a stack and let $G$ be a group object in $\Stk$.
      A \emph{$G$-torsor} over $X$ is a $G$-equivariant morphism $\pi : U \to X$ where $U$ is a stack with an action of $G$ and $X$ is regarded with trivial $G$-action, satisfying the following conditions:
      \begin{defnlist}
        \item
        $\pi : U \to X$ is an effective epimorphism;
        \item\label{item:preincentive}
        the canonical morphism of simplicial diagrams $U_\bullet \to C_\bullet$ is invertible, where $U_\bullet$ is the action groupoid and $C_\bullet$ denotes the \v{C}ech nerve of $\pi : U \to X$.
      \end{defnlist}
    \end{defn}

    \begin{rem}
      It is immediate from the definition that if $\pi : U \to X$ is a $G$-torsor, then we may identify $X$ as the geometric realization of the action groupoid $U_\bullet$.
      That is, there is a colimit diagram
      \[
        \cdots
        \rightrightrightrightarrows G \times G \times U
        \rightrightrightarrows G \times U
        \rightrightarrows U
        \to X.
      \]
    \end{rem}

    \begin{rem}\label{rem:snapwood}
      Condition~\itemref{item:preincentive} in \defnref{defn:torsor} implies that the square
      \[\begin{tikzcd}
        U_1 \ar{r}{d^0}\ar{d}{d^1}
        & U_0 \ar{d}
        \\
        U_0 \ar{r}
        & X
      \end{tikzcd}\]
      is cartesian, i.e., that the canonical morphism
      $$(\mrm{act}, \mrm{pr}) : G \times U \to U \fibprod_X U$$
      is invertible.
      In fact, one can show that this is equivalent to \itemref{item:preincentive}.
    \end{rem}

    \begin{rem}
      We will also consider the variant of \defnref{defn:torsor} over a fixed base scheme $S$.
      This amounts to looking at presheaves on $\Sch_S$ instead of $\Sch$; for example, a stack $X : \Sch^\op \to \Grpd$ over $S$ is the same data as a sheaf $X : \Sch_S^\op \to \Grpd$.
      Below, we will leave $S$ implicit and simply denote it by ``$\pt$''.
    \end{rem}

    \begin{exam}
      The trivial $G$-torsor is the projection $\pi : G \times X \to X$.
      More precisely:
      \begin{defnlist}
        \item
        The $G$-action on $G \times X$ is given by the groupoid object $U_\bullet$:
        \[
          \cdots
          \rightrightrightrightarrows G\times G \times G \times X
          \rightrightrightarrows G \times G \times X
          \rightrightarrows G \times X
        \]
        where the face maps are $d^0 = m \times \id$ (where $m : G \times G \to G$ is the multiplication of the group) and $d^1 = \pr_1$.
        In other words, $G$ acts by multiplication on the first component and trivially on the second.

        \item
        The $G$-equivariant map $\pi : G \times X \to X$ is given by the morphism of simplicial diagrams $U_\bullet \to X^\mrm{triv} = G_\bullet \times X$:
        \[
          \begin{tikzcd}
            \cdots \ar{r}\ar[shift right=0.67ex]{r}\ar[swap,shift left=0.67ex]{r}
            & G\times G\times X \ar[shift left=0.5ex]{r}\ar[swap,shift right=0.5ex]{r}\ar{d}
            & G\times X\ar{d}
            \\
            \cdots \ar{r}\ar[shift right=0.67ex]{r}\ar[swap,shift left=0.67ex]{r}
            & G\times X \ar[shift left=0.5ex]{r}\ar[swap,shift right=0.5ex]{r}
            & X
          \end{tikzcd}
        \]
        which on level $n$ is given by projecting onto the last $n+1$ components of $G^{n+1} \times X$.

        \item
        Note that $G \to \pt$ is an effective epimorphism, since it admits a section $e : \pt \to G$.
        Hence its base change $\pi : G \times X \to X$ is an effective epimorphism.

        \item
        The square
        \[\begin{tikzcd}
          G \times G \times X \ar{r}{m\times \id}\ar{d}{\pr_{2,3}}
          & G \times X \ar{d}{\pr_2}
          \\
          G \times X \ar{r}{\pr_2}
          & X
        \end{tikzcd}\]
        is cartesian, hence $U_\bullet$ is identified with the \v{C}ech nerve of $\pi$ by \remref{rem:snapwood}.
        Indeed, this square is obtained from the cartesian square
        \[\begin{tikzcd}
          G \times G \ar{r}{\pr_1}\ar{d}{\pr_{2}}
          & G \ar{d}
          \\
          G \ar{r}
          & \pt
        \end{tikzcd}\]
        by precomposing with the inverse of the morphism $(m, \pr_2) : G \times G \to G \times G$ (which is invertible since $G$ is a group).
      \end{defnlist}
    \end{exam}

    \begin{defn}
      We say that a $G$-torsor $\pi : U \to X$ is \emph{trivial} if $U$ is $G$-equivariantly isomorphic to $G \times X$ over $X$.
      Note that this is equivalent to the existence of a section $s : X \to U$ (so that $\pi \circ s \simeq \id$), since such $s$ gives rise to an isomorphism
      \[ G \times X \xrightarrow{\id\times s} G \times U \xrightarrow{\mrm{act}} U. \]
    \end{defn}

  \ssec{Quotient stacks}

    \begin{defn}
      Let $G$ be a group stack over a base scheme $S$ (i.e., a group object in $\Stk$).
      Let $U$ be a stack over $S$ with $G$-action.
      The \emph{quotient stack} $[U/G]$ is defined as the geometric realization of the action groupoid $U_\bullet$, i.e.:
      \[ [U/G] := \abs{U_\bullet} = \colim_{[n]\in\bDelta} U_n, \]
      where the colimit is taken in $\Stk_\infty$.
      In particular, there is a colimit diagram
      \[
        \cdots
        \rightrightrightrightarrows G \times G \times U
        \rightrightrightarrows G \times U
        \rightrightarrows U
        \twoheadrightarrow [U/G]
      \]
      in $\Stk_\infty$.
    \end{defn}

    A priori, $[U/G]$ is only an $\infty$-stack.
    In fact, it is $1$-truncated (i.e., a stack) when $G$ is a group algebraic space.
    More generally, the construction $[U/G]$ makes sense for any group $\infty$-stack $G$ acting on an $\infty$-stack $U$, and we have:

    \begin{prop}\label{prop:quotient truncation}
      If $G$ is a $k$-truncated group $\infty$-stack and $U$ is an $n$-truncated $\infty$-stack, then $[U/G]$ is $\max(k+1,n)$-truncated.
      In particular:
      \begin{thmlist}
        \item\label{item:trunc-scheme}
        If $G$ is a group algebraic space and $U$ is a stack, then $[U/G]$ is a stack.
        \item\label{item:trunc-BG}
        If $G$ is a group $n$-stack, then $BG = [\pt/G]$ is an $(n+1)$-stack.
      \end{thmlist}
    \end{prop}
    \begin{proof}
      Since colimits of $\infty$-prestacks are computed sectionwise, the sectionwise geometric realization $\abs{U_\bullet}^\mrm{pre}$ satisfies $\abs{U_\bullet}^\mrm{pre}(T) = \abs{U_\bullet(T)}$ in $\Grpdoo$ for each $T$.
      The long exact sequence on homotopy groups for the fibre sequence $G(T) \to U(T) \to \abs{U_\bullet(T)}$ shows that $\abs{U_\bullet(T)}$ is $\max(k+1,n)$-truncated.
      Since the sheafification functor $L : \Fun(\Sch_S^\op, \Grpdoo) \to \Stk_\infty$ is left exact, it preserves $m$-truncated objects, so $[U/G] = L(\abs{U_\bullet}^\mrm{pre})$ is $\max(k+1,n)$-truncated.
    \end{proof}

    \begin{rem}
      By construction, the quotient morphism $p : U \twoheadrightarrow [U/G]$ is a $G$-torsor.
    \end{rem}

    \begin{rem}
      Let $\abs{U_\bullet}^\mrm{PreStk}$ denote the geometric realization taken in the $\infty$-category of $\infty$-prestacks over $S$ (i.e., in the \inftyCat $\Fun(\Sch_S^\op, \Grpdoo)$).
      Then $\abs{U_\bullet}^\mrm{PreStk}$ is given by the functor of points sending $T \in \Sch_S$ to the geometric realization $\abs{U_\bullet(T)}$ of the simplicial diagram of $\infty$-groupoids $U_\bullet(T)$.
      This is rarely a stack, and there is a canonical morphism of $\infty$-prestacks
      \[ \abs{U_\bullet}^\mrm{PreStk} \to \abs{U_\bullet} = [U/G] \]
      which exhibits $[U/G]$ as the sheafification.
    \end{rem}

    We can compute the functor of points of $[U/G]$ as follows:

    \begin{thm}\label{thm:quotient}
      For every $S$-scheme $T$, there is a (functorial) equivalence of groupoids between $[U/G](T)$ and the groupoid of diagrams
      \[ T \xleftarrow{\pi} Y \xrightarrow{f} U \]
      where $Y$ is an $S$-stack with $G$-action, $f : Y \to U$ is a $G$-equivariant $S$-morphism, and $\pi$ is a $G$-torsor over $T$.
      An isomorphism $(Y, f, \pi) \to (Y', f', \pi')$ is a $G$-equivariant isomorphism $\phi : Y \to Y'$ together with commutative squares expressing the compatibility of $\phi$ with $f$ and $f'$, and with $\pi$ and $\pi'$.
    \end{thm}
    \begin{proof}
      By the Yoneda lemma, $[U/G](T)$ is the groupoid of morphisms $T \to [U/G]$.
      Such a morphism gives rise by base change to a cartesian square
      \[\begin{tikzcd}
        Y \ar{r}{f}\ar[twoheadrightarrow]{d}{\pi}
        & U \ar[twoheadrightarrow]{d}{p}
        \\
        T \ar{r}
        & {[U/G]}
      \end{tikzcd}\]
      where the vertical arrows are $G$-torsors.
      The morphism $f$ is $G$-equivariant, since it lifts to a morphism of simplicial diagrams
      \[ U_\bullet \fibprod_{[U/G]} T \to U_\bullet \]
      by base change, where $U_\bullet$ is the action groupoid of the $G$-action on $U$.
      
      Conversely, given a $G$-torsor $\pi : Y \to T$, with $G$-action on $Y$ given by an action groupoid $Y_\bullet$, and a $G$-equivariant morphism $f : Y \to U$ given by a morphism of simplicial diagrams $Y_\bullet \to U_\bullet$, we may recover $T \to [U/G]$ as the geometric realization
      \[ T \simeq [Y/G] = \abs{Y_\bullet} \to \abs{U_\bullet} = [U/G] \]
      by \thmref{thm:effepi}.
    \end{proof}

    \begin{exam}
      The \emph{classifying stack} of $G$ is defined as the quotient stack
      \[ BG := B_S(G) := [S/G] \]
      with respect to the trivial $G$-action on $S$.
      By \thmref{thm:quotient}, every $G$-torsor $\pi : Y \to T$ for an $S$-scheme $T$ is classified by a morphism $T \to BG$ fitting into a cartesian square
      \[\begin{tikzcd}
        Y \ar{r}\ar[twoheadrightarrow]{d}{\pi}
        & S \ar[twoheadrightarrow]{d}{p}
        \\
        T \ar{r}
        & BG.
      \end{tikzcd}\]
      Informally speaking, the quotient morphism $p : S \twoheadrightarrow BG$ is the universal $G$-torsor.
    \end{exam}

    \begin{rem}
      In particular, the $G$-torsor $p : U \twoheadrightarrow [U/G]$ is classified by the morphism
      \[ [U/G] \to BG \]
      induced by the $G$-equivariant morphism $U \to S$.
      That is, there is a cartesian square
      \[\begin{tikzcd}
        U \ar{r}\ar[twoheadrightarrow]{d}
        & S \ar[twoheadrightarrow]{d}
        \\
        {[U/G]} \ar{r}
        & BG.
      \end{tikzcd}\]
    \end{rem}

    More generally:

    \begin{lem}
      Let $G$ be a group stack over a scheme $S$.
      Let $f : X \to Y$ be a $G$-equivariant morphism of stacks over $S$.
      Then the induced morphism $g : [X/G] \to [Y/G]$ fits in a cartesian square
      \[\begin{tikzcd}
        X \ar{r}{f}\ar[twoheadrightarrow]{d}
        & Y \ar[twoheadrightarrow]{d}
        \\
        {[X/G]} \ar{r}{g}
        & {[Y/G]}
      \end{tikzcd}\]
      where the vertical arrows are the quotient morphisms.
    \end{lem}
    \begin{proof}
      Exercise.
    \end{proof}

  \ssec{Examples of quotient stacks}

    \begin{exam}\label{exam:frame}
      Let $G = \GL_n$ be the general linear group scheme over $S$.
      Let $X$ be an $S$-scheme.
      For every locally free sheaf $\sE$ on $X$ of rank $n$, there exists a $\GL_n$-torsor
      \[ \pi : \on{Isom}_X(\sO_X^n, \sE) \to X \]
      whose total space represents the sheaf $\Sch_{/X}^\op \to \Set$ which sends $t : T \to X$ to the set of isomorphisms $\sO_T^n \simeq t^*\sE$.
      The assignment $\sE \mapsto \on{Isom}_X(\sO_X^n, \sE)$ determines an equivalence between the groupoid of locally free sheaves on $X$ of rank $n$ and the groupoid of $\GL_n$-torsors over $X$.
      In particular, morphisms $X \to \BGL_n$ may be identified with vector bundles over $X$ of rank $n$.
    \end{exam}

    \begin{exam}
      Consider the quotient stack $[\A^1/\bG_m]$ with respect to the weight $1$ scaling action of the multiplicative group scheme $\bG_m = \GL_1$ on the affine line $\A^1$.
      Then the groupoid of morphisms $T \to [\A^1/\bG_m]$ is equivalent to the groupoid of \emph{generalized divisors} on $T$, i.e., pairs $(\sL, s)$ where $\sL$ is a locally free sheaf of rank $1$ on $T$ and $s : \sO_T \to \sL$ is a section.
      Indeed, there is a cartesian square in $\Stk$ of the form
      \[\begin{tikzcd}
        \A^1 \ar{r}\ar{d}
        & \mrm{GDiv} \ar{d}
        \\
        S \ar[twoheadrightarrow]{r}
        & \mrm{B}\bG_m
      \end{tikzcd}\]
      where the right-hand vertical arrow is the forgetful map $(\sL,s)\mapsto \sL$ (or rather, the $\bG_m$-torsor corresponding to $\sL$); the lower horizontal arrow is the quotient map (which classifies the trivial line bundle); and the upper horizontal arrow sends a $T$-point of $\A^1$ given by $f \in \Gamma(T,\sO_T)$ to the generalized divisor $(\sO_T, f : \sO_T \to \sO_T) \in \on{GDiv}(T)$.
      Since the lower horizontal arrow is a $\bG_m$-torsor, it follows that the upper horizontal arrow is as well.
      In particular, it exhibits $\mrm{GDiv}$ as the quotient stack $[\A^1/\bG_m]$ as claimed.
    \end{exam}


\section{Algebraic spaces and stacks}
\label{sec:algebraic}

  \ssec{Algebraic spaces}

    \begin{defn}
      A morphism $f : X \to Y$ in an \inftyCat $\sC$ is a \emph{monomorphism} if the square
      \[\begin{tikzcd}
        X \ar[equals]{r}\ar[equals]{d}
          & X \ar{d}{f}
        \\
        X \ar{r}{f}
        & Y
      \end{tikzcd}\]
      is cartesian.
      In other words, the fibred product $X \fibprod_Y X$ exists and the diagonal $\Delta_f : X \to X \fibprod_Y X$ is invertible.
    \end{defn}

    \begin{exam}\label{exam:bronzy}
      A functor of groupoids $F : \sC \to \sD$ is a monomorphism in the \inftyCat $\Grpd$ if it is fully faithful.
      Equivalently, the induced map on sets of connected components $\pi_0(F) : \pi_0(\sC) \to \pi_0(\sD)$ is injective, and for every object $X \in \sC$ the map of automorphism groups $\Aut_\sC(X) \to \Aut_\sD(Y)$ is bijective.
      More generally, a functor $F : \sC \to \sD$ of \inftyGrpds is a monomorphism in the \inftyCat $\Grpdoo$ if and only if it is fully faithful, or equivalently injective on $\pi_0$ and bijective on $\pi_i(\sC, X) \to \pi_i(\sD, F(X))$ for all $i>0$ and objects $X \in \sC$.
    \end{exam}

    \begin{rem}\label{rem:fumatory}
      Let $X$ be an \inftyGrpd.
      If its diagonal $\Delta : X \to X \times X$ is a monomorphism, then $X$ is discrete (i.e., $X \simeq \pi_0(X)$).
      Indeed, the \inftyGrpd of isomorphisms between two objects $x,y \in X$ is given by the fibred product
      \[ \begin{tikzcd}
        \Maps_X(x, y) \ar{r}\ar{d}
        & X \ar{d}{\Delta}
        \\
        \pt \ar{r}{(x,y)}
        & X \times X
      \end{tikzcd} \]
      Hence $\Maps_X(x, y)$ is either empty or contractible (isomorphic in $\Grpdoo$ to $\pt$) for all $x,y$.
      The converse also holds: for $X$ discrete, the diagonal is always a monomorphism.
    \end{rem}

    \begin{exam}\label{exam:basket}
      A morphism of stacks $f : X \to Y$ is a monomorphism if and only if $f(T) : X(T) \to Y(T)$ is a monomorphism of groupoids for all schemes $T$.
      Thus if $X$ is a stack with monomorphic diagonal, $X$ takes values in sets (by \remref{rem:fumatory}).
    \end{exam}

    \begin{defn}
      A morphism $f : X \to Y$ of stacks is \emph{schematic} (or \emph{representable by schemes}) if for every scheme $V$ and every morphism $v : V \to Y$, the fibred product $X\fibprod_Y V$ is a scheme.
    \end{defn}

    \begin{lem}\label{lem:diagsch}
      Let $X$ be a stack.
      The following conditions are equivalent:
      \begin{thmlist}
        \item\label{item:diagsch/treespeeler}
        The diagonal $\Delta : X \to X \times X$ is schematic.
        
        \item\label{item:diagsch/doublet}
        For every pair of morphisms $u : U \to X$ and $v : V \to X$, where $U$ and $V$ are schemes, the fibred product $U \fibprod_X V$ is a scheme.
        
        \item\label{item:diagsch/greffotome}
        Every morphism $f : U \to X$, where $U$ is a scheme, is schematic.
      \end{thmlist}
    \end{lem}
    \begin{proof}
      We have \itemref{item:diagsch/treespeeler}~$\Rightarrow$~\itemref{item:diagsch/doublet} by the cartesian square
      \[\begin{tikzcd}
        U \fibprod_X V \ar{r}\ar{d}
        & X \ar{d}{\Delta}
        \\
        U \times V \ar{r}{u \times v}
        & X \times X,
      \end{tikzcd}\]
      and \itemref{item:diagsch/doublet}~=~\itemref{item:diagsch/greffotome} by definition.
      Suppose condition~\itemref{item:diagsch/greffotome} holds and let $f=(u,v) : U \to X \times X$ be a morphism.
      Then we have the cartesian square
      \[\begin{tikzcd}
        U \fibprod_{X \times X} X \ar{r}\ar{d}
        & U \fibprod_X U \ar{r}\ar{d}
        & X \ar{d}{\Delta}
        \\
        U \ar{r}{\Delta_U}
        & U \times U \ar{r}{u \times v}
        & X \times X,
      \end{tikzcd}\]
      where the lower horizontal composite arrow is $f$.
      Now $U\fibprod_X U$ is a scheme by \itemref{item:diagsch/greffotome}, so $U \fibprod_{X \times X} X$ is a scheme since $\Delta_U$ is schematic (as a morphism between schemes).
      As $U$ and $f$ vary, this shows \itemref{item:diagsch/treespeeler}.
    \end{proof}

    \begin{defn}\label{defn:algspace}
      Let $X$ be a stack.
      We say that $X$ is an \emph{algebraic space} if it satisfies the following conditions:
      \begin{defnlist}
        \item
        The diagonal $\Delta : X \to X \times X$ is schematic and a monomorphism.

        \item\label{item:wileful}
        There exists a scheme $U$ and a morphism $U \twoheadrightarrow X$ which is étale and surjective, i.e. for every scheme $V$ and every morphism $V \to X$ the base change $U \fibprod_X V \to V$ is an étale surjection (where $U \fibprod_X V$ is a scheme by \lemref{lem:diagsch}).
      \end{defnlist}
      Note that $X$ takes values in sets (\examref{exam:basket}), i.e., an algebraic space is equivalently a sheaf of sets $X : \Sch^\op \to \Set$ with schematic diagonal and an étale atlas as in \itemref{item:wileful}.
    \end{defn}

  \ssec{Algebraic stacks}

    \begin{defn}
      A morphism $f : X \to Y$ of stacks is \emph{representable} if for every scheme\footnote{equivalently, algebraic space} $V$ and every morphism $v : V \to Y$, the fibred product $X\fibprod_Y V$ is an algebraic space.
    \end{defn}

    \begin{defn}\label{defn:algstack}
      A stack $X$ is \emph{algebraic} if it satisfies the following conditions:
      \begin{defnlist}
        \item
        The diagonal $\Delta : X \to X \times X$ is representable.

        \item\label{item:algstack/Neophron}
        There exists a scheme $U$ and a morphism $U \twoheadrightarrow X$ which is smooth and surjective, i.e. for every scheme $V$ and every morphism $V \to X$ the base change $U \fibprod_X V \to V$ is a smooth surjection (where $U \fibprod_X V$ is an algebraic space by the analogue of \lemref{lem:diagsch}).
      \end{defnlist}
      We say that $X$ is \emph{Deligne--Mumford} if in \itemref{item:algstack/Neophron}, \emph{smooth} is replaced by \emph{étale}.
    \end{defn}

    \begin{rem}
      We call the smooth surjection $U \twoheadrightarrow X$ in \itemref{item:algstack/Neophron} an \emph{atlas} for $X$.
    \end{rem}

  \ssec{Algebraicity of quotient stacks}

    \begin{thm}\label{thm:quotientalg}
      Let $G$ be a smooth group algebraic space over a scheme $S$.
      Let $U$ be a stack with $G$-action.
      If $U$ is algebraic, then so is the quotient stack $[U/G]$.
    \end{thm}

    \begin{proof}
      Let us show that the diagonal of $[U/G]$ is representable.
      Given a scheme $T$ and a morphism $(t,t') : T \to [U/G] \times [U/G]$, the claim is that the fibred product $T \fibprod_{[U/G]} T \simeq T \fibprod_{[U/G]\times[U/G]} [U/G]$ is an algebraic space.
      Since the quotient morphism $p : U \twoheadrightarrow [U/G]$ is an effective epimorphism, there exists by \propref{prop:local surj} a scheme $T'$ and an étale surjection $T' \twoheadrightarrow T$ such that the composite $T' \twoheadrightarrow T \to [U/G]$ factors through $p : U \twoheadrightarrow [U/G]$.
      Thus the composite of the cartesian squares
      \[\begin{tikzcd}
        T' \fibprod_{[U/G]} T \ar{r}\ar{d}
        & T \fibprod_{[U/G]} T \ar{r}\ar{d}
        & T \ar{d}{t'}
        \\
        T' \ar[twoheadrightarrow]{r}
        & T \ar{r}{t}
        & {[U/G]}.
      \end{tikzcd}\]
      is identified with the composite of the cartesian squares
      \[\begin{tikzcd}
        T' \fibprod_{[U/G]} T \ar{r}\ar{d}
        & Y \ar[twoheadrightarrow]{r}\ar{d}
        & T \ar{d}{t'}
        \\
        T' \ar{r}
        & U \ar[twoheadrightarrow]{r}{p}
        & {[U/G]}
      \end{tikzcd}\]
      where $Y \twoheadrightarrow T$ is the $G$-torsor classified by $t'$.
      Since $U$ is algebraic (hence has representable diagonal) and $T'$ and $Y$ are algebraic spaces (the latter because $G$ is), it follows by the analogue of \lemref{lem:diagsch} that $T' \fibprod_{[U/G]} T$ is an algebraic space.
      It follows then that $T \fibprod_{[U/G]} T$ is also an algebraic space.

      It remains to show the existence of an atlas for $[U/G]$.
      Since $G \to S$ is a smooth surjection, it follows that the quotient morphism $p : U \twoheadrightarrow [U/G]$ is a smooth surjection.
      (Indeed, given a scheme $V$ and a morphism $v : V \to [V/G]$, the base change of $p$ along $v$ is the $G$-torsor classified by $v$, which is smooth and surjective because $G \to S$ is.)
      Since $U$ is algebraic, there exists a scheme $U'$ and a smooth surjection $U' \twoheadrightarrow U$.
      Then the composite $U' \twoheadrightarrow U \twoheadrightarrow [U/G]$ is smooth and surjective.
    \end{proof}

    \begin{rem}\label{rem:lithotypy}
      Since $p : U \twoheadrightarrow [U/G]$ is a $G$-torsor, the square
      \[\begin{tikzcd}
        G \times U \ar{d}{\mrm{act}}\ar{r}{\pr}
        & U \ar[twoheadrightarrow]{d}{p}
        \\
        U \ar[twoheadrightarrow]{r}{p}
        & {[U/G]}
      \end{tikzcd}\]
      is cartesian (\remref{rem:snapwood}).
      Equivalently, the diagonal of $[U/G]$ fits into the cartesian square
      \begin{equation*}
        \begin{tikzcd}
          G \times U \ar{r}{(\mrm{act}, \pr)}\ar{d}
          & U \times U \ar[twoheadrightarrow]{d}{p\times p}
          \\
          {[U/G]} \ar{r}{\Delta}
          & {[U/G]} \times {[U/G]}
        \end{tikzcd}
      \end{equation*}
    \end{rem}

  \ssec{Recognizing algebraic spaces and DM stacks}

    \begin{thm}\label{thm:subcarinate}
      Let $X$ be an algebraic stack.
      \begin{thmlist}
        \item\label{item:subcarinate/DM}
        $X$ is a Deligne--Mumford stack if and only if its diagonal is unramified.
        \item\label{item:subcarinate/asp}
        $X$ is an algebraic space if and only if its diagonal is a monomorphism.
      \end{thmlist}
    \end{thm}
    
    See \cite[Thm.~8.1, Cor.~8.1.1]{LaumonMoretBailly}.




    \begin{cor}
      Let $X$ be a stack whose diagonal $\Delta_X : X \to X \times X$ is a monomorphism (equivalently, $X : \Sch^\op \to \Grpd$ takes values in $\Set$).
      Then the following conditions are equivalent:
      \begin{thmlist}
        \item
        $X$ is an algebraic space.

        \item
        The diagonal $\Delta_X$ is representable, and there exists an algebraic space $U$ and an étale surjection $U \twoheadrightarrow X$.

        \item
        There exists an algebraic space $U$ and a representable étale surjection $U \twoheadrightarrow X$.
      \end{thmlist}
    \end{cor}

  \ssec{Free and proper actions}

    \begin{prop}\label{prop:backfold}
      Let $G$ be a smooth group algebraic space over a scheme $S$.
      Let $U$ be an algebraic space with $G$-action.
      Then the quotient stack $[U/G]$ is an algebraic space if and only if the action of $G$ on $U$ is \emph{free}, i.e., the morphism
      \[ (\mrm{act}, \pr) : G \times U \to U \times U \]
      is a monomorphism.
    \end{prop}
    \begin{proof}
      By \thmref{thm:quotientalg}, $[U/G]$ is an algebraic stack.
      Thus by \thmref{thm:subcarinate} it will suffice to show that its diagonal is a monomorphism.
      This may be tested after base change along a smooth surjection.
      By \remref{rem:lithotypy}, the base change along $p\times p : U \times U \twoheadrightarrow [U/G] \times [U/G]$ is precisely the morphism $(\mrm{act}, \pr)$ in the statement.
    \end{proof}

    \begin{warn}
      The analogue of \propref{prop:backfold} for schemes is false: there exist schemes $X$ of finite type over a field (which may be taken separated or even proper) and groups $G$ (which may be taken finite discrete) acting freely on $X$ such that the quotient stack $[X/G]$ is not representable by a scheme.
    \end{warn}

    Similarly, we have:

    \begin{prop}\label{prop:keyseater}
      Let $G$ be a smooth group algebraic space over a scheme $S$.
      Let $U$ be an algebraic space with $G$-action.
      Then the quotient stack $[U/G]$ is \emph{separated}, i.e. has proper diagonal, if and only if the action of $G$ on $U$ is \emph{proper}, i.e., the morphism
      \[ (\mrm{act}, \pr) : G \times U \to U \times U \]
      is proper.
    \end{prop}
    \begin{rem}
      In \propref{prop:keyseater}, suppose that $G$ is affine and $U$ has affine diagonal (e.g. $U$ is separated), so that the quotient stack $[U/G]$ has affine diagonal.
      It follows that the action is proper if and only if $[U/G]$ has \emph{finite} diagonal.
      If $S$ has all residue fields of characteristic zero, then this implies that $[U/G]$ has unramified diagonal, hence is Deligne--Mumford by \thmref{thm:subcarinate}.
    \end{rem}

  \ssec{Stabilizers}

    \begin{defn}
      Let $X$ be a stack and $x : \Spec(R) \to X$ an $R$-valued point for some $R\in\CRing$.
      The \emph{stabilizer} of $X$ at the point $x$ is a group stack $\St_X(x)$ over $\Spec(R)$ defined by the cartesian square:
      \[\begin{tikzcd}
        \St_X(x) \ar{r}\ar{d}
        & \Spec(R) \ar{d}{x}
        \\
        \Spec(R) \ar{r}{x}
        & X.
      \end{tikzcd}\]
      Equivalently, it fits in the cartesian square
      \[\begin{tikzcd}
        \St_X(x) \ar{r}\ar{d}
        & X \ar{d}{\Delta_X}
        \\
        \Spec(R)\times\Spec(R) \ar{r}{x\times x}
        & X \times X.
      \end{tikzcd}\]
      Its groupoid of $R'$-points, for an $R$-algebra $R'$, is thus the group object in $\Grpd$ of automorphisms of (the image of) $x$ in the groupoid $X(R')$.
    \end{defn}

    \begin{rem}
      If $X$ has representable (resp. schematic, affine) diagonal, then its stabilizers are all algebraic spaces (resp. schemes, affine schemes).
    \end{rem}

    \begin{exam}
      Let $U$ be an algebraic space with an action of a group algebraic space $G$.
      The stabilizers of the quotient stack $X=[U/G]$ are precisely the stabilizers of the $G$-action.
      More precisely, given an $R$-point $u \in U(R)$ (where $R\in\CRing$), it follows from \remref{rem:lithotypy} that the group algebraic space $\St_{[U/G]}(u)$ fits in the cartesian square
      \[\begin{tikzcd}
        \St_{[U/G]}(u)\ar{r}\ar{d}
        & G \times \Spec(R) \ar{d}{\mrm{act}(-,u)}
        \\
        \Spec(R) \ar{r}{u}
        & U.
      \end{tikzcd}\]
      where $\mrm{act}(-,u)$ is the composite of $\id\times u : G\times\Spec(R) \to G\times U$ and $\mrm{act} : G\times U \to U$.
    \end{exam}

    By \thmref{thm:subcarinate} we have:

    \begin{cor}\label{cor:nonabsorptive}
      Let $X$ be an algebraic stack.
      Then $X$ is an algebraic space if and only if for every $R\in\CRing$ and every $R$-valued point $x : \Spec(R) \to X$ the stabilizer $\St_X(x)$ is trivial.
    \end{cor}

    \begin{rem}
      In \corref{cor:nonabsorptive}, it turns out that it is moreover sufficient to require that the stabilizer at every \emph{geometric} point is trivial.
      See \cite[Thm.~2.2.5(1)]{ConradDrinfeldLevel}.
    \end{rem}


\section{The resolution property}
\label{sec:resolution}

  \ssec{The resolution property}

    \begin{defn}
      Let $\sF$ be a quasi-coherent sheaf on an algebraic stack $X$.
      We say that $\sF$ is \emph{coherent} if there exists a scheme $U$ and a smooth surjection $p : U \twoheadrightarrow X$ such that $p^*(\sF)$ is a coherent sheaf on $U$.
      We say that $\sF$ is \emph{locally free} of rank $n$ if there exists a scheme $U$, a smooth surjection $p : U \twoheadrightarrow X$, and an isomorphism $p^*\sF \simeq \sO_U^{\oplus n}$.
    \end{defn}

    \begin{defn}
      We say that an algebraic stack $X$ has the \emph{resolution property} if there exists a collection $(\sG_\alpha)_\alpha$ of locally free sheaves of finite rank on $X$ such that for every quasi-coherent sheaf $\sF$ on $X$ there exists a surjection
      \[
        \bigoplus_\alpha \sG_\alpha \twoheadrightarrow \sF
      \]
      of quasi-coherent sheaves on $X$.
    \end{defn}

    \begin{rem}\label{rem:forehead}
      If $X$ is quasi-compact and quasi-separated, then it admits the resolution property if and only if for every finite type quasi-coherent sheaf $\sF$, there exists a finite locally free sheaf $\sE$ and a surjection $\sE \twoheadrightarrow \sF$.
      This follows from the fact that every quasi-coherent sheaf on $X$ can be written as a filtered colimit of quasi-coherent sheaves of finite type (see \cite{RydhCompleteness}).
    \end{rem}

    \begin{exam}
      Let $X$ be a scheme and consider the following conditions:
      \begin{defnlist}
        \item $X$ is quasi-affine.
        \item $X$ is quasi-projective.
        \item $X$ admits an ample family of line bundles.
        \item $X$ has the resolution property.
      \end{defnlist}
      We obviously have (i) $\Rightarrow$ (ii) $\Rightarrow$ (iii) $\Rightarrow$ (iv).
    \end{exam}
    \begin{rem}
      It is possible to show that the resolution property implies that $X$ has affine diagonal (for example, this follows from \thmref{thm:cosmographical} below, or see \spcite{0F8C}).
      Conversely, there is apparently no known example of a scheme $X$ of finite type over a field which has affine diagonal but does not admit the resolution property.
    \end{rem}

  \ssec{Global quotient stacks}

    It turns out that, under very mild assumptions on $X$, the resolution property is equivalent to the existence of a global quotient presentation.

    \begin{defn}
      Let $X$ be an algebraic stack.
      We say that $X$ is \emph{quasi-compact} if there exists a quasi-compact scheme $U$ and a smooth surjection $U \twoheadrightarrow X$.
      We say that $X$ is \emph{quasi-separated} if it has quasi-compact quasi-separated diagonal.
    \end{defn}

    \begin{defn}
      Let $X$ be an algebraic stack.
      We say that $X$ has \emph{affine stabilizers} if for every geometric point $x$ of $X$, the stabilizer $\St_X(x)$ is affine.
      For example, if $X$ has affine diagonal, then it has affine stabilizers.
    \end{defn}

    \begin{defprop}\label{defprop:global}
      We say that $X$ is a \emph{global quotient stack} if satisfies the following equivalent conditions:
      \begin{thmlist}
        \item
        There exists a quasi-affine scheme $U$ with an action of $\GL_n$ (for some $n\ge 0$), and an isomorphism $X \simeq [U/\GL_n]$.

        \item
        There exists a $\GL_n$-torsor (for some $n\ge 0$) $U \twoheadrightarrow X$ where $U$ is a quasi-affine scheme.
        
        \item
        There exists a quasi-affine morphism $X \to \BGL_n$ (for some $n\ge 0$).
      \end{thmlist}
    \end{defprop}

    \begin{rem}
      Note that any global quotient stack is quasi-compact quasi-separated with affine stabilizers.
    \end{rem}

    \begin{thm}[Totaro, Gross]\label{thm:cosmographical}
      Let $X$ be an algebraic stack.
      Then $X$ has the resolution property if and only if it is a global quotient stack.
    \end{thm}

  \ssec{Quasi-coherent sheaves on quotient stacks}

    Let us begin by recording the following tautological fact:

    \begin{prop}\label{prop:checkage}
      Let $F : \Sch^\op \to \sV$ be an étale sheaf where $\sV$ is a small \inftyCat admitting limits, and denote by $F^+$ its right Kan extension along the Yoneda embedding $\Sch \hook \Stk_\infty = \Shv(\Sch; \Grpdoo)$.
      Then for every effective epimorphism $f : X \to Y$ in $\Stk_\infty$, the canonical morphism
      \[ F(Y) \to \Tot(F(X_\bullet)) \]
      is invertible where $X_\bullet$ denotes the \v{C}ech nerve of $f$.
    \end{prop}
    \begin{proof}
      If $f : X \to Y$ is an effective epimorphism, then the canonical morphism $\abs{X_\bullet} \to Y$ is invertible.
      Recall that $F^+$ is by definition the unique functor $\Stk_\infty^\op \to \sV$ extending $F$ which preserves limits, i.e., sends colimits in $\Stk_\infty$ to limits in $\sV$.
      The claim follows.
    \end{proof}

    \begin{cor}
      Let $X$ be an algebraic stack.
      Then for every scheme $U$ and smooth surjection $p : U \twoheadrightarrow X$, the canonical functor
      \[ \QCoh(X) \to \Tot(\QCoh(U_\bullet)) \]
      is an equivalence, where $U_\bullet$ is the \v{C}ech nerve of $p$.
    \end{cor}

    \begin{cor}\label{cor:uneagerness}
      Let $G$ be a group stack over a base scheme $S$ acting on a stack $X$ over $S$.
      Then the canonical functor
      \[ \QCoh([X/G]) \to \Tot(\QCoh(X_\bullet)) \]
      is an equivalence, where $X_\bullet = [\cdots \rightrightrightarrows G \times X \rightrightarrows X]$ is the action groupoid.
    \end{cor}

    \begin{defn}\label{defn:slowly}
      Let $G$ be a group stack over a scheme $S$, and let $X$ be a stack over $S$ with $G$-action.
      We define
      \[
        \QCoh^G(X) := \Tot(\QCoh(X_\bullet)).
      \]
      Note that its objects, which we call \emph{$G$-equivariant quasi-coherent sheaves} on $X$, are quasi-coherent sheaves $\sF$ on $X$ together with a (specified) isomorphism $\mrm{act}^*\sF \simeq \pr^*\sF$ on $G \times X$, satisfying a cocycle condition on $G\times G\times X$.
      Thus \corref{cor:uneagerness} says that quasi-coherent sheaves on $[X/G]$ are $G$-equivariant quasi-coherent sheaves on $X$.
    \end{defn}



    \begin{thm}[Thomason]
      Every global quotient stack admits the resolution property.
    \end{thm}
    \begin{proof}
      Let $U$ be a quasi-affine scheme with an action of $G=\GL_n$.
      Since the structure sheaf $\sO_U$ is obviously $G$-equivariant, it defines a locally free sheaf on $[U/G]$.
      Moreover, $\sO_U$ is ample (since $U$ is quasi-affine).
      In this situation, Thomason proved that every $G$-equivariant coherent sheaf on $U$ admits a surjection from a $G$-equivariant locally free sheaf (see \cite[Lem.~2.4, Lem.~2.6]{ThomasonResolution}).
      This immediately implies that $[U/G]$ has the resolution property, translating via \defnref{defn:slowly} and using \remref{rem:forehead}.
    \end{proof}

  \ssec{A quasi-affineness criterion}

    \begin{defn}
      Let $f : X \to Y$ be a morphism of algebraic stacks.
      Let $(\sG_\alpha)_\alpha$ be a collection of coherent sheaves on $X$.
      \begin{defnlist}
        \item
        We say that $(\sG_\alpha)_\alpha$ is \emph{generating over $Y$}, or \emph{$f$-generating}, if for all $\sE \in \QCoh(X)$ there exists an $\sF \in \QCoh(Y)$ and a surjection
        \[
          \bigoplus_\alpha \sG_\alpha \otimes f^*(\sF)
          \twoheadrightarrow \sE.
        \]

        \item
        We say that $(\sG_\alpha)_\alpha$ is \emph{universally generating over $Y$} (or \emph{universally $f$-generating}) if the above property holds after base change to any affine scheme (or equivalently, after base change to any algebraic stack).
      \end{defnlist}
    \end{defn}

    \begin{rem}
      If $Y$ is affine, then $(\sG_\alpha)_\alpha \sub \QCoh(X)$ is generating over $Y$ if and only if for every $\sF \in \QCoh(X)$ there exists a surjection
      \[ \bigoplus_\alpha \sG_\alpha \twoheadrightarrow \sF. \]
    \end{rem}

    \begin{thm}[Gross]\label{thm:qaffgross}
      Let $f : X \to Y$ be a morphism of qcqs algebraic stacks.
      Then $f$ is quasi-affine if and only if it satisfies the following properties:
      \begin{thmlist}
        \item
        The morphism $f$ has affine stabilizers.
        That is, for every affine scheme $V$ and every morphism $v : V \to Y$, the fibre $X\fibprod_Y V$ has affine stabilizers (at geometric points).

        \item
        The collection $(\sO_X)$ is universally generating over $Y$.
      \end{thmlist}
    \end{thm}

    See \cite[Prop.~3.1]{Gross}.

  \ssec{Proof of \thmref{thm:cosmographical}}
    
    



    Let $X$ be an algebraic stack with the resolution property.
    Let $(\sG_\alpha)_\alpha$ be a collection of locally free sheaves of ranks $n_\alpha$ on $X$, and denote by $(F_\alpha)_\alpha$ the corresponding $\GL_{n_\alpha}$-torsors (\examref{exam:frame}).
    For every finite subset of indices $I$, the product (fibred over $X$)
    \[ F_I = \prod_{\alpha\in I} F_\alpha \]
    defines a $G_I := \prod_{\alpha\in I} \GL_{n_\alpha}$-torsor over $X$, over which $F_\alpha$ trivializes for each $\alpha\in I$.
    Note that $F_I \to X$ is an affine morphism (since $G_I$ is affine).
    The limit $F = \lim_I F_I$ over the cofiltered system of finite subsets $I$ is affine over $X$ and in particular algebraic, since the transition maps are affine.

    By \thmref{thm:qaffgross} we deduce that $F$ is a quasi-affine scheme, i.e., that $F \to \Spec(\Z)$ is quasi-affine:
    \begin{enumerate}
      \item 
      $F$ has affine stabilizers.
      Indeed, it is affine over $X$, which has affine stabilizers.

      \item
      $\sO_F$ is universally generating (over $\Spec(\Z)$).
      Indeed, since $F \to X$ is affine, the fact that $(\sG_\alpha)_\alpha \sub \QCoh(X)$ is universally generating over $Y$ implies that $(\sG_\alpha|_F)_\alpha \sub \QCoh(F)$ is universally generating over $\Spec(\Z)$.
      But $\sG_\alpha|_F \simeq \sO_F$ for every $\alpha$ by construction of $F$.
    \end{enumerate}
    Since $\Spec(\Z)$ is quasi-compact, it follows that $I$ may be chosen sufficiently large such that $F_I \to \Spec(\Z)$ is already quasi-affine (see \cite[Thm.~C]{RydhApprox}).

    The $G_I$-torsor $F_I \to \Spec(\Z)$ is classified by a morphism $X \to BG_I$, which is quasi-affine because its base change along the quotient map $\Spec(\Z) \twoheadrightarrow BG_I$ (which is $F_I \to \Spec(\Z)$) is quasi-affine.
    The diagonal embedding of $G_I$ into $\GL_n$, where $n = \sum_\alpha n_\alpha$, has quotient $\GL_n/G_I$ representable by an affine scheme (the ``Stiefel manifold'').
    The induced morphism of classifying stacks $BG_I \to \BGL_n$ is therefore quasi-affine, since its base change along $\Spec(\Z) \twoheadrightarrow \BGL_n$ is $\GL_n/G_I$.
    We have thus constructed a quasi-affine morphism $X \to \BGL_n$, hence $X$ is a global quotient by \propref{defprop:global}.


\section{Derived categories}

  \ssec{Algebraic categories}

    \begin{defn}\label{defn:algebraic}
      Let $\sC$ be a category.
      We say $\sC$ is \emph{algebraic} if there exists an essentially small full subcategory $\sF_\sC \sub \sC$ closed under finite coproducts, and an equivalence
      \[
        \Fun_\Pi(\sF_\sC^\op, \Set) \to \sC,
      \]
      which restricts along the Yoneda embedding $\sF_\sC \hook \Fun_\Pi(\sF_\sC^\op, \Set)$ to the inclusion $\sF_\sC \hook \sC$.
      The subscript $_\Pi$ indicates the full subcategory of finite product-preserving functors.
    \end{defn}

    \begin{defn}
      Let $\sC$ be a category.
      \begin{defnlist}
        \item
        An object $X \in \sC$ is \emph{compact} if and only if $\Hom_\sC(X, -) : \sC \to \Set$ preserves filtered colimits.
        \item
        A compact object $X \in \sC$ is \emph{projective} if and only if $\Hom_\sC(X, -) : \sC \to \Set$ preserves reflexive coequalizers.
      \end{defnlist}
      The category $\sC$ is algebraic if and only if the condition of \defnref{defn:algebraic} holds with $\sF_\sC$ the full subcategory of compact projective objects.
      When not otherwise specified, $\sF_\sC$ will denote this subcategory by default.
    \end{defn}

    \begin{exam}
      The category of sets is algebraic: we have
      \[ \Set \simeq \Fun_\Pi(\Fin^\op, \Set), \]
      where $\Fin \sub \Set$ is the full subcategory of finite sets.
      Certainly any set $X$ represents a product-preserving functor $\Fin^\op \to \Set$ sending $Y \mapsto \Hom(Y, X)$.
      This gives a functor $\Set \to \Fun_\Pi(\Fin^\op, \Set)$.
      Conversely, given a product-preserving functor $F : \Fin^\op \to \Set$, we get the following data:
      \begin{defnlist}
        \item
        Sets $F_0 = F(\initial)$, $F_1 = F(\{1\})$, $F_2 = F(\{1,2\})$, \textellipsis.

        \item
        For every map of finite sets $\{1, 2, \ldots, n\} \to \{1, 2, \ldots, m\}$, an induced map $F_m \to F_n$.

        \item
        Canonical isomorphisms $F_0 \simeq \pt := \{\ast\}$, and $F_n \simeq (F_1)^{\times n}$ for all $n$.
      \end{defnlist}
      The assignment $F \mapsto F_1$ defines a functor from the RHS to $\Set$, and
      it is easy to check that these two functors define an equivalence of categories.
      Moreover, the compact projective objects in $\Set$ are exactly the finite sets.
    \end{exam}

    \begin{exam}
      The category $\Ab$ of abelian groups is algebraic:
      \begin{equation}\label{eq:07ugs1}
        \Ab \simeq \Fun_\Pi(\sF_\Ab^\op, \Set)
      \end{equation}
      where $\sF_\Ab$ is the category of finitely generated free abelian groups.
      Moreover, the compact projective objects are exactly those that belong to $\sF_\Ab$.

      Note that $\sF_\Ab$ is equivalent to the category whose objects are natural numbers, morphisms are matrices of integers
      \[
        \Hom(n, m)
        = \Hom_\Ab(\bZ^{\oplus n}, \bZ^{\oplus m})
        = \Mat_{m\times n}(\bZ),
      \]
      and the composition rule is given by matrix multiplication.

      An object of the RHS of \eqref{eq:07ugs1} amounts to the data of sets $F_0, F_1, F_2, \ldots$ with isomorphisms $F_n \simeq (F_1)^{\times n}$ for all $n$, and for every matrix $\phi \in \Mat_{m \times n}(\bZ)$, a map $F_\phi : F_m \to F_n$.
      This data is subject to identities imposed by functoriality.
      From this we can derive an \emph{underlying set} $M := F_1 \in \Set$ equipped with various \emph{operations} encoded by matrices.
      For example, every $(n\times 1)$-matrix
      \[ \phi = \begin{pmatrix}a_1\\\ldots\\a_n\end{pmatrix} \]
      encodes the operation $F_\phi : M^{\times n} \to M$ of ``forming the linear combination with coefficients $a_i$'':
      \[
        (x_1, ..., x_n) \mapsto \sum_i a_i x_i.
      \]
      In particular:
      \begin{defnlist}
        \item
        There is a (commutative) \emph{addition map} $M \times M \to M$, $(x_1,x_2) \mapsto x_1+x_2$, encoded by the matrix
        \[\begin{pmatrix}1\\1\end{pmatrix}.\]

        \item
        There is a zero element $0 \in M$, encoded as the map $0 : \pt \to M$ coming from the empty $(0\times 1)$-matrix.

        \item
        There is an \emph{additive inverse map} $M \to M$, $x_1 \mapsto -x_1$, encoded by the $(1\times 1)$-matrix $\begin{pmatrix} -1 \end{pmatrix}$.
      \end{defnlist}
      Essentially, this is just a very redundant way of specifying an abelian group.
    \end{exam}

    \begin{exam}\label{exam:Mod}
      More generally, for every commutative ring $R$ the category $\Mod_R$ of $R$-modules is algebraic:
      \begin{equation*}
        \Mod_R \simeq \Fun_\Pi(\sF_R^\op, \Set)
      \end{equation*}
      where $\sF_R$ is the category of finitely generated free $R$-modules.
      The compact projective objects are the finitely generated \emph{projective} $R$-modules (and not, in general, just the finitely generated free ones).
    \end{exam}

    If $\sC$ is algebraic, then it is the free completion of $\sF_\sC$ by filtered colimits and reflexive coequalizers.

    \begin{prop}\label{prop:0adgus}
      Let $\sC$ be an algebraic category.
      For every category $\sD$ admitting filtered colimits and reflexive coequalizers, the canonical functor
      \[ \Fun'(\sC, \sD) \to \Fun(F_\sC, \sD), \]
      from the category of functors $\sC \to \sD$ that preserve filtered colimits and reflexive coequalizers, is an equivalence.
    \end{prop}

    \begin{remind}
      A \emph{reflexive pair} in a category $\sC$ is a diagram
      \[ \begin{tikzcd}
        x \ar[bend left=30]{r}{f}\ar[bend right=30, swap]{r}{g}
        & y \ar[swap]{l}{s},
      \end{tikzcd} \]
      where $f \circ s = g\circ s$, i.e., $s$ is a common section of $f$ and $g$.
      \emph{Reflexive coequalizers} are colimits of reflexive pairs.
      Note that reflexive pairs are nothing else than diagrams indexed by the full subcategory $\bDelta_{\le 1}^\op \sub \bDelta^\op$ spanned by $[0]$ and $[1]$.
    \end{remind}

  \ssec{(Connective) nonabelian derived categories}

    In an algebraic category $\sC$, every object is built out of objects of $\sF_\sC$ under filtered colimits and reflexive coequalizers.
    The idea of Quillen's nonabelian derived category is to build an \inftyCat out of $\sF_\sC$ by replacing reflexive coequalizers by geometric realizations.
    Following \cite{CesnaviciusScholze}, we will refer to this construction as the \emph{animation} of $\sC$.

    \begin{defn}\label{defn:animation}
      Let $\sC$ be an algebraic category.
      Given an \inftyCat $\sC^+$ admitting colimits, we say that a fully faithful functor $\sF_\sC \hook \sC^+$ exhibits $\sC^+$ as an \emph{animation} of $\sC$ if it satisfies the following universal property: for every \inftyCat $\sD$ admitting filtered colimits and geometric realizations, restriction induces an equivalence
      \[ \Fun'(\sC^+, \sD) \to \Fun(\sF_\sC, \sD) \]
      from the \inftyCat of functors $\sC^+ \to \sD$ that preserve filtered colimits and geometric realizations.
      Informally speaking, $\sC^+$ is freely generated by $\sF_\sC$ under filtered colimits and geometric realizations.

      We will denote by $\Anim(\sC)$ the animation of an algebraic category $\sC$.
      An \emph{animated object} of $\sC$ is an object of $\Anim(\sC)$.
    \end{defn}

    \begin{thm}[Quillen, Lurie]\label{thm:model}\leavevmode
      \begin{thmlist}
        \item
        The functor $\Fin \hook \Set \hook \Grpdoo$ (\remref{rem:07pgsu}) exhibits $\Grpdoo$ as an animation of $\Set$.

        \item
        For any algebraic category $\sC$, the Yoneda embedding
        \[ \sF_\sC \hook \Fun_\Pi(\sF_\sC^\op, \Grpdoo) \]
        exhibits its target as an animation of $\sC$.
      \end{thmlist}
    \end{thm}
    \begin{proof}
      See \cite[Prop.~5.5.8.15]{HTT}.
    \end{proof}

    \begin{notat}
      We denote by $\Anima$ the animation $\Anim(\Set) \simeq \Grpdoo$ of the category of sets, and will refer to its objects as \emph{anima}.
      (In other words, ``animum'' is a synonym for \inftyGrpd, but reflects that we want to think of anima as ``derived sets'' rather than a generalization of groupoids.)
    \end{notat}



  \ssec{Interlude: stable \texorpdfstring{$\infty$}{oo}-categories}

    \begin{defn}
      Let $\sC$ be an \inftyCat with a zero object $0$ (i.e., an object which is both terminal and initial).
      Let $f : X \to Y$ be a morphism in $\sC$.
      The \emph{cofibre} of $f$, denoted $\Cofib(f)$, is the colimit of the diagram
      \[ 0 \gets X \xrightarrow{f} Y. \]
      The \emph{fibre} of $f$, denoted $\Fib(f)$, is the limit of the diagram
      \[ 0 \to Y \xleftarrow{f} X. \]
    \end{defn}


    \begin{defn}
      Let $\sC$ be an \inftyCat with zero object $0$.
      The \emph{suspension} of an object $X \in \sC$ is the object
      \[ \Sigma(X) := \Cofib(X \to 0). \]
      The \emph{loop space} of $X$ is the object
      \[ \Omega(X) := \Fib(0 \to X). \]
    \end{defn}




    \begin{defprop}
      Let $\sC$ be an \inftyCat.
      We say that $\sC$ is \emph{stable} if it satisfies the following equivalent conditions:
      \begin{thmlist}
        \item
        $\sC$ admits finite limits and a zero object (an object which is both terminal and initial), and the loop space functor $\Omega : \sC \to \sC$ is an equivalence.
        \item
        $\sC$ admits finite colimits and a zero object, and the suspension functor $\Sigma : \sC \to \sC$ is an equivalence.
        \item
        $\sC$ admits finite limits, finite colimits, and a zero object; and any commutative square is cartesian if and only if it is cocartesian.
      \end{thmlist}
    \end{defprop}

    \begin{exam}
      The \inftyCat $\Spt$ is stable.
    \end{exam}

    \begin{defn}
      Let $\sC$ be a stable \inftyCat.
      An \emph{exact triangle} in $\sC$ is a cartesian (hence also cocartesian) square
      \[\begin{tikzcd}
        X \ar{r}{f}\ar{d}
        & Y \ar{d}{g}\\
        0 \ar{r}
        & Z
      \end{tikzcd}\]
      where $0$ is a zero object.
      (Such an exact triangle gives rise to a canonical isomorphism $\Cofib(g) \to Z$.)
      We will usually depict it by the diagram
      \[ X \xrightarrow{f} Y \xrightarrow{g} Z, \]
      but note that a null-homotopy $g\circ f \simeq 0$ must be specified as part of the data of an exact triangle.
    \end{defn}

    \begin{rem}
      The homotopy category $h(\sC)$ of a stable \inftyCat admits a canonical triangulated structure, where the exact/distinguished triangles are those coming from exact triangles in $\sC$ (via the canonical functor $\sC \to h(\sC)$).
    \end{rem}

  \ssec{Nonabelian derived categories}


    \begin{prop}
      Let $\sC$ be an algebraic category.
      If $\sC$ is additive, then the forgetful functor
      \[
        \Fun_\Pi(\sF_\sC^\op, \Ab)
        \to \Fun_\Pi(\sF_\sC^\op, \Set) \simeq \sC
      \]
      is an equivalence.
    \end{prop}

    Indeed, every finite product-preserving functor $X : \sF_\sC^\op \to \Ab$ automatically takes values in abelian groups.
    We will see an analogue of this statement for the animation.
    First we need to introduce the analogue of abelian groups.

    \begin{defn}\leavevmode
      \begin{defnlist}
        \item
        A \emph{pointed animum} is an animum $X$ equipped with a morphism $\pt \to X$.

        \item
        A \emph{spectrum} is a sequence of pointed anima $$X = (X_0, X_1, X_2, \ldots)$$ together with isomorphisms
        \[ X_n \simeq \Omega(X_{n+1}) \]
        for all $n\ge 0$.

        \item
        A spectrum $X$ is \emph{connective} if the $n$th component $X_n$ is $n$-connective for every $n\ge 0$ (i.e., if $\pi_i(X_n) = 0$ for all $i<n$).
      \end{defnlist}
    \end{defn}

    Informally speaking, we say that a spectrum $X$ is an \emph{infinite delooping} of $X_0$, or that it gives $X_0$ the structure of an \emph{infinite loop space}.

    \begin{notat}
      The \inftyCats $\Spt$ and $\Spt_{\ge0}$ of spectra and connective spectra, respectively, can be defined as the limits of the following towers in the \inftyCat of \inftyCats:
      \begin{align*}
        &\cdots
        \xrightarrow{\Omega} \Anima_\bullet
        \xrightarrow{\Omega} \Anima_\bullet
        \xrightarrow{\Omega} \Anima_\bullet,\\
        &\cdots
        \xrightarrow{\Omega} (\Anima_\bullet)_{\ge 2}
        \xrightarrow{\Omega} (\Anima_\bullet)_{\ge 1}
        \xrightarrow{\Omega} (\Anima_\bullet)_{\ge 0},
      \end{align*}
      where $\Anima_\bullet$ denotes the \inftyCat of pointed anima, and the subscript $\ge n$ indicates the full subcategory of $n$-connective pointed anima.
      We have projections
      \[ \Omega^{\infty-n} : \Spt \to \Anima_\bullet, \]
      sending $(X_0,X_1,\ldots)\mapsto X_n$.
    \end{notat}

    In the following important theorem, \emph{$\sE_\infty$-groups} are the analogues of abelian groups in the derived/animated world.
    Informally speaking, an $\sE_\infty$-group structure on an animum is an infinite collection of group structures, all commuting with each other up to coherent homotopy (cf. the Eckmann--Hilton argument which implies that if a set admits two commuting group laws, both agree and are commutative).

    \begin{thm}[Infinite loop space machine]\leavevmode
      \begin{thmlist}
        \item
        Let $X$ be a spectrum.
        Then the infinite loop space $\Omega^\infty(X)$ admits an $\Einfty$-group structure.
        More precisely, $\Omega^\infty : \Spt \to \Anima_\bullet$ lifts to a functor
        \[ \Omega^\infty : \Spt \to \{\text{$\Einfty$-groups}\}. \]

        \item
        Restricted to the full subcategory of \emph{connective} spectra, this defines an equivalence of \inftyCats
        \[
          \Spt_{\ge 0} \simeq \{\text{$\Einfty$-groups}\}.
        \]
        In particular, any $\Einfty$-group structure on a pointed animum $X$ gives rise to a unique infinite delooping $\mathrm{B}^\infty X$.
        Here $\mathrm{B}^\infty X$ is a connective spectrum such that $\Omega^\infty \mathrm{B}^\infty X \simeq X$.
      \end{thmlist}
    \end{thm}

    \begin{prop}
      Let $\sC$ be an additive algebraic category.
      Then there is a canonical equivalence
      \[\Anim(\sC) \simeq \Fun_\Pi(\sF_\sC^\op, \Spt_{\ge 0}).\]
      In particular, there exists a canonical fully faithful embedding
      \[
        \Sigma^\infty : \Anim(\sC) \hook \Fun_\Pi(\sF_\sC^\op, \Spt)
      \]
      whose target is a \emph{stable} \inftyCat, and whose essential image is closed under colimits and extensions.
    \end{prop}

    \begin{defn}
      The \emph{derived \inftyCat} of $\sC$ is the stable \inftyCat
      \[
        \D(\sC) := \Fun_\Pi(\sF_\sC^\op, \Spt).
      \]
      We also let $\D(\sC)_{\ge 0}$ denote the essential image of the fully faithful embedding $\Anim(\sC) \hook \D(\sC)$.
      We will say that an object $X \in \D(\sC)$ is \emph{connective} if it lies in $\D(\sC)_{\le 0}$.
    \end{defn}

  \ssec{Derived functors}

    \begin{constr}\label{constr:LF}
      Let $F : \sC \to \sD$ be a functor of algebraic categories.
      Its restriction to $\sF_\sC$ gives rise by the universal property of animation to an essentially unique functor $\L F : \Anim(\sC) \to \Anim(\sD)$ such that:
      \begin{defnlist}
        \item
        $\L F$ preserves filtered colimits and geometric realizations.

        \item
        There is a commutative diagram
        \[ \begin{tikzcd}
          \sF_\sC \ar{r}{F|_{\sF_\sC}}\ar[hookrightarrow]{d}
          & \sD \ar[hookrightarrow]{r}
          & \Anim(\sD).
          \\
          \Anim(\sC) \ar[swap]{rru}{\bL F}
        \end{tikzcd} \]
      \end{defnlist}
      Moreover, $\L F$ preserves finite coproducts (and hence all colimits) if and only if $F$ preserves finite coproducts.
    \end{constr}

    \begin{rem}
      We have canonical isomorphisms
      \[ \on{\pi_0} {\L F}(X) \simeq F( \pi_0(X) ), \]
      functorial in $X \in \Anim(\sC)$, where $\pi_0$ is left adjoint to the fully faithful functor $\sC \hook \Anim(\sC)$.
      Indeed, both functors agree by definition after restriction to $\sF_\sC \sub \sC$ (since $\pi_0$ is identity on discrete objects), so this follows by universal property of $\Anim(\sC)$.
    \end{rem}


    \begin{prop}\label{prop:0aysgf}
      Let $F : \sC \to \sD$ and $G : \sD \to \sE$ be functors of algebraic categories which preserve filtered colimits and reflexive coequalizers.
      Then there is a canonical equivalence
      \[\L G \circ \L F \simeq \L (G\circ F)\]
      under either of the following conditions:
      \begin{thmlist}
        \item[{\em(i)}]\label{item:a0fsh8/1}
        $F$ sends $\sF_\sC$ to $\sF_\sD$.
        (More generally, it is enough that for every $X \in \sF_\sC$, $F(X) \in \sD$ is a filtered colimit of objects in $\sF_\sD$.)

        \item[{\em(ii)}]\label{item:a0fsh8/2}
        $\L G : \Anim(\sD) \to \Anim(\sE)$ preserves discrete objects.
        (More generally, it is enough that for every $X \in \sF_\sC$, $\L G(F(X)) \in \Anim(\sE)$ is discrete).
      \end{thmlist}
    \end{prop}
    \begin{proof}
      In either case, one easily verifies that $\L G \circ \L F$ satisfies the universal property of $\L (G \circ F)$.
    \end{proof}

    By universal properties of stabilization we have:
    \begin{prop}
      Let $\sC$ and $\sD$ be additive algebraic categories and $F : \Anim(\sC) \to \Anim(\sD)$ a functor.
      \begin{thmlist}
        \item
        If $F$ commutes with $\Omega$ then it extends uniquely to a functor
        \[ F : \D(\sC) \to \D(\sD) \]
        such that $\Omega^{\infty-n} \circ F \simeq F \circ \Omega^{\infty-n}$ for all $n\ge 0$.
        Informally speaking, it sends $(X_0,X_1,\ldots) \mapsto (F(X_0),F(X_1),\ldots)$.

        \item
        If $F$ preserves colimits and zero objects, then it extends uniquely to a colimit-preserving functor
        \[ F : \D(\sC) \to \D(\sD) \]
        such that $F \circ \Sigma^{\infty-n} \simeq \Sigma^{\infty-n} \circ F$ for all $n\ge 0$.
        Here $\Sigma^{\infty-n} : \Anim(-) \to \D(-)$ is left adjoint to $\Omega^{\infty-n}$.
      \end{thmlist}
    \end{prop}

  \ssec{Animated modules}

    Let $A$ be a ring\footnote{%
      From now on, all our rings be assumed commutative.
    } and consider the algebraic category $\Mod_A$ of $A$-modules (\examref{exam:Mod}).
    Let's specialize the animation construction to this case.

    \begin{defn}
      Let $\sF_A \sub \Mod_A$ denote the full subcategory of finitely generated free $A$-modules.
      An \emph{animated $A$-module} is a finite product-preserving functor $M : \sF_A^\op \to \Anima$.
      We write
      \[ \D(A)_{\ge 0} := \Anim(\Mod_A) \]
      for the \inftyCat of animated $A$-modules.
      (Cohomologically: $\D(A)^{\le 0}$.)
    \end{defn}

    \begin{rem}
      Let us unpack the definition a bit.
      An animated $A$-module $M$ amounts to the following data:
      \begin{defnlist}
        \item
        For every integer $n\ge 0$, an animum $M_n \in \Anima$.
        \item
        For every $A$-linear map $\phi : A^{\oplus n} \to A^{\oplus m}$ (or $\phi \in \Mat_{m \times n}(A)$), a map of anima $M_\phi : M_m \to M_n$.
        \item\label{item:a0ps7ugf}
        For every two $A$-linear maps $\phi : A^{\oplus n} \to A^{\oplus m}$ and $\psi : A^{\oplus m} \to A^{\oplus l}$, a homotopy $h_{\phi,\psi} : M_\phi \circ M_\psi \simeq M_{\psi \circ \phi}$ of maps $M_l \to M_n$.
        \item\label{item:a0psdg1}
        For every three $A$-linear maps $\phi$, $\psi$, $\omega$, a tetrahedron-shaped diagram expressing a ``higher'' homotopy between the homotopies $h_{\phi,\psi}$, $h_{\psi,\omega}$, $h_{\phi,\omega\circ\psi}$, and $h_{\psi\circ\phi,\omega}$.
        \item
        \textellipsis
      \end{defnlist}
      This data is subject to the condition that the canonical map $M_n \to (M_1)^{\times n}$ is invertible for every $n\ge 0$ (in particular, $M_0 \simeq \pt$).
      We summarize \itemref{item:a0ps7ugf} and \itemref{item:a0psdg1} by saying that the maps $M_\phi$ are functorial \emph{up to coherent homotopy}.

      In particular, this data encodes:
      \begin{defnlist}
        \item
        The underlying animum $M^\circ := M_1 \in \Anima$.

        \item
        Operations $(M^\circ)^{\times n} \to M^\circ$ on $M^\circ$, for every $\phi \in \Mat_{n\times 1}(A)$.
        In particular, an addition operation $\mrm{add} : M^\circ \times M^\circ \to M^\circ$.

        \item
        An action of $A$ on $M^\circ$, i.e., a map $A \to \End(M^\circ)$ given by
        \[ A \simeq \Mat_{1\times 1}(A) = \Hom_{\sF_A}(1, 1) \xrightarrow{M} \Maps_{\Anima}(M_1, M_1) = \End(M^\circ). \]
        The endomorphism induced by $a \in A$ is the operation encoded by the matrix $a \in \Mat_{1\times 1}(A)$.

        \item
        Associativity up to coherent homotopy.
        For example, given three points $x,y,z \in M$ (i.e., maps of anima $\pt \to M^\circ$) we have a homotopy
        \[ \mrm{add}(\mrm{add}(x,y), z) \simeq \mrm{add}(x, \mrm{add}(y,z)). \]
        Diagrammatically,
        \[ \begin{tikzcd}
          M^\circ \times M^\circ \times M^\circ \ar{r}{\mrm{add} \times \id}\ar{d}{\id \times \mrm{add}}
          & M^\circ \times M^\circ \ar{d}{\mrm{add}}
          \\
          M^\circ \times M^\circ \ar{r}{\mrm{add}}
          & M^\circ.
        \end{tikzcd} \]
      \end{defnlist}
      Informally speaking, we can think of an animated $A$-module as an animum equipped with a homotopy coherent $A$-module structure.
    \end{rem}

\section{The cotangent complex}
\label{sec:cotangent}

  Given a morphism of schemes $f : X \to Y$, recall the quasi-coherent sheaf of relative algebraic K\"ahler differentials
  \[ \Omega_{X/Y} \in \QCoh(X). \]
  Given another morphism $g : Y \to Z$, there is a right-exact sequence
  \[
    f^*\Omega_{Y/Z}
    \to \Omega_{X/Z}
    \to \Omega_{X/Y}
    \to 0.
  \]
  In this lecture we will use the formalism of the previous lecture to define the ``derived functor of $\Omega_{-/-}$'', which will allow us to extend the above sequence to the left.

  \ssec{Animated rings}

    Given a commutative ring $R$, we denote by $\CAlg_R$ the category of commutative $R$-algebras.

    \begin{prop}
      Denote by $\Poly_R \sub \CAlg_R$ the full subcategory spanned by finitely generated polynomial algebras $R[T_1,\ldots,T_m]$, $m\ge 0$.
      The inclusion $\Poly_R \hook \CAlg_R$ extends to an equivalence
      \[
        \Fun_\Pi(\Poly_R^\op, \Set)
        \to \CAlg_R.
      \]
      In particular, the category $\CAlg_R$ is algebraic.
    \end{prop}

    \begin{defn}
      Denote by $\ACAlg_R := \Anim(\CAlg_R)$ the animation of $\CAlg_R$.
      Objects of $\ACAlg_R$ will be called \emph{animated (commutative) $R$-algebras}.
      When $R=\Z$, we abbreviate $\ACRing := \ACAlg_\Z$ and call its objects \emph{animated (commutative) rings}.
    \end{defn}

    \begin{rem}
      Note that $\CAlg_R$ is not additive: for $A,B\in\CAlg_R$, the coproduct is the tensor product $A \otimes_R B$, which is usually distinct from the product $A \times B$.
    \end{rem}



    \begin{notat}
      Let $R$ be a commutative ring.
      We denote by $\CAlgMod_R$ the category of pairs $(A, M)$ where $A\in\CAlg_R$ and $M \in \Mod_A$; a morphism $(A,M) \to (A',M')$ is an $R$-algebra homomorphism $A \to A'$ together with an $A'$-module homomorphism $M \otimes_A A' \to M'$.
    \end{notat}

    \begin{defprop}\label{prop:mijl}
      Let $R$ be a commutative ring.
      \begin{thmlist}
        \item
        Denote by $\CAlgMod_R^\free \sub \CAlgMod_R$ the full subcategory spanned by pairs $(A,M)$ where $A=R[T_1,\ldots,T_m]$, $m\ge 0$, and $M=A^{\oplus n}$, $n\ge 0$.
        The inclusion $\CAlgMod_R^\free \hook \CAlgMod_R$ extends to an equivalence
        \[
          \Fun_\Pi(\CAlgMod_R^{\free,\op}, \Set)
          \to \CAlgMod_R.
        \]
        In particular, the category $\CAlgMod_R$ is algebraic.

        \item
        Consider the canonical functor
        \begin{equation*}
          \pi : \Anim(\CAlgMod_R) \to \Anim(\CAlg_R) = \ACAlg_R
        \end{equation*}
        induced by $\CAlgMod_R \to \CAlg_R$, $(A,M)\mapsto A$.
        Then $\pi$ is a cocartesian fibration.
        Given an animated ring $A$, the \inftyCat of \emph{animated $A$-modules} is the fibre $\AMod_A$ over $A\in\CAlg_R$.

        \item
        Let $\sF_A \sub \AMod_A$ denote the full subcategory spanned by finite direct sums $A^{\oplus n}$, $n\ge0$.
        Then there are canonical equivalences
        \[ \AMod_A \simeq \Fun_\Pi(\sF_A^\op, \Anima) \simeq \Fun_\Pi(\sF_A^\op, \Spt_{\ge 0}). \]

        \item
        The \inftyCat of \emph{derived $A$-modules} is
        \[ \D(A) := \Fun_\Pi(\sF_A^\op, \Spt). \]
        We write $\D(A)_{\ge 0}$ for the essential image of the fully faithful functor $\AMod_A \hook \D(A)$.
      \end{thmlist}
    \end{defprop}


    \begin{defprop}\label{defn:redfin}
      Let $A \to B$ be a morphism of animated rings.
      We say that it is \emph{flat} if it satisfies the following equivalent conditions:
      \begin{thmlist}
        \item
        The functor $(-)\otimes^\L_A B : \D(A)_{\ge 0} \to \D(B)_{\ge 0}$ is left-exact, i.e., it preserves discrete objects.

        \item
        The induced ring homomorphism $\pi_0(A) \to \pi_0(B)$ is flat, and the canonical homomorphisms
        \[ \pi_i(A) \otimes_{\pi_0(A)} \pi_0(B) \to \pi_i(B) \]
        are bijective for all $i\ge 0$.
      \end{thmlist}
    \end{defprop}

    \begin{defn}
      We say that $A \to B$ is \emph{étale}, resp. \emph{smooth}, if it is flat and the induced ring homomorphism $\pi_0(A) \to \pi_0(B)$ is étale, resp. smooth (in the sense of ordinary commutative algebra).
    \end{defn}

  \ssec{The cotangent complex}

    \begin{constr}
      Let $R$ be a commutative ring.
      Consider the functor $\CAlg_R \to \CAlgMod_R$ sending $A \mapsto (A, \Omega_{A/R})$.
      Its restriction to $\CAlg_R^\free$ extends uniquely to a sifted colimit-preserving functor
      \[
        \ACAlg_R = \Anim(\CAlg_R) \to \Anim(\CAlgMod_R)
      \]
      such that the diagram
      \[\begin{tikzcd}
        \ACAlg_R \ar{rr}\ar[equals]{rd}
        && \Anim(\CAlgMod_R) \ar{ld}{\pi}
        \\
        & \ACAlg_R &
      \end{tikzcd}\]
      commutes.
      The image of $A\in\Anim(\CAlg_R)$ may be regarded as a pair $(A, \L\Omega_{A/R})$ where $\L\Omega_{A/R} \in \Anim(\Mod_A)$.
    \end{constr}

    \begin{defn}
      Given $A\in\ACAlg_R$, the (relative) \emph{cotangent complex} is $\L_{A/R} := \L\Omega_{A/R}$, which we regard as an object of $\D(A)_{\ge 0}$, i.e., as a connective derived $A$-module.
    \end{defn}
    
    By construction, we have $\pi_0 \L_{A/R} \simeq \Omega_{\pi_0(A)/\pi_0(R)}$.
    We also have the following further properties:



    \begin{prop}\label{thm:antiministerialist}
      \begin{thmlist}
        \item
        Let $A \to B$ be a morphism in $\Anim(\CAlg_R)$.
        Then there is an exact triangle
        \[
          \L_{A/R} \otimes^\L_A B
          \to \L_{B/R}
          \to \L_{B/A}
        \]
        in $\D(B)$.

        \item
        For every cocartesian square
        \[\begin{tikzcd}
          A \ar{r}\ar{d}
          & B \ar{d}
          \\
          A' \ar{r}
          & B'
        \end{tikzcd}\]
        in $\Anim(\CAlg_R)$ (exhibiting $B'$ as the derived tensor product $A' \otimes^\L_A B$), there is a canonical isomorphism
        \[
          \L_{B/A} \otimes^\L_B B'
          \simeq \L_{B'/A'}.
        \]
      \end{thmlist}
    \end{prop}

  \ssec{The universal property of the cotangent complex}

    \begin{defn}
      Let $R$ be a commutative ring.
      Consider the functor
      \[ \CAlgMod_R \to \CAlg_R \]
      sending a pair $(A,M)$ to the trivial square-zero extension $A\oplus M$, with multiplication $(a, m) \cdot (a', m') = (a a', a'm + am')$.
      This gives rise to a functor
      \[ \Anim(\CAlgMod_R) \to \Anim(\CAlg_R) = \ACAlg_R \]
      on animations.
      We still denote the image of $(A, M)$ by $A\oplus M$ and call it the \emph{trivial square-zero extension} of $A$ by $M\in\D(A)_{\ge 0}$.
    \end{defn}

    \begin{defn}
      Let $R\in\CRing$, $A \in \ACAlg_R$, and $M \in \AMod_R$ an animated $R$-module.
      The canonical morphism $A\oplus M \to A$, given informally by $(a,m) \mapsto a$, induces a map
      \[
        \Maps_{\ACAlg_R}(A,A\oplus M)
        \to \Maps_{\ACAlg_R}(A,A).
      \]
      The animum $\Der_R(A, M)$ of \emph{$R$-linear derivations of $A$ with values in $M$} is the homotopy fibre of this map at the point $\id_A$.
      In other words, a point of $\Der_R(A,M)$ amounts to a dashed arrow making the below diagram commute:
      \[\begin{tikzcd}
        & R\ar{ld}\ar{rd} &
        \\
        A \ar[dashed]{rr}\ar[equals]{rd}
        & & A \oplus M \ar{ld}
        \\
        & A &
      \end{tikzcd}\]
      The \emph{trivial derivation} is defined informally by $a \mapsto (a,0)$.
    \end{defn}

    \begin{thm}\label{thm:sleep}
      Let $A$ be an animated $R$-algebra.
      Then the cotangent complex $\L_{A/R}$ corepresents the functor $\D(A)_{\ge 0} \to \Anima$, $M \mapsto \Der_R(A, M)$.
      That is, there are canonical isomorphisms
      \[
        \Maps_{\D(A)}(\L_{A/R}, M)
        \simeq \Der_R(A, M)
      \]
      functorial in $M\in\D(A)_{\ge 0}$.
    \end{thm}
    \begin{proof}
      If $A = R[T_1,\ldots,T_m]$ is a polynomial algebra, we have
      \[ \L_{A/R} \simeq \Omega_{A/R} \]
      which is free on $m$ generators.
      In general, write $A$ as a geometric realization of a simplicial diagram of filtered systems of polynomial algebras.
    \end{proof}

  \ssec{Quasi-coherent and perfect complexes on stacks}

    \begin{defn}
      We extend the flat topology on $\CRing^\op$ to $\ACRing^\op$ as follows: we say that a morphism $A \to B$ in $\ACRing$ is \emph{faithfully flat} if it is flat in the sense of \defnref{defn:redfin}, and the induced ring homomorphism $\pi_0(A) \to \pi_0(B)$ is faithfully flat.
    \end{defn}

    \begin{thm}[Lurie, To\"en]\label{thm:jargonish}
      The functor
      \[
        \ACRing \to \Catoo,
        \quad R \mapsto \D(R)
      \]
      satisfies descent for the flat topology.
    \end{thm}
    \begin{proof}
      Similar to the proof of \thmref{thm:module descent}.
    \end{proof}

    \begin{cor}
      The functor
      \[
        \CRing \to \Catoo,
        \quad R \mapsto \D(R)
      \]
      satisfies descent for the flat topology.
    \end{cor}

    \begin{defn}\label{defn:Dqc}
      Let $X$ be a stack.
      The stable \inftyCat of quasi-coherent sheaves $\Dqc(X)$ is the limit
      \[
        \Dqc(X) := \lim_{(R,x)} \D(R)
      \]
      over the category of pairs $(R,x)$ where $R\in\CRing$ and $x\in X(R)$.
      More precisely,
      \[ \Dqc : \Stk^\op \to \Catoo \]
      is the right Kan extension of the presheaf $\Spec(R) \mapsto \D(R)$ along the inclusion $\Aff \hook \Stk$.
      We refer to objects of $\Dqc(X)$ as \emph{quasi-coherent complexes}.
    \end{defn}

    \begin{exam}
      If $X=\Spec(R)$ is affine, then we have the tautological equivalence
      \[ \R\Gamma(X, -) : \Dqc(X) \to \D(R). \]
    \end{exam}

    \begin{exam}
      Let $G$ be a group stack over a scheme $S$, and let $X$ be a stack over $S$ with $G$-action.
      The \inftyCat of \emph{$G$-equivariant quasi-coherent complexes} on $X$ is defined as the totalization
      \[
        \Dqc^G(X) := \Tot(\Dqc(X_\bullet))
      \]
      where $X_\bullet = [\cdots \rightrightrightarrows G \times X \rightrightarrows X]$ is the action groupoid.
      Informally speaking, its objects are quasi-coherent complexes on $X$ equipped with a homotopy coherent $G$-action.
      By \propref{prop:checkage} there is a canonical equivalence
      \[ \Dqc([X/G]) \simeq \Dqc^G(X). \]
    \end{exam}

    \begin{prop}\label{prop:rascalion}
      Let $X$ be a stack.
      If $X$ is algebraic (resp. Deligne--Mumford, schematic), then the limit in \defnref{defn:Dqc} can be taken only over pairs $(R,x)$ such that $x : \Spec(R) \to X$ is smooth (resp. étale, resp. an open immersion).
    \end{prop}
    \begin{proof}
      
    \end{proof}

    \begin{defn}\leavevmode
      \begin{defnlist}
        \item
        Let $R$ be a commutative ring.
        We say that a derived $R$-module $M \in \D(R)$ is \emph{perfect} if it is contained in the full subcategory of $\D(R)$ generated by $R$ under finite colimits, finite limits, and direct summands.
        
        \item
        Let $X$ be a stack.
        We say that a quasi-coherent complex $\sF \in \Dqc(X)$ is \emph{perfect} if for every $(R,x)$ the derived $R$-module $\R\Gamma(\Spec(R), x^*\sF)$ is perfect.
      \end{defnlist}
      We write $\Dperf(X) \sub \Dqc(X)$ for the full subcategory spanned by perfect complexes, so that
      \[ \Dperf(X) \simeq \lim_{(R,x)} \Dperf(R) \]
      where $\Dperf(R)$ is the \inftyCat of perfect derived $R$-modules.
    \end{defn}

    \begin{cor}
      Let $p : U \twoheadrightarrow X$ be an effective epimorphism.
      Then $\sF \in \Dqc(X)$ is perfect if and only if $\L p^*(\sF) \in \Dqc(U)$ is perfect.
    \end{cor}
    \begin{proof}
      Follows from \propref{prop:checkage}.
    \end{proof}

    \begin{cor}
      Let $U$ be a stack with $G$-action, and denote by $p : U \twoheadrightarrow [U/G]$ the quotient.
      Then a quasi-coherent complex $\sF \in \Dqc([U/G]) \simeq \Dqc^G(U)$ is perfect if and only if the underlying quasi-coherent complex $p^*\sF \in \Dqc(U)$ is perfect. 
    \end{cor}

  \ssec{Derived stacks}

    \begin{defn}
      A \emph{derived stack} is a functor $X : \ACRing \to \Grpdoo$ satisfying étale descent.
      We denote the \inftyCat of derived stacks by $\DStk = \Shv(\ACRing^\op; \Grpdoo)$.
    \end{defn}

    \begin{exam}
      An \emph{affine derived scheme} is the functor $\Spec(A) : \ACRing \to \Grpdoo$, $B \mapsto \Maps(A, B)$, corepresented by an animated ring $A \in \ACRing$.
    \end{exam}

    \begin{defn}
      Given a derived stack $X$, the restriction of the functor $X : \ACRing \to \Grpdoo$ along $\CRing \hook \ACRing$ is called the \emph{classical truncation} of $X$ and is denoted $X_\cl$.
      If the \inftyGrpd $X(R)$ is $1$-truncated for every $R\in\CRing$, then $X_\cl : \CRing \to \Grpd$ is a stack.
      For example, this will be the case if $X$ is a derived algebraic stack (see \defnref{defn:fervescence}).
    \end{defn}

    \begin{notat}
      We distinguish limits in $\Stk$ and $\DStk$ by writing $\R\lim$ instead of $\lim$ for the latter.
      For example, given a diagram of stacks $X \to Z \gets Y$, the usual fibre product is the classical truncation of the derived one:
      \[
        (X \fibprodR_Z Y)_\cl
        \simeq X \fibprod_Z Y,
      \]
      while $X \fibprodR_Z Y \simeq X \fibprod_Z Y$ if and only if $X$ and $Y$ are Tor-independent over $Z$.
    \end{notat}

    \begin{defn}
      Given a derived stack $X$, we define the \inftyCat $\Dqc(X)$ by a further right Kan extension of $\Dqc : \Stk^\op \to \Catoo$ along the fully faithful functor $\Stk \hook \DStk$ (induced by the fully faithful functor $\CRing \hook \ACRing$).
      Thus for any derived stack $X$ we have
      \[ \Dqc(X) \simeq \lim_{(A,x)} \D(A), \]
      where the limit is now taken over the \inftyCat of pairs $(A \in \ACRing, x \in X(A))$.
    \end{defn}

    We have the obvious derived versions of Definitions~\ref{defn:algspace} and \ref{defn:algstack}:

    \begin{defn}[Derived algebraic spaces]\leavevmode
      \begin{defnlist}
        \item      
        A derived stack $X$ is \emph{$0$-Artin}, or a \emph{derived algebraic space}, if its diagonal $X \to X \times X$ is schematic and induces a monomorphism on classical truncations, and there exists an étale surjection $U \twoheadrightarrow X$ where $U$ is an affine derived scheme.

        \item
        A morphism $f : X \to Y$ is \emph{$0$-Artin}, or \emph{representable}, if for every affine $V$ and every morphism $V \to Y$, the fibre $X \fibprodR_Y V$ is a derived algebraic space.

        \item
        A $0$-Artin morphism $f : X \to Y$ is \emph{flat}, \emph{smooth}, or \emph{surjective} if for every affine $V$ and every morphism $V \to Y$, there exists a derived scheme $U$ and an étale surjection $U \twoheadrightarrow X\fibprod_Y V$ such that the composite $U \twoheadrightarrow X \fibprod_Y V \to V$ has the respective property.
      \end{defnlist}
    \end{defn}

    \begin{defn}[Derived algebraic stacks]\label{defn:fervescence}\leavevmode
      \begin{defnlist}
        \item
        A derived stack $X$ is \emph{$1$-Artin}, or a \emph{derived algebraic stack}, if its diagonal is representable and there exists a smooth surjection $U \twoheadrightarrow X$ where $U$ is a derived algebraic space.

        \item
        A morphism of derived stacks $f : X \to Y$ is \emph{$1$-representable}, or \emph{representable by derived algebraic stacks}, if for every affine derived scheme $V$ and every morphism $v : V \to Y$, the derived fibre $X \fibprodR_Y V$ is a derived algebraic stack.

        \item
        A derived algebraic stack is \emph{Deligne--Mumford} if it admits an \emph{étale} surjection from a derived scheme.
        Equivalently, its classical truncation is Deligne--Mumford.
      \end{defnlist}
    \end{defn}

    Continuing the same procedure, we can define inductively:

    \begin{defn}[Derived Artin stacks]\leavevmode
      \begin{defnlist}
        \item
        For $n>0$, a morphism of derived stacks $f : X \to Y$ is \emph{$(n-1)$-Artin} if for every affine $V$ and every morphism $V \to Y$, $X \fibprodR_Y V$ is $(n-1)$-Artin.
        A derived stack $X$ is \emph{$n$-Artin} if its diagonal is $(n-1)$-Artin and there exists a smooth surjection $U \twoheadrightarrow X$ where $U$ is a derived scheme.
        An $(n-1)$-Artin morphism $f : X \to Y$ is \emph{flat}, \emph{smooth}, or \emph{surjective}, if for any affine $V$ and any morphism $V \to Y$, there exists a derived scheme $U$ and a smooth surjection $U \twoheadrightarrow X\fibprod_Y V$ such that the composite $U \twoheadrightarrow X \fibprod_Y V \to V$ has the respective property.

        \item
        A derived stack is \emph{Artin} if it is $n$-Artin for some $n$.
        A morphism of derived stacks is \emph{Artin} if it is $n$-Artin for some $n$.
        A morphism of derived stacks is \emph{flat}, \emph{smooth}, or \emph{surjective}, if it is $n$-Artin with the respective property for some $n$.
      \end{defnlist}
    \end{defn}

    \begin{rem}
      An $n$-Artin stack takes values in $n$-groupoids, i.e., in \inftyGrpds that are $n$-truncated.
    \end{rem}

    \begin{rem}
      Classically, the terms \emph{algebraic stack} and \emph{Artin stack} are used interchangeably.
      We follow Gaitsgory in redefining Artin stacks as ``higher'' Artin stacks, and algebraic stacks to be $1$-Artin stacks.
    \end{rem}

  \ssec{Cotangent complexes of algebraic stacks}

    \begin{defn}
      Let $f : X \to Y$ be a morphism in $\DStk$.
      Given an animated ring $A \in \ACRing$, an $A$-point $x \in X(A)$, and a connective derived $A$-module $M\in\D(A)_{\ge 0}$, consider the commutative square
      \[\begin{tikzcd}
        X(A\oplus M) \ar{r}\ar{d}
        & X(A) \ar{d}
        \\
        Y(A\oplus M) \ar{r}
        & Y(A)
      \end{tikzcd}\]
      induced by the morphism of animated rings $A \oplus M \to A$, $(a,m) \mapsto a$.
      This gives rise to a map of anima (or \inftyGrpds)
      \[
        X(A\oplus M)
        \to Y(A\oplus M) \fibprod_{Y(A)} X(A).
      \]
      The animum $\Der_x(X/Y, M)$ of \emph{derivations of $f$ at $x$ with values in $M$} is the fibre at the point $(y', x)$, where $y' \in Y(A\oplus M)$ is the image of $y=f(x) \in Y(A)$ along the trivial derivation $A \to A\oplus M$, $a\mapsto (a,0)$.
    \end{defn}

    \begin{defn}
      Let $f : X \to Y$ be a morphism of (derived) stacks and $\sL \in \Dqc(X)$ an \emph{eventually connective} quasi-coherent complex on $X$ (i.e., $\sL[n]$ is connective for some integer $n$).
      We say that $\sL$ is a \emph{(relative) cotangent complex} for $f : X \to Y$ if for every $A\in\ACRing$ and $x:\Spec(A) \to X$, the inverse image $\L x^*\sL$ corepresents the functor $\Der_x(X/Y, -)$.
      That is, there are isomorphisms
      \[
        \Maps_{\D(A)}(\L x^*\sL, M)
        \simeq \Der_x(X/Y, M)
      \]
      functorial in $M \in \D(A)_{\ge 0}$, where by abuse of notation we identify $\L x^*\sL \in \Dqc(\Spec(A))$ with the derived $A$-module $\R\Gamma(\Spec(A), \L x^*\sL) \in \D(A)$.
    \end{defn}

    \begin{thm}
      Let $f : X \to Y$ be a $1$-representable morphism of derived stacks.
      \begin{thmlist}
        \item
        There exists a cotangent complex $\L_{X/Y}$ for $f$.
      
        \item
        The cotangent complex $\L_{X/Y}$ is $(-1)$-connective.
        That is, for every derived scheme $U$ and every smooth morphism $p : U \to X$, the inverse image $p^*\L_{X/Y}$ is $(-1)$-connective, i.e.
        \[ \on{H}^{-i}(U, \L_{X/Y}) = \pi_i \R\Gamma(U, \L_{X/Y}) = \pi_i \Maps_{\D(U)}(\sO_U, p^*\L_{X/Y}) = 0 \]
        for all $i<-1$.
        If $f$ is representable by derived algebraic spaces (or derived Deligne--Mumford stacks), then $\L_{X/Y}$ is in fact connective.
      \end{thmlist}
    \end{thm}

    \begin{proof}
      To construct a cotangent complex $\sL$ for $f$, it will suffice by definition of $\Dqc(X)$ to construct for every $A \in \ACRing$ and every morphism $x : \Spec(A) \to X$, a quasi-coherent complex $\sL(x) \in \Dqc(\Spec(A)) \simeq \D(A)$ which corepresents the functor $\Der_x(X/Y, -)$, in such a way that $\sL(x)$ is stable under derived base change as $(A,x)$ varies.

      Let $x : \Spec(A) \to X$ be a morphism where $A\in\ACRing$.
      Since $f$ is $1$-representable, the derived fibre $F = X\fibprodR_Y \Spec(A)$ is algebraic.
      Supposing that $\L_{F/\Spec(A)}$ exists, it is easy to see that the collection
      \[ (a^*\L_{F/\Spec(A)})_{A,x}, \]
      where $a = (f,\id) : \Spec(A) \to F$, defines a cotangent complex for $\L_{X/Y}$.
      Replacing $f$ by $F \to \Spec(A)$, we may therefore assume that $Y$ is affine and $X$ is algebraic.

      Thus write $Y=\Spec(R)$.
      Suppose first that $X$ is also affine, say $X=\Spec(A)$.
      Then we let $\L_{X/Y} \in \Dqc(X)_{\ge 0}$ be the connective quasi-coherent complex corresponding to the connective derived $A$-module
      \[
        \L_{R/A} := \Cofib(\L_R \otimesL_R A \to \L_A) \in \D(A)_{\ge 0}.
      \]
      Using \thmref{thm:sleep} it is easy to check that is indeed a cotangent complex for $X \to Y$.

      Suppose next that $X$ is a derived Deligne--Mumford stack.
      By \propref{prop:rascalion}, the collection
      \[
        (\L_{\Spec(A)/\Spec(R)})_{(A,x)} \in \lim_{(A,x)} \D(\Spec(A))_{\ge 0} \simeq \Dqc(X)_{\ge 0},
      \]
      where the limit is taken over pairs $(A\in\ACRing,x\in X(A))$ such that $x : \Spec(A) \to X$ is étale, determines a connective quasi-coherent complex $\L_{X/Y}$ over $X$ with a canonical identification $\L x^*\L_{X/Y} \simeq \L_{R/A}$ for every $A\in\ACAlg_R$ and étale morphism $x : \Spec(A) \to X$.
      It is easy to check that $\L_{X/Y}$ is a cotangent complex for $f$.

      Finally consider the case of a derived algebraic stack $X$.
      Let $A\in\ACRing$ and $x : \Spec(A) \to X$ a morphism.
      Choose a smooth surjection $p : U \twoheadrightarrow X$ where $U$ is a derived algebraic space.
      Since $p$ is an effective epimorphism, there exists an étale cover $\Spec(A') \twoheadrightarrow \Spec(A)$ up to which $x$ lifts to $u : \Spec(A') \to U$.
      By étale descent, it will suffice to construct $\sL(\Spec(A') \to X) \in \Dqc(\Spec(A'))$ instead of $\sL(x)$.
      We thus replace $A$ by $A'$ and assume that $x$ lifts to $u : \Spec(A) \to U$. 

      Let $M \in \D(A)_{\ge 0}$.
      Note that there is a canonical base point of $\Der_x(X/Y, M)$ given by the image of $x \in X(A)$ in $X(A\oplus M)$ by the trivial derivation.
      It is clear from the definitions that we have a fibre sequence
      \[
        \Der_u(U/X, M)
        \to \Der_u(U/Y, M)
        \to \Der_x(X/Y, M).
      \]
      Since $U \to X$ and $U \to X \to Y$ are representable (by derived algebraic spaces) we have
      \begin{align*}
        \Der_u(U/X, M)
        &\simeq \Maps_{\D(A)}(\L u^* \L_{U/X}, M)\\
        \Der_u(U/Y, M)
        &\simeq \Maps_{\D(A)}(\L u^* \L_{U/Y}, M)
      \end{align*}
      by the cases already considered.
      Since $p$ is formally smooth, the map $\Der_u(U/Y,M) \to \Der_u(X/Y,M)$ is surjective (on $\pi_0$).
      It follows that there is also a fibre sequence
      \[
        \Der_x(X/Y, M)
        \to \Maps_{\D(A)}(\L u^* \L_{U/X}, M[1])
        \to \Maps_{\D(A)}(\L u^* \L_{U/Y}, M[1]),
      \]
      since the second map here is a delooping of $\Der_u(U/X,M) \to \Der_u(U/Y,M)$.
      Hence
      \[
        \Der_x(X/Y, M)
        \simeq \Maps_{\D(A)}(\sL(x), M)
      \]
      where we set
      \[
        \sL(x) =
        \L u^* \Cofib(\L_{U/X}[-1] \to \L_{U/Y}[-1])
        \simeq \L u^* \Fib(\L_{U/X} \to \L_{U/Y}).
      \]
      Since $u^*\L_{U/X}$ and $u^*\L_{U/Y}$ are stable under derived base change in $u$ (they obviously assemble to the quasi-coherent complexes $\L_{U/X}$ and $\L_{U/Y}$ respectively), so is $\sL(x)$.
      Since connectivity is stable under colimits and $\L_{U/X}[-1]$ and $\L_{U/Y}[-1]$ are $(-1)$-connective (by the cases above), it is clear that $\sL(x)$ is $(-1)$-connective.
      Thus the collection $(\sL(x))_{A,x}$ assembles to a $(-1)$-connective cotangent complex $\L_{X/Y} \in \Dqc(X)$.
    \end{proof}

    \thmref{thm:antiministerialist} immediately globalizes as follows:
    \begin{prop}\label{prop:sublimationist}\leavevmode
      \begin{thmlist}
        \item\label{item:sublimationist/trans}
        Let $f : X \to Y$ be a morphism of derived stacks over a derived stack $S$.
        Then there is an exact triangle
        \[
          \L f^* \L_{Y/S}
          \to \L_{X/S}
          \to \L_{X/Y}
        \]
        in $\Dqc(X)$.

        \item\label{item:sublimationist/bc}
        For every cartesian square
        \[\begin{tikzcd}
          X' \ar{r}\ar{d}{p}
          & Y' \ar{d}
          \\
          X \ar{r}
          & Y
        \end{tikzcd}\]
        in $\DStk$ (exhibiting $X'$ as the derived fibred product $X\fibprodR_Y Y'$), there is a canonical isomorphism
        \[
          \L p^* \L_{X/Y}
          \simeq \L_{X'/Y'}.
        \]
      \end{thmlist}
    \end{prop}

    \begin{exam}\label{exam:unexpanding}
      Let $G$ be a smooth group algebraic space over a base scheme $S$, and let $U$ be a (derived) algebraic stack over $S$ with a $G$-action.
      We can describe the cotangent complex of the quotient stack $X=[U/G]$ as follows.
      Consider the cartesian square (where as usual, all products are fibred over $S$):
      \[\begin{tikzcd}
        G \times U \ar{d}{\mrm{act}}\ar{r}{\pr}
        & U \ar[twoheadrightarrow]{d}{p}
        \\
        U \ar[twoheadrightarrow]{r}{p}
        & {[U/G]}.
      \end{tikzcd}\]
      We have
      \[
        \L_{U/[U/G]}
        \simeq \L d^*\L\mrm{act}^* \L_{U/[U/G]}
        \simeq \L d^*\L_{G\times U/U}
        \simeq \L d^*\L\pr_1^*\L_{G/S}
      \]
      where $d = (e,\id) : U \to G \times U$ and $\pr_1 : G\times U \to G$, using \propref{prop:sublimationist}\itemref{item:sublimationist/bc} twice.
      Since $\pr_1 \circ d$ factors as the projection $f : U \to S$ followed by the identity section $e : S \to G$, we get
      \[
        \L_{U/[U/G]}
        \simeq \L f^* \L e^*\L_{G/S}
        \simeq \L f^* \mcal{g}
      \]
      where $\mcal{g} = \L e^*\L_{G/S} \simeq e^*\Omega_{G/S}$ is the dual Lie algebra of $G$ (recall that $G$ is smooth).

      Finally, we have by \propref{prop:sublimationist}\itemref{item:sublimationist/trans} an exact triangle
      \[
        \L p^* \L_{[U/G]}
        \to \L_{U}
        \to \L_{U/[U/G]}
      \]
      where $p : U \twoheadrightarrow [U/G]$ is the quotient morphism and $\L_{U}$ and $\L_{[U/G]}$ denote the cotangent complexes relative to the base $S$.
      Under the equivalence $\Dqc([U/G]) \simeq \Dqc^G(U)$, $\L_{[U/G]}$ may be regarded as the quasi-coherent complex
      \[
        \Fib(\L_U \to \L f^*\mcal{g}) \in \Dqc(U)
      \]
      with a natural $G$-action (induced naturally by the action on $U$).
      For example, if $U$ is a smooth scheme, then this is a $2$-term complex with $\Omega_U$ in degree $0$ and $f^*\mcal{g}$ in (homological) degree $-1$.
      Note that if $G$ is finite (and hence étale), we have $\mcal{g} \simeq 0$; this is compatible with the fact that $\L_{[U/G]}$ should be connective (since $[U/G]$ is Deligne--Mumford).
    \end{exam}

    Specializing to the absolute case, the following tautological reformulation will be useful in practice:

    \begin{prop}\label{prop:inexpungible}
      Let $X$ be a derived stack over a commutative ring $R$.
      Then $X$ admits a cotangent complex $\bL_{X} := \bL_{X/R}$ over $\Spec(R)$ if and only if the following conditions hold:
      \begin{thmlist}
        \item
        For every animated $R$-algebra $A$ and every $x \in X(A)$, denote by $F_x(N)$ the fibre at $x$ of the map
        \[
          X(A\oplus N)
          \to X(A)
        \]
        for every $N \in \D(A)_{\ge 0}$.
        Then the functor $F_x(-)$ is corepresented by an eventually connective derived $A$-module $M_x$.

        \item
        For every morphism of animated $R$-algebras $A \to B$ and every connective derived $B$-module $N$, the commutative square
        \[\begin{tikzcd}
          X(A \oplus N) \ar{r}\ar{d}
          & X(B \oplus N) \ar{d}
          \\
          X(A) \ar{r}
          & X(B)
        \end{tikzcd}\]
        is cartesian.
      \end{thmlist}
      Moreover, under these conditions we have $x^*\L_X \simeq M_x$ for every $A\in\ACAlg_R$ and $x \in X(A)$ (modulo the equivalence $\Dqc(\Spec(A)) \simeq \D(A)$).
    \end{prop}

  \ssec{Smoothness}

    \begin{prop}\label{prop:caboceer}
      Let $f : X \to Y$ be a morphism of derived algebraic stacks.
      Then $f$ is smooth if and only if $f_\cl : X_\cl \to Y_\cl$ is locally of finite presentation and $\L_{X/Y}$ is perfect of Tor-amplitude $\le 0$ (or equivalently, $[-1, 0]$), i.e., if and only if $\L_{X/Y}$ is a perfect complex such that
      \[
        \pi_i (\L_{X/Y} \otimesL_{\sO_X} \sF) = 0
      \]
      for all $\sF\in\QCoh(X)$ and all $i>0$.
    \end{prop}

    \begin{exam}
      If $X$ is a smooth algebraic stack over a commutative ring $R$, then by definition there exists a smooth scheme $U$ and a smooth representable surjection $p : U \twoheadrightarrow X$.
      We have the exact triangle (\propref{prop:sublimationist})
      \[
        p^* \L_{X}
        \to \L_{U}
        \to \L_{U/X}.
      \]
      Since $U$ is smooth over $\Spec(R)$, $\L_{U} \simeq \Omega_{U}^1$ (= cotangent sheaf of $U$ over $\Spec(R)$) is locally free and has Tor-amplitude in $[0,0]$.
      Since $p : U \twoheadrightarrow X$ is smooth and representable, $\L_{U/X}$ is also has Tor-amplitude in $[0,0]$.
      It follows that the fibre $p^*\L_{X}$ has Tor-amplitude in $[-1,0]$.
      Since $p$ is smooth surjective, it follows that $\L_{X}$ has Tor-amplitude in $[-1,0]$.
    \end{exam}

    This suggests the following generalization of smoothness:

    \begin{defn}
      Let $f : X \to Y$ be a morphism of derived algebraic stacks.
      We say that $f$ is \emph{homotopically smooth} if $f_\cl$ is locally of finite presentation and $\L_{X/Y}$ is a perfect complex.
      We say that $f$ is \emph{homotopically $n$-smooth}, $n\ge 0$, if moreover $\L_{X/Y}$ is of Tor-amplitude $\le n$.
    \end{defn}

    \begin{exam}\label{exam:qsm}
      We say that $f : X \to Y$ is \emph{quasi-smooth} if it is homotopically $1$-smooth.
      This admits the following equivalent characterization: there exists a smooth surjection $U \twoheadrightarrow X$ such that $f|_U$ factors via a smooth morphism $Y' \to Y$ and a closed immersion $U \to Y'$ which exhibits $U$ as the derived zero locus of a section $s$ of a vector bundle $E$ over $Y'$.
      \begin{equation*}
        \begin{tikzcd}
          U \ar{r}\ar{d}
          & Y' \ar{d}{s}\ar{r}
          & Y
          \\
          Y' \ar{r}{0}
          & E
          &
        \end{tikzcd}
      \end{equation*}
      See \cite[Prop.~2.3.14]{KhanRydh}.
      In fact, it is possible to generalize this to characterize homotopical $n$-smoothness by taking into account ``shifted'' vector bundles $E[-i]$, $0\le i< n$.
    \end{exam}


    \begin{rem}
      The notion of homotopical smoothness is not as interesting when $X$ and $Y$ are required to be classical (underived).
      For example, if $X$ and $Y$ are classical noetherian schemes and $i : X \hook Y$ is a closed immersion of finite Tor-amplitude (this is automatic say if $Y$ is regular), then the following conditions are equivalent (see \cite[Thm.~1.3]{Avramov}):
      \begin{enumerate}
        \item[(a)]
        $i$ is homotopically smooth, i.e., the relative cotangent complex $\L_{X/Y}$ is a perfect complex;
        \item[(b)]
        $i$ is homotopically $1$-smooth, i.e., the relative cotangent complex $\L_{X/Y}$ is a perfect complex of Tor-amplitude $[0,1]$;
        \item[(b$^\prime$)]
        $i$ is a regular (or \emph{lci}) closed immersion.
      \end{enumerate}
    \end{rem}

    Whereas homotopical smoothness is a rather restrictive condition in classical algebraic geometry, we will see that in the derived setting it often comes for free: natural derived enhancements of singular moduli functors that arise in practice typically tend to be homotopically smooth.
    The following fact may be regarded as a conceptual explanation for this phenomenon: it turns out that homotopical smoothness is equivalent to being locally of finite presentation in a derived sense:

    \begin{thm}[Lurie]\label{thm:imperatorian}
      Let $f : X\to Y$ be a morphism of derived algebraic stacks.
      Then $f$ is homotopically smooth if and only if $f$ is \emph{locally homotopically of finite presentation}, i.e., if for every affine derived scheme $V=\Spec(A)$ and every morphism $V \to Y$, the derived fibre $X \fibprodR_Y V$ preserves filtered colimits when regarded as a functor $\ACAlg_A \to \Grpdoo$.
    \end{thm}

    Here we used the following observation to make sense of the definition of ``locally homotopically of finite presentation'':

    \begin{rem}
      Let $X \to Y$ be a morphism of derived stacks where $Y=\Spec(A)$ is an affine derived scheme.
      Then $X$ may be regarded equivalently as a functor
      \[ X: \ACAlg_A \to \Grpdoo,\]
      where $\ACAlg_A = \ACRing_{A\backslash -}$ is the \inftyCat of $A$-algebras, via the canonical equivalence
      \[
        \Shv(\ACRing^\op; \Grpdoo)_{/\Spec(A)}
        \simeq \Shv(\ACAlg_A^\op; \Grpdoo).
      \]
    \end{rem}
    \begin{rem}
      By way of comparison, let us recall that a morphism of classical stacks $f : X \to Y$ is locally of finite presentation if and only if for every affine scheme $V=\Spec(A)$ and every morphism $V \to Y$, the fibre $X\fibprod_Y V$ preserves filtered colimits when regarded as a functor $\CAlg_A \to \Grpd$.
    \end{rem}

  \ssec{Tangent and cotangent bundles}

    \begin{defn}
      Let $X$ be a derived stack.
      Given a perfect complex $\sE \in \Perf(X)$, we denote by $\V(\sE)$ the stack of cosections of $\sE$, or equivalently sections of $\sE^\vee$.
      That is, given a derived scheme $T$ over $X$, the $T$-points of $\V(\sE)$ over $X$ are morphisms $\sE|_T \to \sO_T$ in $\Perf(T)$: its functor of points $\DSch_X^\op \to \Grpdoo$ is
      \begin{equation*}
        (T \xrightarrow{t} X)
        \mapsto \Maps_{\Dqc(T)}(\L t^*\sE, \sO_T).
      \end{equation*}
      We call $\V(\sE)$ the \emph{derived vector bundle} associated with $\sE$.\footnote{%
        This agrees with Grothendieck's convention for vector bundles, in which $\sE \mapsto \V(\sE)$ is contravariant, rather than the dual convention taking the assignment $\sE \mapsto \Spec(\Sym(\sE^\vee))$.
      }
    \end{defn}

    \begin{thm}
      If $\sE$ is $(-n)$-connective, then $\V(\sE)$ is $n$-Artin over $X$.
    \end{thm}

    \begin{defn}
      Let $X$ be a homotopically smooth derived Artin stack over a commutative ring $R$.
      The \emph{tangent} and \emph{cotangent bundles} of $X$ (over $R$) are the derived vector bundles
      \begin{equation*}
        T_X := \V(\L_X),
        \quad T^*_X := \V(\L_X^\vee),
      \end{equation*}
      respectively.
      Similarly, given a homotopically smooth morphism $f : X \to Y$, the relative tangent and cotangent bundles are the derived vector bundles
      \begin{equation*}
        T_{X/Y} := \V(\L_{X/Y}),
        \quad T^*_{X/Y} := \V(\L_{X/Y}^\vee),
      \end{equation*}
      over $X$.
    \end{defn}

    \begin{exam}
      If $X$ is a smooth scheme over $R$, then $\L_{X} \simeq \Omega_{X}^1[0]$ is the cotangent sheaf in degree zero, so $T_X$ and $T^*_X$ are the usual tangent and cotangent bundles.
    \end{exam}

    \begin{exam}\label{exam:oversand}
      Suppose $X = [U/G]$ where $G$ is a smooth group scheme over $R$ and $U$ is a smooth $R$-scheme with $G$-action.
      We saw in \examref{exam:unexpanding} that the cotangent complex $\L_{[U/G]}$ is the $2$-term $G$-equivariant complex
      \[
        \left[ \Omega_U \to \mcal{g} \boxtimes \sO_U \right]
      \]
      with $\Omega_U$ in degree $0$, $\mcal{g}\boxtimes \sO_U$ in (homological) degree $-1$, and with differential the ``infinitesimal coaction'' map.
      Let $\mfr{g} := \V(\mcal{g})$ and $\mfr{g}^* := \V(\mcal{g}^\vee)$.
      We see that $T_{[U/G]} \fibprod_{[U/G]} U = \V(\L_{[U/G]}|_U)$ is the stacky quotient $[T_U/\mfr{g} \times U]$ of the infinitesimal action map $\mfr{g} \times U \to T_U$, and the tangent bundle $T_{[U/G]}$ is the further quotient by $G$:
      \begin{equation}
        T_{[U/G]} \simeq \big[ [T_U/G] / [\mfr{g} \times U/G] \big]
      \end{equation}
      where $[\mfr{g} \times U/G]$ is the adjoint quotient stack.
      Its dual, the cotangent bundle $T^*_{[U/G]}$, is the derived zero locus of the infinitesimal coaction map, further quotiented by $G$.
      That is, form the homotopy cartesian square
      \begin{equation*}
        \begin{tikzcd}
          \mu^{-1}(0) \ar{r}\ar{d}
          & T^*_U \ar{d}
          \\
          U \ar{r}{0}
          & \mfr{g}^* \times U.
        \end{tikzcd}
      \end{equation*}
      Here $\mu^{-1}(0)$ is equivalently the derived zero locus of the moment map $\mu : T^*_U \to \mfr{g}^* \times U \to \mfr{g}^*$.
      Thus $T^*_{[U/G]}$ is the derived Hamiltonian reduction
      \begin{equation}
        T^*_{[U/G]} \simeq [\mu^{-1}(0)/G].
      \end{equation}
    \end{exam}

    \begin{rem}
      In \examref{exam:oversand}, the derived zero locus $\mu^{-1}(0)$ is rarely classical, and $T^*_{[U/G]}$ is quasi-smooth but rarely smooth.
      Thus even if one only cares about smooth stacks $X$, it becomes necessary to use the language of derived algebraic geometry as soon as we want to look at the cotangent bundle $T^*_X$ (without truncating important information).
    \end{rem}

    \begin{defn}\label{defn:normal}
      Let $f : X \to Y$ be a homotopically smooth morphism of derived Artin stacks.
      The \emph{normal} and \emph{conormal bundles} are
      \begin{equation*}
        N_{X/Y} := T_{X/Y}[1] := \V(\L_{X/Y}[-1]),
        \quad N^*_{X/Y} := T^*_{X/Y}[-1] := \V(\L^\vee_{X/Y}[1]),
      \end{equation*}
      respectively.
    \end{defn}

    \begin{exam}
      If $f : X \to Y$ is a regular closed immersion between schemes, then the cotangent complex $\bL_{X/Y} \simeq \sN_{X/Y}[1]$ is the shifted conormal sheaf, so $N_{X/Y}$ is nothing else than the usual normal bundle.
    \end{exam}


\section{Deformation theory of perfect complexes}
\label{sec:mperf}

  \ssec{The moduli stack of perfect complexes}

    \begin{constr}
      Let $\MPerf$ denote the functor
      \[
        \ACRing \to \Grpdoo,
        \quad
        A \mapsto \Dperf(A)^\simeq
      \]
      sending an animated ring to the \inftyGrpd of perfect derived $A$-modules.
      It follows from \thmref{thm:jargonish} that this satisfies étale descent, hence defines a derived stack.
    \end{constr}

    \begin{rem}
      The derived stack $\MPerf$, or at least its classical truncation, was first studied by A.~Hirschowitz and C.~Simpson \cite{HirschowitzSimpson}, and analyzed in more detail by B.~To\"en and M. Vaqui\'e \cite{ToenVaquie}.
    \end{rem}

    \begin{constr}
      The \emph{universal perfect complex} is the perfect complex
      $$\sE^\univ \in \Dperf(\MPerf)$$
      defined such that for every animated ring $A$ and every morphism $x : \Spec(A) \to \MPerf$ classifying a perfect complex $\sE \in \Dperf(A)$, we have
      \[ x^*(\sE^\univ) \simeq \sE. \]
      It follows formally that for \emph{every} derived stack $X$ and every perfect complex $\sE \in \Dperf(X)$, there is a unique morphism $f : X \to \MPerf$ and an isomorphism $f^*(\sE^\univ) \simeq \sE$.
    \end{constr}

    \begin{rem}\label{rem:autotetraploidy}
      We can similarly consider the larger stacks $\sM_{\Dcoh}$ and $\sM_{\Dpscoh}$ sending $A\in\ACRing_R$ to $\Dcoh(A)^\simeq$ or $\Dpscoh(A)^\simeq$, respectively.
      Here $\Dpscoh(A) \sub \Dqc(A)$ is the full subcategory of \emph{pseudocoherent} derived $A$-modules (sometimes called \emph{almost perfect} $A$-modules) and $\Dcoh(A) \sub \Dpscoh(A)$ is the full subcategory of \emph{coherent} derived $A$-modules.\footnote{%
        When $A$ is noetherian these are defined as follows: $M \in \D(A)$ is pseudocohererent if it is homologically bounded above ($\pi_i(M) = 0$ for $i\gg 0$) and its homotopy groups $\pi_i(M)$ are finitely generated $\pi_0(A)$-modules (see \cite[Def.~7.2.4.10]{HA}); it is \emph{coherent} if it is additionally bounded below ($\pi_i(M) = 0$ for $i\ll 0$).
      }
      There are open immersions of derived stacks
      \[
        \sM_{\Dperf} \hook \sM_{\Dcoh} \hook \sM_{\Dpscoh},
      \]
      but these larger stacks do not admit a cotangent complex (because the perfect complexes are precisely the dualizable objects in $\Dpscoh(X)$).
    \end{rem}

  \ssec{The cotangent complex of the moduli of perfect complexes}

    \begin{thm}\label{thm:repope}
      The perfect complex
      \[
        \L_{\MPerf} := \sE^\univ \otimesL \sE^{\univ,\vee}[-1]
      \]
      is a cotangent complex for the derived stack $\MPerf$.
    \end{thm}

    \begin{lem}\label{lem:anatomist}
      Let $A$ be an animated ring and $M \in \Dperf(A)$ a perfect derived $A$-module.
      For every $N\in\Dperf(A)_{\ge 0}$ denote by $F_M(N)$ the fibre at $M$ of the map of anima
      \[
        \Dperf(A\oplus N)^\simeq
        \to \Dperf(A)^\simeq
      \]
      given by extending scalars along the canonical homomorphism $A\oplus N \to A$.
      Then we have canonical isomorphisms
      \[
        F_M(N)
        \simeq \Maps_{\D(A)}(M\otimesL_A M^\vee[-1], N),
      \]
      natural in $N$.
    \end{lem}
    \begin{proof}
      By definition, $F_M(N)$ is the \inftyGrpd of deformations of $M$ along $A \oplus N \to A$; that is, it is the \inftyGrpd of pairs $(\widetilde{M}, \theta)$ where $\widetilde{M}$ is an $A\oplus N$-module and $\theta$ is an isomorphism $\widetilde{M} \otimesL_{A\oplus N} A \simeq M$.
      Since the square
      \[\begin{tikzcd}
        A\oplus N \ar{r}\ar{d}
        & A \ar{d}
        \\
        A \ar{r}
        & A \oplus N[1]
      \end{tikzcd}\]
      is cartesian, this is equivalent to the \inftyGrpd of deformations of $M \otimesL_A (A\oplus N[1]) \simeq M \oplus (M\otimesL_A N[1])$ along the trivial derivation $A \to A\oplus N[1]$.
      Equivalently, this is the \inftyGrpd of automorphisms of $M\oplus (M\otimesL N[1])$ over $A\oplus N[1]$ which extend to the identity $\id_M : M = M$ along $A \oplus N[1] \to A$.
      That is,
      \[
        F_M(N)
        \simeq \End_{A\oplus N[1]}(M \oplus (M\otimesL N[1])) \fibprod_{\End_{\D(A)}(M)} \{\id_M\}
      \]
      where we can write $\End$ instead of $\Aut$ since every such endomorphism is necessarily invertible.
      Thus we have
      \[
        \begin{multlined}
          F_M(N)  
          \simeq \Maps_{\D(A)}(M, (M\oplus (M\otimesL N[1])) \fibprod_M 0)
          \\
          \simeq \Maps_{\D(A)}(M, M\otimesL N[1])
          \simeq \Maps_{\D(A)}(M\otimesL M^\vee[-1], N)
        \end{multlined}
      \]
      where the last isomorphism follows from the fact that $M$ is perfect, hence dualizable.
    \end{proof}

    \begin{proof}[Proof of \thmref{thm:repope}]
      Let $A\in\ACRing$ be an animated ring and $x : \Spec(A) \to \MPerf$ an $A$-point classifying a perfect derived $A$-module $M \in \Dperf(A)$.
      By \lemref{lem:anatomist}, the animum of derivations $\Der_x(X, M)$ (relative to $\Spec(\Z)$) is corepresented by $M\otimesL M^\vee[-1]$.
      Moreover, if $A \to B$ is a morphism of animated rings, then we have an isomorphism
      \[
        (M \otimesL_A M^\vee[-1]) \otimesL_A B
        \simeq
        N \otimesL_B N^\vee[-1]
      \]
      where $N = M\otimesL_A B$ is the perfect derived $B$-module classified by
      $$\Spec(B) \to \Spec(A) \xrightarrow{x} \MPerf.$$
      It follows that the collection $(M \otimesL M^\vee[-1])$ assembles, as $(A,x)$ varies, into a perfect complex on $\MPerf$ which is a cotangent complex for $\MPerf$.
      By construction, this perfect complex is nothing else than $\sE^\univ \otimesL \sE^{\univ,\vee}[-1] \in \Dperf(\MPerf)$.
    \end{proof}

\section{Moduli stacks of complexes, sheaves, and bundles}

  \ssec{Mapping stacks}

    \begin{constr}
      Let $S$ be a derived stack and let $X$ and $Y$ be derived stacks over $S$.
      The \emph{derived mapping stack} $\uMaps_S(X, Y)$ is the functor
      \[
        \uMaps_S(X, Y) : \ACRing \to \Grpdoo,
        \quad
        R \mapsto \Maps_S(X \fibprodR_S \Spec(R), Y)
      \]
      sending $R$ to the \inftyGrpd of $S$-morphisms $X_R := X \fibprodR_S \Spec(R) \to Y$.
    \end{constr}


    \begin{rem}
      If $X$ is \emph{flat} over $S$, then the classical truncation of $\uMaps_S(X, Y)$ coincides with the classical mapping or Hom stack:
      \[
        \uMaps_S(X, Y)_\cl
        \simeq \uHom_{S_\cl}(X_\cl, Y_\cl).
      \]
    \end{rem}

    \begin{defn}
      The \emph{evaluation morphism}
      \[
        \ev : \Maps_S(X,Y) \fibprodR_S X \to Y
      \]
      is the morphism classified by the identity $\id : \Maps_S(X,Y) \to \Maps_S(X,Y)$.
    \end{defn}

    \begin{prop}\label{prop:pyramides}
      Let $f : X \to Y$ be a morphism of derived algebraic stacks.
      If $f$ is proper, representable, and of finite Tor-amplitude, then $\L f^* : \Dqc(Y) \to \Dqc(X)$ admits a left adjoint $f_\sharp$.
      For perfect complexes $\sF\in\Dperf(X)$, it is given by the formula $f_\sharp(\sF) := \R f_*(\sF^\vee)^\vee$.
    \end{prop}
    \begin{proof}
      If $\sF \in \Dqc(X)$ is perfect, hence dualizable, we may set
      \[
        f_\sharp(\sF) := \R f_*(\sF^\vee)^\vee.
      \]
      It is clear that this defines a left adjoint to $\L f^* : \Dperf(Y) \to \Dperf(X)$:
      \begin{align*}
        \Maps(\R f_*(\sF^\vee)^\vee, \sG)
        &\simeq \Maps(\sO_Y, \sG \otimesL \R f_*(\sF^\vee))\\
        &\simeq \Maps(\sG^\vee, \R f_*(\sF^\vee))\\
        &\simeq \Maps(\L f^*(\sG)^\vee, \sF^\vee)\\
        &\simeq \Maps(\sF, \L f^*(\sG)),
      \end{align*}
      for every perfect complex $\sG \in \Dperf(Y)$, where we recall that for a dualizable object $A$ the functors $A \otimesL (-)$ and $A^\vee \otimesL (-)$ are adjoint to one another.
      
      If $\sF$ is not perfect, to define $f_\sharp(\sF) \in \Dqc(Y)$ is equivalent to define $\L v^* f_\sharp (\sF)$ for every morphism $v : \Spec(R) \to Y$ where $R\in\ACRing$.
      By left transposition from the derived base change formula $\L f^* \R v_* \simeq  \R u_* \L f_R^*$, we have $\L v^* f_\sharp (\sF) \simeq f_{R,\sharp} \L u^*(\sF)$ whenever $f_\sharp$ and $f_{R,\sharp}$ exist.
      Therefore we may replace $Y$ by $\Spec(R)$ and assume $Y$ is affine.

      In this case, $X$ is a derived algebraic space which is in particular separated over $\Spec(R)$ and hence quasi-compact quasi-separated.
      For such $X$, a (generalization of) an important theorem of Thomason asserts that every quasi-coherent complex can be written as a filtered colimit of perfect complexes (see e.g. \cite[Prop.~9.6.1.1]{SAG} or \cite[Thm.~1.40]{KhanKstack}).
      It follows that there is a unique colimit-preserving functor $f_\sharp : \Dqc(X) \to \Dqc(Y)$ which restricts to $\R f_*(-^\vee)^\vee$ on perfect complexes.
      This is clearly left adjoint to $\L f^*$: if $\sF \in \Dqc(X)$ is a colimit of a filtered system $(\sF_\alpha)_\alpha$ of perfect complexes, then we have:
      \begin{align*}
        \Maps(f_\sharp(\sF), \sG)
        &\simeq \Maps(\colim_\alpha f_\sharp(\sF_\alpha), \sG)\\
        &\simeq \lim_\alpha \Maps(f_\sharp(\sF_\alpha), \sG)\\
        &\simeq \lim_\alpha \Maps(\sF_\alpha, \L f^*(\sG))\\
        &\simeq \Maps(\colim_\alpha \sF_\alpha, \L f^*(\sG))\\
        &\simeq \Maps(\sF, \L f^*(\sG)),
      \end{align*}
      for every $\sG \in \Dqc(Y)$.
    \end{proof}

    \begin{rem}\label{rem:dicyclic}
      In the situation of \propref{prop:pyramides}, it is possible to show that the functor $\R f_*$ preserves colimits and admits a right adjoint $f^!$.
      The quasi-coherent complex $\omega_{X/Y} := f^!(\sO_Y) \in \Dqc(X)$ is called the relative \emph{dualizing complex}.
      One then has the following alternative description of the functor $f_\sharp$:
      \[
        f_\sharp(-) \simeq \R f_*(- \otimes \omega_{X/Y}).
      \]
    \end{rem}

    \begin{thm}\label{thm:skainsmate}
      Suppose $S$ is algebraic and $X$ and $Y$ are derived stacks over $S$.
      Set $H := \uMaps_S(X,Y)$ and consider the diagram
      \[
        H \xleftarrow{\pi} X \fibprodR_S H \xrightarrow{\ev} Y
      \]
      where $\pi$ is the projection.
      If $X$ is proper of finite Tor-amplitude and representable over $S$, and $Y$ admits a cotangent complex $\L_{Y/S}$ over $S$, then the perfect complex
      \[
        \L_{H/S} \simeq \pi_\sharp \L \ev^* (\L_{Y/S})
      \]
      is a relative cotangent complex for $H$ over $S$.
    \end{thm}
    \begin{proof}
      For simplicity we take $S=\pt=\Spec(k)$ (and omit it from the notation).
      Given $R\in\ACAlg_k$, an $R$-point $h \in H(R)$ classifying a morphism $f : X_R \to Y$, and a connective derived $R$-module $M \in \D(R)_{\ge 0}$, derivations of $H$ at $h$ are extensions of the morphism $f : X_R \to Y$ along $X_R \hook X_{R\oplus M}$.
      Since the latter can be regarded as the trivial square-zero extension of $X_R$ by $p_R^*(M)$, where $p_R : X_R \to \Spec(R)$ is the projection, these are classified by the cotangent complex of $Y$, i.e.,
      \[
        \Der_h(H, M) \simeq \Maps_{\D(X_R)}(\L f^*\L_Y, \L p_R^*(M)).
      \]
      The assumptions on $X$ imply that $\L p_R^*$ admits a left adjoint $p_{R,\sharp}$, hence $\Der_h(H, -)$ is corepresented by $p_{R,\sharp} \L f^* \L_Y$.
      Now $\L_H := \pi_\sharp \L \ev^* (\L_{Y})$ is the unique perfect complex on $H$ such that $\L h^*(\L_H) \simeq p_{R,\sharp} \L f^* (\L_{Y})$, since we have the commutative diagram
      \[\begin{tikzcd}
        H
        & X \fibprodR_S H \ar[swap]{l}{\pi}\ar{r}{\ev}
        & Y
        \\
        \Spec(R) \ar{u}{h}
        & X_R \ar{u}\ar[equals]{r}\ar[swap]{l}{p_R}
        & X_R \ar{u}{f}
      \end{tikzcd}\]
      where the left-hand square is cartesian.
      It follows that $\L_H$ is a cotangent complex for $H$.
    \end{proof}

  \ssec{Moduli of complexes, sheaves, and bundles}

    \begin{defn}\label{defn:postremote}
      Let $k$ be a commutative ring and $X$ a derived algebraic stack over $k$.
      \begin{defnlist}
        \item
        The \emph{moduli stack of perfect complexes on $X$} is the derived mapping stack
        \[
          \sM_{\Perf(X)} = \uMaps(X, \MPerf)
        \]
        where we omit the subscript $\Spec(k)$ from the notation in the mapping stack.
        By definition, its $R$-points for an animated $k$-algebra $R$ are $k$-morphisms $X_R := X \times \Spec(R) \to \MPerf$, i.e., perfect complexes on $X_R$.

        \item
        Let $G$ be a smooth group scheme over $k$.
        The \emph{moduli stack of $G$-torsors over $X$} (a.k.a. \emph{principal $G$-bundles} over $X$) is the derived mapping stack
        \[
          \sM_{\Bun_G(X)} = \uMaps(X, BG).
        \]
        By definition, its $R$-points for an animated $k$-algebra $R$ are $k$-morphisms $X_R \to BG$, i.e., $G$-torsors on $X_R$.

        \item
        The \emph{moduli stack of vector bundles over $X$} is the substack $\sM_{\Vect(X)} \sub \sM_{\Perf(X)}$ defined as follows: for an animated $k$-algebra $R$, an $R$-point of $\sM_{\Perf(X)}$ belongs to $\sM_{\Vect(X)}$ if and only if the corresponding perfect complex $\sF \in \Dperf(X_R)$ is connective and flat (over $X_R$).

        \item\label{item:postremote/coh}
        The \emph{moduli stack of coherent sheaves on $X$} is the substack $\sM_{\Coh(X)}$ of the derived mapping stack
        \[
          \sM_{\Dcoh(X)} = \uMaps(X, \Dcoh)
        \]
        defined as follows: for an animated $k$-algebra $R$, an $R$-point of $\sM_{\Dcoh(X)}$ belongs to $\sM_{\Coh(X)}$ if and only if the corresponding coherent complex $\sF \in \Dcoh(X_R)$ is connective and flat over $\Spec(R)$.
      \end{defnlist}
    \end{defn}


    \begin{rem}
      For $R$ an ordinary $k$-algebra, the \inftyGrpd of $R$-points of $\sM_{\Vect(X)}$ is equivalent to the $1$-groupoid of locally free sheaves of finite rank on $X_R$.
      Note that there is a canonical isomorphism of derived stacks
      $$\sM_{\Vect(X)} \simeq \coprod_{n\ge 0} \sM_{\Bun_{\GL_n}(X)},$$
      see \examref{exam:frame}.
    \end{rem}

    \begin{lem}\label{lem:unresembling}
      Let $X$ be a derived algebraic stack over an animated commutative ring $R$.
      Let $\sF \in \Dcoh(X)$ be a connective pseudocoherent complex.
      If $R$ is discrete and $\sF$ is flat over $R$, then $\sF$ is discrete, i.e., $\sF \in \Coh(X)$.
    \end{lem}
    \begin{proof}
      Recall that $\sF$ is discrete if for any smooth morphism $u : U \twoheadrightarrow X$ where $U$ is affine, we have $\pi_i(\L u^*\sF) = 0$ for all $i\ne 0$.
      Since $\L u^*$ preserves pseudocoherence we may assume that $X=\Spec(A)$ is affine.
      Recall that a connective complex $M$ over $A$ is flat over $R$ if and only if $\pi_0(M)$ is flat over $\pi_0(R)$, and
      \[
        \pi_i(M)
        \simeq \pi_0(M) \otimesL_{\pi_0(R)} \pi_i(R)
      \]
      for all $i\ge 0$.
      Thus if $M$ is flat over $R$ and $R$ is discrete (i.e. $\pi_i(R) = 0$ for $i\ne 0$), then we have $\pi_i(M)=0$ for $i\ne 0$.
    \end{proof}

    \begin{rem}
      For $R$ an ordinary $k$-algebra, \lemref{lem:unresembling} shows that the \inftyGrpd of $R$-points of $\sM_{\Coh(X)}$ is equivalent to the $1$-groupoid of coherent sheaves on $X_R$ which are flat over $\Spec(R)$.
    \end{rem}

    \begin{lem}\label{lem:onymal}
      If $X$ is smooth over $k$, then we have an inclusion $\sM_{\Coh(X)} \sub \sM_{\Perf(X)}$.
    \end{lem}
    \begin{proof}
      Let $R \in \ACAlg_k$.
      An $R$-point of $\sM_{\Coh(X)}$ corresponds to a coherent complex $\sF \in \Dcoh(X_R)$ which is connective and flat over $R$.
      To show that $\sF$ is perfect it is enough to check that for every point $s : \Spec(\kappa) \to \Spec(R)$, $\sF$ restricts to a perfect complex $\L i_s^*(\sF)$ over the derived fibre $X_s := X \otimesL_k \kappa$:
      \begin{equation*}
        \begin{tikzcd}
          X_s \ar{r}{i_s}\ar{d}
          & X_R \ar{d}{f}
          \\
          \Spec(\kappa) \ar{r}{s}
          & \Spec(R).
        \end{tikzcd}
      \end{equation*}
      Note that $\L i_s^*(\sF)$ is pseudocoherent (because $*$-inverse image always preserves pseudocoherence).
      By the projection and base change formulas we have
      \begin{equation*}
        \R i_{s,*} \L i_s^*(\sF)
        \simeq \sF \otimesL \R i_{s,*} (\sO_{X_s})
        \simeq \sF \otimesL \L f^* \R s_* (\sO_{\Spec(\kappa)}),
      \end{equation*}
      which is discrete because $\R s_* (\sO_{\Spec(\kappa)})$ is discrete and $\sF$ is flat over $R$.
      It follows that $\L i_s^*(\sF)$ is itself discrete, and hence is a coherent sheaf.
      Since $X_s$ is smooth over $\kappa$ and hence regular, we conclude that $\L i_s^*(\sF)$ is perfect as claimed.
    \end{proof}

    \begin{rem}
      In the proof of \lemref{lem:onymal} we implicitly used the notion of regularity of algebraic stacks and the fact that a smooth algebraic stack over a field is regular.
      Let us say by definition that a derived algebraic stack $X$ is \emph{regular} if every coherent sheaf $\sF \in \Coh(X)$ is a perfect complex, i.e., its image in $\Dqc(X)$ belongs to the subcategory $\Dperf(X)$.
      Then we have:
      \begin{defnlist}
        \item\label{item:Trigona}
        For a noetherian ring $R$, $\Spec(R)$ is regular in our sense if and only if its local rings are regular in the sense of Serre \cite{SerreTor}.
        
        \item
        If $X$ admits a smooth surjection $p : \Spec(R) \twoheadrightarrow X$ where $R$ is a regular noetherian ring, then $X$ is regular (recall that $\sF \in \Coh(X)$ is perfect if and only if $p^*(\sF)$ is perfect and use \itemref{item:Trigona}).

        \item
        If $X \to \Spec(R)$ is a smooth morphism of finite presentation where $R$ is regular noetherian, then $X$ is regular.
        Indeed, for any smooth atlas $\Spec(A) \twoheadrightarrow X$, $A$ is a smooth $R$-algebra and hence is regular.
      \end{defnlist}
    \end{rem}

    \begin{prop}
      The inclusion
      \[
        \sM_{\Vect(X)}
        \to \sM_{\Perf(X)}
      \]
      is an open immersion; that is, for any animated $k$-algebra $R$ and any morphism $\Spec(R) \to \sM_{\Perf(X)}$, the base change
      \[
        \sM_{\Vect(X)}\fibprodR_{\sM_{\Perf(X)}}\Spec(R)
        \to \Spec(R)
      \]
      is an open immersion of derived schemes.
      If $X$ is smooth, then the same holds for the inclusion
      \begin{equation*}
        \sM_{\Coh(X)}
        \to \sM_{\Perf(X)}.
      \end{equation*}
    \end{prop}
    \begin{proof}
      Note first that both morphisms are monomorphisms, since for every animated $k$-algebra $R$ the induced functors $\sM_{\Vect(X)}(R) \to \sM_{\Coh(X)}(R) \to \sM_{\Perf(X)}(R)$ are monomorphisms of \inftyGrpds (see \examref{exam:bronzy}).
      Now if $\Spec(R) \to \sM_{\Perf(X)}$ classifies a perfect complex $\sF \in \Dperf(X_R)$, then the base changes in the statement are exactly the open loci where $\sF$ is connective and flat (over $X_R$ or $\Spec(R)$, respectively).
    \end{proof}

    \begin{notat}\label{notat:caramelen}
      By construction, we have an ``evaluation'' morphism $\ev : X \times \sM_{\Perf(X)} \to \sM_{\Perf}$.
      We denote by $\sE_X := \L \ev^*(\sE^\univ)$ the inverse image of the universal perfect complex along this map.
      By abuse of notation, we also denote its restriction to $\sM_{\Coh(X)}$ and $\sM_{\Vect(X)}$ in the same way.
    \end{notat}

    \begin{cor}\label{cor:bytownitite}
      Let $X$ be as in \defnref{defn:postremote}.
      The derived stack $\sM_{\Perf(X)}$ admits a relative cotangent complex
      \[
        \L_{\sM_{\Perf(X)}}
        = \pr_{2,\sharp}(\sE_X \otimesL \sE_X^\vee[1])
        \simeq \pr_{2,*}(\sE_X \otimesL \sE_X^\vee[1] \otimesL \pr_1^*(\omega_X))
      \]
      where $\pr_i$ are the two projections from $X \times \sM_{\Perf(X)}$.
      Moreover, we have the same formula for the cotangent complexes of $\sM_{\Coh(X)}$ and $\sM_{\Vect(X)}$ (modulo the abuse of notation mentioned in \ref{notat:caramelen}).
    \end{cor}
    \begin{proof}
      Combine Theorems~\ref{thm:skainsmate} and \ref{thm:repope}.
      The second isomorphism follows from \remref{rem:dicyclic}.
    \end{proof}

    \begin{rem}\label{rem:bridgepot}
      Once we know that $\sM_{\Coh(X)}$ and $\sM_{\Vect(X)}$ are algebraic and locally finitely presented (on classical truncations), it will follow from \corref{cor:bytownitite} that they are homotopically smooth.
      Moreover, it will be homotopically $(n-1)$-smooth if $X$ is of dimension $\le n$ (using \propref{prop:caboceer}).
      For example, it will be smooth (and classical) when $X$ is a curve and quasi-smooth when $X$ is a surface.
      Note in contrast that when $X$ is a surface, the \emph{classical truncations} of $\sM_{\Coh(X)}$ and $\sM_{\Vect(X)}$ are neither smooth nor homotopically smooth.
    \end{rem}

    \begin{rem}
      One can similarly write down the cotangent complex of $\sM_{\Bun_G(X)}$ using \thmref{thm:skainsmate} and our computation $\L_{BG} \simeq \mfr{g}^\vee[-1]$ (\examref{exam:unexpanding}).
      In particular, it will also be homotopically smooth, and smooth when $X$ is of dimension $\le 1$.
    \end{rem}

  \ssec{The Artin--Lurie representability theorem}

    \begin{thm}[Artin--Lurie]\label{thm:torgoch}
      Let $k$ be a commutative ring, which we assume is excellent or more generally a G-ring.
      Let $X$ be a derived stack over $k$.
      Then $X$ is algebraic if and only if the following conditions hold:
      \begin{thmlist}
        \item
        $X$ admits a cotangent complex $\L_X$ (relative to $k$).

        \item\label{item:holidaymaker}
        The restriction of $X$ to ordinary $k$-algebras takes values in $1$-groupoids.

        \item\emph{Almost of finite presentation.}
        The functor $X : \ACAlg_k \to \Grpdoo$ preserves filtered colimits when restricted to $n$-truncated algebras for all $n\ge 0$.
        (For example, $X$ is locally homotopically of finite presentation, see \thmref{thm:imperatorian}.)

        \item\emph{Integrability.}
        For every complete local noetherian $k$-algebra $R$, the canonical map $X(R) \to \lim_n X(R/\mfr{m}^n)$ is invertible, where $\mfr{m}\sub R$ is the maximal ideal.

        \item\emph{Nil-completeness.}
        For every $R\in\ACAlg_k$, the canonical map $X(R) \to \lim_n X(\tau_{\le n}(R))$ is invertible.

        \item\emph{Infinitesimal cohesion.}
        For every cartesian square in $\ACAlg_k$
        \[\begin{tikzcd}
          A' \ar{r}\ar{d}
          & A \ar[twoheadrightarrow]{d}
          \\
          B' \ar[twoheadrightarrow]{r}
          & B
        \end{tikzcd}\]
        such that $A \to B$ and $B' \to B$ are surjective on $\pi_0$ with nilpotent kernel, $X$ sends the square to a cartesian square.
      \end{thmlist}
    \end{thm}

    \begin{rem}
      \thmref{thm:torgoch} gives criteria for a derived stack to be algebraic.
      This also implies that its classical truncation is algebraic.

      Once we know that $X$ is algebraic, we can further detect whether it is Deligne--Mumford using its diagonal (\thmref{thm:subcarinate}), or alternatively by checking that its cotangent complex is connective.
      Similarly, we can check whether it is an algebraic space by checking that its restriction to ordinary algebras takes values in sets instead of groupoids (\thmref{thm:subcarinate}).
    \end{rem}

    
    It is possible to check the conditions of \thmref{thm:torgoch} for the moduli stacks we have been considering (except for $\sM_{\Perf(X)}$, which doesn't satisfy condition~\itemref{item:holidaymaker}).
    The main point is the computation of the cotangent complex (\corref{cor:bytownitite}); we leave the other verifications to the reader.

    \begin{thm}\label{thm:splenius}
      Let $k$ be a $G$-ring and $X$ an algebraic space which is proper and of finite Tor-amplitude over $k$.
      Then the following derived stacks are algebraic:
      \begin{thmlist}
        \item
        The moduli stack $\sM_{\Vect(X)}$ of vector bundles over $X$.

        \item
        The moduli stack $\sM_{\Bun_G(X)}$ of $G$-bundles over $X$, for every smooth group scheme $G$ over $k$.

        \item
        The moduli stack $\sM_{\Coh(X)}$ of coherent sheaves on $X$, if $X$ is smooth over $k$.
      \end{thmlist}
      Moreover, they are smooth (hence classical) if $X$ is a relative curve, and quasi-smooth if $X$ is a relative surface.
    \end{thm}

    The last claim also follows from our computation of the cotangent complex, see \remref{rem:bridgepot}.


\section{Cohomology of stacks}
\label{sec:cohomology}

\ssec{Abelian sheaves}

  Let $k$ denote a base field and $\Sch_k$ the category of locally of finite type $k$-schemes.
  Given $X \in \Sch_k$ and a commutative ring $\Lambda$ we denote by $\D(X; \Lambda)$ either:
  \begin{enumerate}
    \item If $k=\bC$, the stable \inftyCat $\Shv(X(\bC), \D(\Lambda))$ of sheaves on the topological space $X(\bC)$ with values in the derived \inftyCat of $\Lambda$-modules.
    \item If $\Lambda=\bZ/n\bZ$ where the $n\neq \on{char}(k)$, the stable \inftyCat $\Shv(X_\et, \D(\Lambda))$ of sheaves on the small étale site $X_\et$ with values in the derived \inftyCat of $\Lambda$-modules.
  \end{enumerate}
   
  \begin{thm}\label{thm:vesania}
    The presheaf $\D^* : \Sch_k^\op \to \Catoo$ determined by the assignment
    \begin{equation}\label{eq:vesania}
      X \mapsto \D(X; \Lambda),
      \quad f \mapsto f^*
    \end{equation}
    satisfies descent for the étale topology.
  \end{thm}

  \begin{rem}
    By \examref{exam:pointless}, it follows moreover that $\D^*$ satisfies descent for smooth surjections.
  \end{rem}

  \begin{constr}
    Let $\AlgStk$ denote the \inftyCat of  algebraic stacks locally of finite type over $k$.
    By \thmref{thm:vesania}, there exists a unique étale sheaf $\D^* : \AlgStk_k^\op \to \Catoo$ extending \eqref{eq:vesania}.
    More precisely, it is the right Kan extension, given on $X\in\AlgStk_k$ by the formula
    \begin{equation*}
      \D(X) \simeq \lim_{(T,t)} \D(T)
    \end{equation*}
    where the limit is taken over the category of pairs $(T,t)$ where $T$ is a scheme and $t : T \to X$ is a smooth morphism.
  \end{constr}

  \begin{thm}[Six operations]\label{thm:cowherd}
    We have the following operations on the \inftyCats $\D(X)$ for $X \in \AlgStk_k$:
    \begin{thmlist}
      \item
      An adjoint pair of bifunctors
      \begin{align*}
        \otimes &: \D(X) \times \D(X) \to \D(X),\\
        \uHom &: \D(X)^\op \times \D(X) \to \D(X)
      \end{align*}
      for all $X \in \AlgStk_k$.

      \item
      For every morphism $f : X \to Y$ in $\AlgStk_k$, an adjoint pair
      \[
        f^* : \D(Y) \to \D(X),
        \quad
        f_* : \D(X) \to \D(Y).
      \]

      \item
      For every morphism $f : X \to Y$ in $\AlgStk_k$, an adjoint pair
      \[
        f_! : \D(X) \to \D(Y),
        \quad
        f^! : \D(Y) \to \D(X).
      \]
    \end{thmlist}
    Moreover, they satisfy the following properties:
    \begin{thmlist}
      \item
      \emph{Base change formula:}
      For every cartesian square
      \begin{equation*}
        \begin{tikzcd}
          X' \ar{r}{g}\ar{d}{p}
          & Y' \ar{d}{q}
          \\
          X \ar{r}{f}
          & Y
        \end{tikzcd}
      \end{equation*}
      there is a canonical isomorphism
      \begin{equation*}
        q^*f_! \simeq g_! p^*.
      \end{equation*}

      \item
      \emph{Projection formula:}
      For every morphism $f : X \to Y$ in $\AlgStk_k$, there is a canonical isomorphism
      \begin{equation*}
        f_!(-) \otimes (-)
        \simeq f_!(- \otimes f^*(-)).
      \end{equation*}

      \item
      \emph{Forgetting supports:}
      If $f$ has proper diagonal, there is a canonical morphism
      \begin{equation*}
        f_! \to f_*
      \end{equation*}
      which is invertible when $f$ is proper.

      \item
      \emph{Étale pull-back:}
      If $f$ is étale, there is a canonical isomorphism $f^! \simeq f^*$.

      \item\emph{Localization:}
      If $X \in \Stk_k$ and $i : Z \hook X$ is a closed immersion with complementary open immersion $j : U \hook X$, then there are canonical exact triangles
      \begin{align*}
        &j_!j^* \to \id \to i_!i^*\\
        &i_*i^! \to \id \to j_*j^!.
      \end{align*}
    \end{thmlist}
  \end{thm}

  This theorem was proven by Y.~Liu and W.~Zheng, see \cite{LiuZheng}.

  \begin{rem}
    There is a unique way to extend all the above constructions to \emph{derived} algebraic stacks in such a way that we still have localization triangles: since the inclusion of the classical truncation $i : X_\cl \hook X$ is a surjective closed immersion, we must have $\D(X) \simeq \D(X_\cl)$.
    By base change formulas, all four operations associated with a morphism $ f : X \to Y$ must also be identified with the corresponding operations for $f_\cl : X_\cl \to Y_\cl$.
  \end{rem}

  \begin{rem}
    Moreover, if we extend $\D(-)$ to \emph{higher} Artin stacks (and thus to all derived Artin stacks) with the same definition, then we still have the six operations in this generality.
  \end{rem}

\ssec{Co/homology}

  Given a (derived) algebraic stack $X$ locally of finite type over $k$, let $a_X : X \to \Spec(k)$ denote the projection.
  We define
  \begin{align*}
    \Ccoh(X; \Lambda) &:= R\Gamma(f_*f^*\Lambda) \simeq R\Gamma(X; \Lambda_X),\\
    \CBM(X; \Lambda) &:= R\Gamma(f_*f^!\Lambda) \simeq R\Gamma(X; \omega_X),
  \end{align*}
  where $\Lambda_X = f^*\Lambda$ and $\omega_X = f^!\Lambda$ denote the constant and dualizing sheaves, respectively.
  These are the complexes of cochains and Borel--Moore chains on $X$, respectively.
  We also write
  \begin{align*}
    \H^*(X; \Lambda) &:= \H^*(\Ccoh(X; \Lambda)) \simeq \H^*(X; \Lambda_X),\\
    \H^\BM_*(X; \Lambda) &:= \H^{-*}(\CBM(X; \Lambda)) \simeq \H^{-*}(X; \omega_X).
  \end{align*}

  \thmref{thm:cowherd} yields the following consequences:

  \begin{prop}\label{prop:mush}\leavevmode
    \begin{thmlist}
      \item
      \emph{Proper push-forward:}
      Let $f : X \to Y$ be a proper morphism in $\AlgStk_k$.
      Then there is a canonical morphism
      \begin{equation*}
        f_* : \CBM(X; \Lambda) \to \CBM(Y; \Lambda).
      \end{equation*}

      \item
      \emph{Étale pull-back:}
      Let $f : X \to Y$ be an étale morphism in $\AlgStk_k$.
      Then there is a canonical morphism
      \begin{equation*}
        f^! : \CBM(Y; \Lambda) \to \CBM(X; \Lambda).
      \end{equation*}

      \item
      \emph{Localization triangle:}
      Let $X \in \Stk_k$ and $i : Z \hook X$ a closed immersion with complementary open immersion $j : U \hook X$.
      Then there is a canonical exact triangle
      \begin{equation*}
        \CBM(Z; \Lambda)
        \xrightarrow{i_*} \CBM(X; \Lambda)
        \xrightarrow{j^!} \CBM(U; \Lambda).
      \end{equation*}
    \end{thmlist}
  \end{prop}

  We also have the following consequence of \thmref{thm:vesania}:

  \begin{cor}\label{cor:ectopy}
    On the \inftyCat $\AlgStk_k$, the presheaves
    \begin{align*}
      X &\mapsto \Ccoh(X; \Lambda), \quad f \mapsto f^*\\
      X &\mapsto \CBM(X; \Lambda), \quad f \mapsto f^!
    \end{align*}
    satisfy descent for the étale topology.
  \end{cor}

\ssec{Intersection theory}

  We are finally in position to see how working with complexes of chains (as objects in the derived \inftyCat) rather than their homology groups leads to a streamlined approach to (virtual, stacky) intersection theory.
  Details of the following constructions can be found in \cite{virtual}.

  Recall the normal bundle from \defnref{defn:normal}.
  The following is a generalization of Verdier's deformation to the normal bundle \cite{Verdier}:

  \begin{defthm}
    Let $f : X \to Y$ be a homotopically smooth morphism of derived Artin stacks.
    The \emph{normal deformation} $D_{X/Y}$ is the derived mapping stack
    \begin{equation*}
      D_{X/Y} = \uMaps_{Y\times\A^1}(Y \times \{0\}, X \times \A^1).
    \end{equation*}
    \begin{thmlist}
      \item
      If $X$ and $Y$ are $n$-Artin, then $D_{X/Y}$ is $(n+1)$-Artin.

      \item
      There is a commutative diagram of cartesian squares
      \[\begin{tikzcd}
        X \ar{r}{0}\ar{d}{0}
        & X \times \A^1 \ar[leftarrow]{r}\ar{d}{\widehat{f}}
        & X \times \bG_m \ar{d}{f\times\id}
        \\
        N_{X/Y} \ar{r}\ar{d}
        & D_{X/Y} \ar[leftarrow]{r}\ar{d}
        & Y \times \bG_m\ar{d}
        \\
        \{0\} \ar{r}
        & \A^1 \ar[leftarrow]{r}
        & \bG_m.
      \end{tikzcd}\]
    \end{thmlist}
  \end{defthm}

  See \cite[\S 1.4]{virtual} and \cite{HekkingKhanRydh}.

  \begin{constr}
    Let $f : X \to Y$ be a homotopically smooth morphism of derived algebraic stacks locally of finite type over $k$.
    There is a canonical map
    \begin{equation}
      \sp_{X/Y} : \CBM(Y; \Lambda) \to \CBM(N_{X/Y}; \Lambda)
    \end{equation}
    defined as the composite
    \begin{equation*}
      \begin{multlined}
        \CBM(Y; \Lambda)
        \xrightarrow{\mrm{incl}} \CBM(Y; \Lambda) \oplus \CBM(Y; \Lambda)(1)[1]\\
        \simeq \CBM(Y\times\bG_m; \Lambda)[-1]
        \xrightarrow{\partial} \CBM(N_{X/Y}; \Lambda)
      \end{multlined}
    \end{equation*}
    where the splitting comes from the unit section of $\bG_m$ and $\partial$ is the boundary map in the localization triangle
    \begin{equation*}
      \CBM(N_{X/Y}; \Lambda)
      \to \CBM(D_{X/Y}; \Lambda)
      \to \CBM(Y\times\bG_m; \Lambda)
      \xrightarrow{\partial}
    \end{equation*}
  \end{constr}

  \begin{notat}
    For an integer $d\in\bZ$, we set $\vb{d} := (d)[2d]$, where $(d)$ denotes the Tate twist.
  \end{notat}

  \begin{constr}
    Let $f : X \to Y$ be a morphism in $\AlgStk_k$.
    Suppose $f : X \to Y$ is quasi-smooth, i.e., homotopically $1$-smooth (\examref{exam:qsm}), of relative virtual dimension $d$.
    Then $\bL_{X/Y}$ is in Tor-amplitude $[-1, 1]$ and $N_{X/Y}$ is a ``vector bundle stack''.
    We have the generalized homotopy invariance isomorphism
    \begin{equation*}
      \CBM(X; \Lambda) \simeq \CBM(N_{X/Y}; \Lambda)\vb{d}.
    \end{equation*}
    since the projection $N_{X/Y} \to X$ is of relative dimension $-d$.
    The \emph{quasi-smooth pull-back}, or \emph{virtual pull-back}, is the canonical map
    \begin{equation}
      f^! : \CBM(Y; \Lambda) 
      \xrightarrow{\sp_{X/Y}} \CBM(N_{X/Y}; \Lambda)
      \simeq \CBM(X; \Lambda)\vb{-d}.
    \end{equation}
  \end{constr}

  \begin{rem}
    Note that, even if $X$ and $Y$ are schemes, the above construction passes through the algebraic stacks $N_{X/Y}$ and $D_{X/Y}$ (which are not schemes unless $f : X \to Y$ is a closed immersion).
    Similarly, if $X$ and $Y$ are $1$-Artin, we need to make use of the extension of $\D(-)$ and the six operations to higher Artin stacks.
  \end{rem}

  \begin{defn}
    Let $X$ be a quasi-smooth derived algebraic stack of relative virtual dimension $d$ over $\Spec(k)$.
    The projection $a_X : X \to \Spec(k)$ gives rise to the pull-back
    \begin{equation*}
      a_X^! : \CBM(\Spec(k)) \to \CBM(X)\vb{-d}
    \end{equation*}
    and hence to the canonical element
    \begin{equation*}
      [X] \in \CBM(X)\vb{-d}
      \quad \leadsto \quad
      [X] \in \H^\BM_{2d}(X)(-d)
    \end{equation*}
    called the \emph{virtual fundamental class} of $X$.
  \end{defn}

  \begin{rem}
    The element $[X] \in \CBM(X)\vb{-d}$ corresponds to a canonical morphism
    \begin{equation*}
      \Lambda_X\vb{d} \to a_X^!(\Lambda)
    \end{equation*}
    in $\D(X; \Lambda)$.
    This gives rise to a natural transformation
    \begin{equation}
      a_X^*(-)\vb{d}
      \to a_X^*(-) \otimes a_X^!(\Lambda)
      \xrightarrow{\mrm{can}} a_X^!(-)
    \end{equation}
    or by adjunction a trace map $a_{X,!}a_X^*\vb{d} \to \id$.
    In the relative case, where $f : X \to Y$ is a quasi-smooth morphism of relative virtual dimension $d$, we similarly get a natural transformation
    \begin{equation}
      \tr_f : f_!f^*\vb{d} \to \id.
    \end{equation}
  \end{rem}

  \begin{thm}[Poincaré duality]\leavevmode
    \begin{thmlist}
      \item
      If $f : X\to Y$ is smooth, then the natural transformation $f^*(-)\vb{d} \to f^!(-)$ is invertible.
      Equivalently, $\tr_f$ is the counit of an adjunction $(f_!, f^*\vb{d})$.

      \item
      For any smooth algebraic stack $X$ in $\AlgStk_k$, cap product with $[X]$ determines a canonical isomorphism
      \begin{equation*}
        (-)\cap [X] : \Ccoh(X) \to \CBM(X)\vb{-d}.
      \end{equation*}
    \end{thmlist}
  \end{thm}
  \begin{proof}
    If $f : X \to Y$ is smooth, then the diagonal $\Delta : X \to X \fibprod_Y X$ is still quasi-smooth.
    Thus we have a natural transformation $\tr_\Delta$, which gives rise to a unit for the adjunction $(f_!, f^*\vb{d})$.
    The second statement follows from the first.
    See \cite{KhanDualet}.
  \end{proof}

  \begin{exam}
    Let $\sM_S$ denote the moduli stack $\sM_{\Coh(S)}$ (or $\sM_{\Vect(S)}$, $\sM_{\Bun_G(S)}$) for $S$ an algebraic surface.
    Then since $\sM_S$ is quasi-smooth (\thmref{thm:splenius}), we have constructed a (virtual) fundamental class $[\sM_S] \in \H^\BM_*(\sM_S)$.
    We remark that the traditional method \cite{BehrendFantechi} does not apply here since $\sM_S$ is far from being Deligne--Mumford.
  \end{exam}

  In this framework it is easy to prove the following formula for intersection products.
  If $X$ is a smooth $k$-scheme, the cap product in cohomology gives rise by Poincaré duality to an intersection product
  \begin{equation*}
    \CBM(X)\vb{-p} \otimes \CBM(X)\vb{-q}
    \to \CBM(X)\vb{-p-q+d}.
  \end{equation*}
  If $Y$ is quasi-smooth of virtual dimension $d$ and proper over $X$, the virtual fundamental class gives rise to a class in $\CBM(X)\vb{-d}$ by proper push-forward.

  \begin{thm}[Non-transverse Bézout formula]
    Let $Y$ and $Z$ be smooth or lci closed subvarieties of $X$, of dimension $p$ and $q$ respectively.
    Then there is a canonical homotopy
    \begin{equation*}
      [Y] \cdot [Z] \simeq [Y \fibprodR_X Z]
    \end{equation*}
    in $\CBM(X)\vb{-p-q+d}$.
  \end{thm}

  Note that while the left-hand side consists of usual cycle classes, the right-hand side is genuinely virtual unless the intersection is transverse (that is to say, unless the derived intersection $Y \fibprodR_X Z$ reduces to the classical scheme-theoretic intersection).

\ssec{Quotient stacks}

  \begin{defn}
    Let $G$ be a linear algebraic group over the base field $k$.
    Let $X \in \AlgStk_k$ be an algebraic stack with $G$-action.
    The complex of \emph{equivariant Borel--Moore chains} is defined by
    \begin{equation*}
      \Chom^{\BM,G}(X) := R\Gamma([X/G], f^!(\Lambda_{BG})) \simeq \CBM([X/G]; \Lambda)\vb{g}
    \end{equation*}
    where $f : [X/G] \to BG$ is the projection of the quotient stack to the classifying stack, $g=\dim(G)$, and the isomorphism is Poincaré duality for $BG$.
  \end{defn}

  The following two statements, proven in \cite{Equilisse}, show that this construction can be described by (algebraic approximations to) the Borel construction.

  Choose a filtered system $(V_\alpha)_\alpha$ of $G$-representations where the transition maps $V_\alpha \hook V_\beta$ are monomorphisms.
  Let $W_\alpha \sub V_\alpha$ be $G$-invariant closed subschemes such that:
  \begin{enumerate}
    \item $G$ acts freely on $U_\alpha := V_\alpha \setminus W_\alpha$,
    \item $U_\alpha \sub U_{\alpha+1}$ for all $\alpha$,
    \item We have $\codim_{V_\alpha}(W_\alpha) \to \infty$ as $n \to \infty$.
  \end{enumerate}
  Let $U_\infty$ denote the ind-algebraic space $\{U_\alpha\}_\alpha$.
  For example, for $G=\bG_m$ the obvious choices give $[U_\alpha/G] = \P_k^\infty$.

  \begin{thm}
    There is a canonical isomorphism
    \begin{equation*}
      \Chom^{\BM,G}(X)
      \simeq \CBM(X \fibprod^G U_\infty)\vb{-\dim(U_\infty/G)}
      := \lim_\alpha \CBM(X \fibprod^G U_\alpha)\vb{-d_\alpha}
    \end{equation*}
    where $X \fibprod^G U_\alpha := [(X \fibprod U_\alpha)/G]$ is the quotient by the (free) diagonal action and $d_\alpha=\dim(U_\alpha/G)$.
  \end{thm}

  \begin{thm}
    There is a cartesian square of \inftyCats
    \begin{equation*}
      \begin{tikzcd}
        \D([X/G]) \ar[hookrightarrow]{r}\ar{d}
        & \D(X \fibprod^G U_\infty) \ar{d}
        \\
        \D(X) \ar[hookrightarrow]{r}
        & \D(X \times U_\infty)
      \end{tikzcd}
    \end{equation*}
    where every arrow is $*$-pullback, and the horizontal arrows are fully faithful.
  \end{thm}

  Informally speaking, this means that a sheaf on $[X/G]$ amounts to the data of a sheaf $\sF$ on $X$, a sheaf $\sG$ on $X \fibprod^G U_\infty$, and an isomorphism $\sF|_{X\times U_\infty} \simeq \sG|_{X\times U_\infty}$.

\ssec{Concentration and localization}

  Let $X\in\AlgStk_k$ and let $i : Z \hook X$ be a closed immersion.
  Let $\Sigma$ be a set of line bundles on $X$.

  \begin{quest}[Concentration]
    When is the induced map
    \begin{equation}\label{eq:fenite}
      i_* : \CBM(Z; \Lambda)[c_1(\Sigma)^{-1}]
      \to \CBM(X; \Lambda)[c_1(\Sigma)^{-1}]
    \end{equation}
    invertible?
  \end{quest}

  The following was proven in \cite{virloc}:

  \begin{thm}
    Assume that $X$ has affine stabilizers.
    Suppose that for every point $x \in X\setminus Z$ there exists a line bundle $L \in \Sigma$ whose restriction along $B\Aut(x) \hook X$ is trivial.
    Then concentration holds, i.e., \eqref{eq:fenite} is invertible.
  \end{thm}

  \begin{cor}\label{cor:fugitively}
    Let $T$ be a split algebraic torus acting on a Deligne--Mumford stack $X \in \AlgStk_k$.
    Let $Z$ be the closed substack of fixed points\footnote{%
      Since $X$ is a stack, the appropriate definition of $Z$ here is subtle.
      Briefly, $Z$ is the homotopy fixed point stack with respect to an appropriate reparametrization of the torus action.
      See \cite[Cor.~3.7]{virloc}.
    }.
    Then concentration holds with $\Sigma$ the set of all nontrivial characters of $BT$ (pulled back to $[X/T]$): in particular, we have a canonical isomorphism
    \begin{equation*}
      i_* : \Chom^{\BM,T}(Z;\Lambda)[c_1(\Sigma^{-1})]
      \to \Chom^{\BM,T}(X;\Lambda)[c_1(\Sigma)^{-1}].
    \end{equation*}
  \end{cor}

  The localization triangle (\propref{prop:mush}) gives a very useful way to prove results of 
  this form, since it reduces the problem to $\Sigma$-acyclicity of Borel--Moore chains on the complement $X\setminus Z$.

  From this one can derive:

  \begin{cor}[Virtual localization]
    Let $T$ be a split algebraic torus acting on a Deligne--Mumford stack $X \in \AlgStk_k$.
    Assume $X$ is quasi-smooth and let $Z$ be the fixed locus as in \corref{cor:fugitively}.
    Then we have a canonical homotopy
    \begin{equation*}
      [X] \simeq i_*([Z] \cap e(N_{Z/X})^{-1})
    \end{equation*}
    in $\Chom^{\BM,T}(X)[c_1(\Sigma^{-1})]$.
  \end{cor}

  When $X$ is smooth this is the Atiyah--Bott localization formula.
  In the quasi-smooth case it is the virtual localization formula of Graber--Pandharipande \cite{GraberPandharipande}.
  Unlike \emph{op. cit.} we do not need to assume $X$ admits a global embedding into an ambient smooth stack, or that the cotangent complex $\bL_{Z/X}$ admits a global resolution by vector bundles.
  Again, these improvements are possible because we work at the level of Borel--Moore chains as objects of the derived \inftyCat $\D(\Lambda)$.



\bibliographystyle{alpha}

\begin{thebibliography}{EHKSY}

  \bibitem[AKLPR]{virloc} D.~Aranha, A.\,A.~Khan, A.~Latyntsev, H.~Park, C.~Ravi, \textit{Localization theorems for algebraic stacks}.  \arXiv{2207.01652} (2022).

  \bibitem[Aok]{Aoki} K.~Aoki, \textit{Tensor triangular geometry of filtered objects and sheaves}.  \arXiv{2001.00319} (2020).

  \bibitem[Avr]{Avramov} L.~Avramov, \textit{Locally complete intersection homomorphisms and a conjecture of Quillen on the vanishing of cotangent homology}.  Ann. Math.~{\bf{150}} (1999), no.~2, 455–487.

  \bibitem[BF]{BehrendFantechi} K.~Behrend, B.~Fantechi, \textit{The intrinsic normal cone}, Invent.~Math.~{\bf{128}} (1997), no. 1, 45--88.

  \bibitem[CS]{CesnaviciusScholze} K.~\c{C}esnavicius, P.~Scholze, \textit{Purity for flat cohomology}. \arXiv{1912.10932} (2019).

  \bibitem[Ci]{CisinskiHCHA} D.-C.~Cisinski, \textit{Higher categories and homotopical algebra}.  Cambridge Studies in Advances Mathematics~{\bf{180}} (2019).
  
  \bibitem[Con]{ConradDrinfeldLevel} B.~Conrad, \textit{Arithmetic moduli of generalized elliptic curves}.  J. Inst. Math. Jussieu~{\bf{6}} (2007), no.~2, 209--278.

  \bibitem[GP]{GraberPandharipande} T.~Graber, R.~Pandharipande, \textit{Localization of virtual classes}. Invent.~Math.~{\bf{135}} (1999), no.~2, 487--518.
  
  \bibitem[GR]{GaitsgoryRozenblyum} D.~Gaitsgory, N.~Rozenblyum, \textit{A study in derived algebraic geometry. Vols. I--II}. Math.~Surv.~Mono.~{\bf{221}} (2017).

  \bibitem[GZ]{GabrielZisman} P.~Gabriel, M.~Zisman, \textit{Calculus of fractions and homotopy theory}.
  Ergebnisse der Mathematik und ihrer Grenzgebiete, Band~{\bf{35}} (1967). Springer--Verlag.

  \bibitem[Gro]{Gross} Ph.~Gross, \textit{Tensor generators on schemes and stacks}.  Algebr.~Geom.~{\bf{4}}, no. 4, 501--522 (2017).

  \bibitem[HKR]{HekkingKhanRydh} J.~Hekking, A.\,A.~Khan, D.~Rydh, \textit{Deformation to the normal cone and blow-ups via derived Weil restrictions}.  In preparation.
  
  \bibitem[HS]{HirschowitzSimpson} A.~Hirschowitz, C.~Simpson, \textit{Descente pour les n-champs}.  \mathAG{9807049} (2001).

  \bibitem[KRy]{KhanRydh} A.\,A.~Khan, D.~Rydh, \textit{Virtual Cartier divisors and blow-ups}. \arXiv{1802.05702} (2018).

  \bibitem[KRa]{Equilisse} A.\,A.~Khan, C.~Ravi, \textit{Equivariant generalized cohomology via stacks}.  \arXiv{2209.07801} (2022).

  \bibitem[Kha1]{virtual} A.\,A.~Khan, \textit{Virtual fundamental classes for derived stacks I}. \arXiv{1909.01332} (2019).

  \bibitem[Kha2]{KhanKstack} A.\,A.~Khan, \textit{K-theory and G-theory of derived algebraic stacks}.  Jpn. J. Math.~{\bf{17}} (2022), 1--61.

  \bibitem[Kha3]{KhanDualet} A.\,A.~Khan, \textit{Absolute Poincaré duality in étale cohomology}.  Forum Math. Sigma~{\bf{10}} (2022), no.~10, e99.  \arXiv{2111.02875}.

  \bibitem[Kon]{Kontsevich} M.~Kontsevich, \textit{Enumeration of rational curves via torus actions}, in: The moduli space of curves (Texel Island, 1994), 335--368, Progr. Math.~{\bf{129}} (1995).

  \bibitem[LMB]{LaumonMoretBailly} G.~Laumon, L.~Moret-Bailly, \textit{Champs algébriques}.  Ergebnisse der Mathematik und ihrer Grenzgebiete~{\bf{39}} (2000).

  \bibitem[LZ]{LiuZheng} Y.~Liu, W.~Zheng, \textit{Enhanced six operations and base change theorem for higher Artin stacks}. \arXiv{1211.5948} (2012).

  \bibitem[Lur1]{HTT} J.~Lurie, \textit{Higher Topos Theory}.  Ann. Math. Stud.~{\bf{170}} (2009).
  
  \bibitem[Lur2]{HA} J.~Lurie, \textit{Higher Algebra}, version of 2017-09-18.  Available at: \url{https://www.math.ias.edu/~lurie/papers/HA.pdf}.

  \bibitem[Lur3]{SAG} J.~Lurie, \textit{Spectral algebraic geometry}, version of 2018-02-03.  Available at: \url{https://www.math.ias.edu/~lurie/papers/HA.pdf}.

  \bibitem[Lur4]{Kerodon} J.~Lurie, \textit{Kerodon}, \url{https://kerodon.net}.

  \bibitem[Ryd1]{RydhApprox} D.~Rydh, \textit{Noetherian approximation of algebraic spaces and stacks}.  J.~Algebra~{\bf{422}} (2015), 105--147.

  \bibitem[Ryd2]{RydhCompleteness} D.~Rydh, \textit{Approximation of sheaves on algebraic stacks}.  Int.~Math.~Res.~Not.~{\bf{2016}} (2016), no. 3, 717--737.

  \bibitem[SP]{Stacks} J.~de Jong, \textit{The Stacks Project}, \url{https://stacks.math.columbia.edu}.

  \bibitem[Se]{SerreTor} J.-P. Serre, \textit{Algèbre locale. Multiplicités}.  Lecture Notes in Mathematics~{\bf{11}} (1989).

  \bibitem[TV]{ToenVaquie} B.~To\"en, M.~Vaqui\'e, \textit{Moduli of objects in dg-categories}.  Ann.~Sci.~Éc.~Norm.~Supér.~{\bf{40}} (2007), no.~3, 387--444.

  \bibitem[TV2]{HAG2} B.~Toën, G.~Vezzosi, \textit{Homotopical algebraic geometry. II: Geometric stacks and applications}.  Mem. Am. Math. Soc.~{\bf{902}} (2008).

  \bibitem[Tho]{ThomasonResolution} R.\,W.~Thomason, \textit{Equivariant resolution, linearization, and Hilbert’s fourteenth problem over arbitrary base schemes}.  Adv.~Math.~{\bf{65}} (1987), 16--34.


  \bibitem[Ver]{Verdier} J.-L. Verdier, \textit{Le théorème de Riemann-Roch pour les intersections complètes}, in: Séminaire de géométrie analytique, Paris, France, 1974--75.  Astérisque~{\bf{36}}-{\bf{37}} (1976), 189--228.

\end{thebibliography}

Institute of Mathematics,
Academia Sinica,
Taipei 10617,
Taiwan

\end{document}